\documentclass{amsart}

 \usepackage[foot]{amsaddr}
\let\oldtocsection=\tocsection
 
\let\oldtocsubsection=\tocsubsection 
 
\let\oldtocsubsubsection=\tocsubsubsection
 
\renewcommand{\tocsection}[2]{\vspace{0.5em}\hspace{0em}\oldtocsection{#1}{#2}}
\renewcommand{\tocsubsection}[2]{\vspace{0.5em}\hspace{1em}\oldtocsubsection{#1}{#2}}
\renewcommand{\tocsubsubsection}[2]{\vspace{0.5em}\hspace{2em}\oldtocsubsubsection{#1}{#2}}
\usepackage{graphicx,cancel,xcolor,hyperref,comment,graphicx,geometry}

\usepackage{tikz}
\usepackage{graphicx,url,etoolbox}
\usepackage{lipsum}
\usepackage{dsfont}

\usepackage{fancyhdr}
\usepackage{lastpage}

\def\nline{\\ \noalign{\medskip}}

\newtheorem{theoreme}{Theorem}[section]
\newtheorem{pro}[theoreme]{Proposition}
\newtheorem{lemma}[theoreme]{Lemma}
\newtheorem{rem}[theoreme]{Remark}

\newtheorem{definition}[theoreme]{Definition}
\theoremstyle{definition}

\numberwithin{equation}{section}

 \renewenvironment{proof}{{\bfseries \noindent Proof.}}{\demo}
\newcommand\xqed[1]{%
  \leavevmode\unskip\penalty9999 \hbox{}\nobreak\hfill
  \quad\hbox{#1}}
\newcommand\demo{\xqed{$\square$}}

\hypersetup{bookmarks, colorlinks, urlcolor=blue, citecolor=blue, linkcolor=blue, hyperfigures, pagebackref,
    pdfcreator=LaTeX, breaklinks=true, pdfpagelayout=SinglePage, bookmarksopen=true,bookmarksopenlevel=2}

\usepackage{comment}
\def\u2{\u^2}
\def\u3{\u^3}
\def\u4{\u^4}
\def\u5{\u^5}
\def\y1{\y^1}
\def\y2{\y^2}
\def\y3{\y^3}
\def\y4{\y^4}
\def\y5{\y^5}

\def\lam{\lambda}

\def\R{\mathbb R}

\def\C{\mathbb C}
\def\HH{\mathcal H}

\def\AA{\mathcal A}

\def\la {{\lambda}}

\newcommand {\nc}   {\newcommand}
\nc {\be}   {\begin{equation}} \nc {\ee}   {\end{equation}} \nc
{\beq}  {\begin{eqnarray}} \nc {\eeq}  {\end{eqnarray}} \nc {\beqs}
{\begin{eqnarray*}} \nc {\eeqs} {\end{eqnarray*}}
\def\edc{\end{document}}

\providecommand{\abs}[1]{\lvert#1\rvert}%absolute value
\providecommand{\norm}[1]{\lVert#1\rVert}%norm   
\DeclareMathOperator{\sign}{sign}   

\usepackage{tikz}
\usetikzlibrary{decorations.pathmorphing,patterns,scopes,intersections,calc}
\usepackage{caption}
         
\begin{document}
\title[\fontsize{7}{9}\selectfont  ]{Stability results of an elastic/viscoelastic transmission problem of locally coupled waves with non smooth coefficients}
\author{Mohammad Akil$^{1}$}
\author{Ibtissam Issa$^{1,2}$}
\author{Ali Wehbe$^{1}$}
\address{$^1$Lebanese University, Faculty of sciences 1, Khawarizmi Laboratory of  Mathematics and Applications-KALMA, Hadath-Beirut, Lebanon.}
\address{$^2$ Aix-Marseille University, I2M, Marseille-France.}
\email{mohamadakil1@hotmail.com, ibtissam.issa@etu.univ-amu.fr, ali.wehbe@ul.edu.lb}
%\date{}
%%%%%%%%%%%%%%%%%%%%%%%%%%%%%%%%%%%%%%%%%%%%%%%%%%%%%%%%%%%%
%Keyword
%%%%%%%%%%%%%%%%%%%%%%%%%%%%%%%%%%%%%%%%%%%%%%%%%%%%%%%%%%%%
%\keywords{}
%%%%%%%%%%%%%%%%%%%%%%%%%%%%%%%%%%%%%%%%%%%%%%%%%%%%%%%%%%%%
%Abstract
%%%%%%%%%%%%%%%%%%%%%%%%%%%%%%%%%%%%%%%%%%%%%%%%%%%%%%%%%%%%

%\pagenumbering{roman}
%\maketitle
%\tableofcontents
%\clearpage
%\pagenumbering{arabic}
%\setcounter{page}{1}
%\setcounter{equation}[section]
\setcounter{equation}{0}
%\abstractname{.}
\begin{abstract}
We investigate the stabilization of a locally coupled wave equations with only one internal viscoelastic damping of Kelvin-Voigt type (see System \eqref{eq1}-\eqref{eq3}). The main novelty in this paper is that both the damping and the coupling coefficients are non smooth (see \eqref{bc}). First, using a general criteria of Arendt-Batty, combined with an uniqueness result, we prove that our system is strongly stable. Next, using a spectrum approach, we prove the non-exponential (uniform) stability of the system. Finally, using a frequency domain approach, combined with a piecewise multiplier technique and the construction of a new multiplier satisfying some ordinary differential equations, we show that the energy of smooth solutions of the system decays polynomially of type $t^{-1}$.

%Indeed,  The method is based on the frequency domain approach combining with multiplier method .
%The purpose of this paper is to investigate the stabilization of a system of two wave equations coupled through velocities with only one localized internal Kelvin-Voigt damping. The main novelty in this paper is that the considered system is coupled and that the damping and coupling coefficients are discontinuous. %Firstly, using a unique continuation result, we prove that the system is strongly stable. Secondly, we show that the system lacks exponential stability, and we establish a polynomial energy decay rate of type $t^{-1}$ for smooth initial data.
%First, using a general criteria of Arendt and Batty with a unique continuation result we show the strong stability of our system in the absence of the compactness of the resolvent. Next, we show that our system lacks exponential stability. Hence, we look for a polynomial decay rate for smooth initial data for our system by applying a frequency domain approach combining with a multiplier method. Indeed, we establish a polynomial energy decay rate of type $t^{-1}$ for smooth initial data.
\end{abstract}
\maketitle
\pagenumbering{roman}
\maketitle
\tableofcontents
\clearpage
\pagenumbering{arabic}
\setcounter{page}{1}
\newpage
\section{Introduction} 
\subsection{Motivation and aims}
There are several mathematical models representing physical damping. The most often encountered type of damping in vibration studies are linear viscous damping and Kelvin-Voigt damping which are special cases of proportional damping. Viscous damping usually models external friction forces such as air resistance acting on the vibrating structures and is thus called "external damping", while Kelvin-Voigt damping originates from the internal friction of the material of the vibrating structures and thus called "internal damping". In 1988, F. Huang in \cite{Huang-falun} considered a wave equation with globally distributed Kelvin-Voigt damping, i.e. the damping coefficient is strictly positive on the entire spatial domain. He proved that the corresponding semigroup is not only exponentially stable, but also is analytic (see Definition \ref{Defsta}, Theorem \ref{hp} and Theorem \ref{analytic} below). Thus, Kelvin-Voigt damping is stronger than the viscous damping when globally distributed. Indeed, it was proved that the semigroup corresponding to the system of wave equations with global viscous damping is exponentially stable but not analytic (see \cite{chen1} for the one dimensional system and \cite{bardos} for the higher dimensional system). However, the exponential stability of a wave equation is still true even if the viscous damping is localized, via a smooth or a non smooth damping coefficient, in a suitable subdomain satisfying some geometric conditions (see \cite{bardos}).  Nevertheless, when viscoelastic damping is distributed locally, the situation is more delicate and such comparison between viscous and viscoelastic damping is not valid anymore. Indeed, the stabilization of the wave equation with local Kelvin-Voigt damping is greatly influenced by the smoothness of the damping coefficient and the region where the damping is localized (near or faraway from the boundary) even in the one-dimensional case. So, the stabilization of systems (simple or coupled) with local Kelvin-Voigt damping has attracted the attention of many authors  (see the Literature below for the history of this kind of damping). From a mathematical point of view, it is important to study the stability of a system coupling a locally damped  wave equation with a conservative one. Moreover, the study of this kind of systems is also motivated by several physical considerations and occurs in many applications in engineering and mechanics. In this direction, recently in 2019, Hassine and Souayeh in \cite{hassine2019}, studied the stabilization of a system of global coupled wave equations with one localized Kelvin-Voigt damping. The system is described by
\begin{equation}\label{sys2019}
\left\{\begin{array}{l}
u_{tt}-\left(u_x+b(x)u_{tx}\right)_x+v_t=0,\hspace{0.7cm} (x,t)\in (-1,1)\times \R^+,\\ \\
v_{tt}-cv_{xx}-u_t=0,\hspace{2.5cm}(x,t)\in (-1,1)\times \R^+,\\ \\
u(0,t)=v(0,t)=0,u(1,t)=v(1,t)=0,\hspace{0.5cm} t>0,\\ \\
u(x,0)=u_0(x),u_{t}(x,0)=u_1(x),\hspace{1cm} x\in (-1,1),\\ \\
v(x,0)=v_0(x),v_{t}(x,0)=v_1(x),\hspace{1.3cm} x\in (-1,1),
\end{array}\right.
\end{equation}
where $c>0$, and $b\in L^{\infty}(-1,1)$ is a non-negative function. They assumed that the damping coefficient is given by $b(x) = d\mathds{1}_{[0,1]}(x)$, where $d$ is a strictly positive constant. The Kelvin-Voigt damping 
$\left(b(x)u_{tx}\right)_x$ is applied at the first equation and the second equation is indirectly damped through the coupling between the two equations.
Under the two conditions that the Kelvin-Voigt damping is localized near the boundary and the two waves are globally coupled, they obtained a polynomial energy decay rate of type $t^{-{\frac{1}{6}}}$. Then the stabilization of System \eqref{sys2019} in the case where the Kelvin-Voigt damping is localized in an arbitrary subinterval of $(-1,+1)$ and the two waves are locally coupled has been left as an open problem. In addition, we believe that the energy decay rate obtained in  \cite{hassine2019} can be improved. So, we are interested in studying this open problem.   

The main aim of this paper is to study the stabilization of a system of localized coupled wave equations with only one Kelvin-Voigt damping localized via non-smooth coefficient in a subinterval of the domain. The system is described by
\begin{eqnarray}
u_{tt}-\left(au_x+b(x)u_{tx}\right)_x+c(x)\ y_t&=&0,\quad (x,t)\in (0,L)\times \R^+,\label{eq1}\\ 
y_{tt}-y_{xx}-c(x)\ u_t&=&0,\quad (x,t)\in (0,L)\times \R^+,\label{eq2}
\end{eqnarray}
with fully Dirichlet boundary conditions, 
\begin{equation}\label{eq3}
u(0,t)=u(L,t)=y(0,t)=y(L,t)=0,\quad \forall\ t\in \R^+,
\end{equation}
where
\begin{equation}\label{bc}
b(x)=\left\{\begin{array}{ccc}
b_0&\text{if}&x\in (\alpha_1,\alpha_3)\\
0&&\text{otherwise}
\end{array}
\right.
\quad \text{and}\quad c(x)=\left\{\begin{array}{ccc}
c_0&\text{if}&x\in (\alpha_2,\alpha_4)\\
0&&\text{otherwise}
\end{array}
\right.
\end{equation}
and $a>0, b_0>0$ and $c_0>0$, and where we consider $0<\alpha_1<\alpha_2<\alpha_3<\alpha_4<L$. This system is considered with the following initial data
\begin{equation}\label{eq4}
u(\cdot,0)=u_0(\cdot),\ u_t(\cdot,0)=u_1(\cdot),\ y(\cdot,0)=y_0(\cdot)\quad \text{and}\quad y_t(\cdot,0)=y_1(\cdot).
\end{equation}
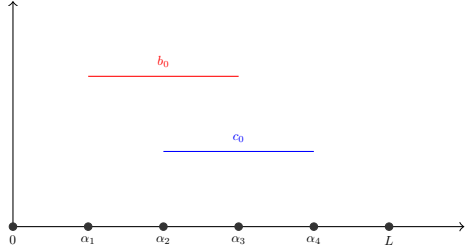
\begin{center}
\begin{tikzpicture}

\draw[->] (0,0)--(6,0);
\draw[->](0,0)--(0,3);

\draw[-,red](1,2)--(3,2);
\draw[-,blue](2,1)--(4,1);

          \node at (1,0) [circle, scale=0.3, draw=black!80,fill=black!80] {};
\node[black,below] at (1,0){\scalebox{0.5}{$\alpha_1$}};   
      
            \node at (2,0) [circle, scale=0.3, draw=black!80,fill=black!80] {};
\node[black,below] at (2,0){\scalebox{0.5}{$\alpha_2$}};        
       
                 \node at (3,0) [circle, scale=0.3, draw=black!80,fill=black!80] {};
\node[black,below] at (3,0){\scalebox{0.5}{$\alpha_3$}};   

          \node at (4,0) [circle, scale=0.3, draw=black!80,fill=black!80] {};
\node[black,below] at (4,0){\scalebox{0.5}{$\alpha_4$}};  

          \node at (5,0) [circle, scale=0.3, draw=black!80,fill=black!80] {};
\node[black,below] at (5,0){\scalebox{0.5}{$L$}};

          \node at (0,0) [circle, scale=0.3, draw=black!80,fill=black!80] {};
\node[black,below] at (0,0){\scalebox{0.5}{$0$}}; 

\node[red,above] at (2,2){\scalebox{0.5}{$b_0$}};   

\node[blue,above] at (3,1){\scalebox{0.5}{$c_0$}};

\end{tikzpicture}

\end{center}

 \subsection{Literature}% There are several mathematical models representing physical damping. The most often encountered type of damping in vibration studies are linear viscous damping and Kelvin-Voigt damping which are special cases of proportional damping. Viscous damping usually models external friction forces such as air resistance acting on the vibrating structures and is thus called "external damping", while Kelvin-Voigt damping originates from the internal friction of the material of the vibrating structures and thus called "internal damping". When the Kelvin-Voigt damping is global (i.e., distributed over the entire spatial domain), the semigroup is not only exponentially stable but also analytic (see \cite{Huang-falun}). Thus, Kelvin-Voigt damping is stronger than the viscous damping when globally distributed since the latter only produces exponential stability and not any level of smoothing property (see \cite{chen1} for the one dimensional system and \cite{bardos} for the higher dimensional system).
The wave is created when a vibrating source disturbs the medium. In order to restrain those vibrations, several dampings can be added such as Kelvin-Voigt damping which is originated from the extension or compression of the vibrating particles. This damping is a viscoelastic structure having properties of both elasticity and viscosity. In the recent years,
many researchers showed interest in problems involving this kind of damping (local or global) where different types of stability have been showed. 
%When the damping is localized (i.e., distributed only on the proper subset of the spatial domain), such a comparison is not valid anymore. So, the localization of the Kelvin-Voigt coefficient plays a critical role in the stabilization of the system. The stabilization of systems (simple or coupled) with Kelvin-Voigt damping has attracted the attention of many authors.
In particular, in the one dimensional  case, it was proved that the smoothness of the damping coefficient affects critically the studying of the stability  and regularity of the solution of the system. Indeed, in the one dimensional case we can consider the following system
\begin{equation}\label{introeq}
\left\{\begin{array}{l}
u_{tt}-\left(u_x+b_1(x)u_{tx}\right)_x=0,\hspace{1cm} -1\leq x\leq 1, t>0,\\ \\
u(1,t)=u(-1,t)=0,\hspace{2.2cm} t>0,\\ \\
u(x,0)=u_0(x),u_{t}(x,0)=u_1(x),\quad -1\leq x\leq 1,
\end{array}\right.
\end{equation}
with $b_1\in L^{\infty}(-1,1)$ and 
\begin{equation}
b_1(x)=\left\{\begin{array}{ccc}
0&\text{if}&x\in (0,1),\\
a_1(x)&\text{if}& x\in (-1,0),
\end{array}
\right.
\end{equation}
where the function $a_1(x)$ is non-negative.
The case of local Kelvin-Voigt damping was first studied in 1998 \cite{chen-1998,liu-liu-1998}, it was proved that the semigroup loses exponential stability and smooth property when the damping is local and $a_1=1$ or $b_1(\cdot)$ is the characteristic function of any subinterval of the domain. This surprising result initiated the study of an  elastic system with local Kelvin-Voigt damping. In 2002, K. Liu and Z. Liu  proved that system \eqref{introeq} is exponentially stable if $b_1^{\prime}(.)\in C^{0,1}([-1,1])$ (see \cite{liu-liu-2002}). Later, in \cite{zhang-2010}, the smoothness on $b_1$ was weakened to $b_1(\cdot)\in C^{1}([-1,1])$ and a condition on $a_1$ was taken. In 2004, Renardy's results \cite{Renardy} hinted that the solution of the system \eqref{introeq} may be exponentially stable under smoother conditions on the damping coefficient. This result was confirmed by  K. Liu, Z. Liu and Q. Zhang in \cite{liu-liu-zhang2017}. On the other hand, Liu and Rao in 2005 (see \cite{RaoLiu01}) proved that the semigroup corresponding to system \eqref{introeq} is polynomially stable of order almost 2 if $a_1(.)\in C(0,1)$ and $a_1(x)\geq a_1 \geq 0$ on (0,1). The optimality of this order was later proved in \cite{alves-rivera-2014}. In 2014, Alves and al., in \cite{Alves2}, considered the transmission problem of a material composed of three components; one of them is a Kelvin–Voigt viscoelastic material, the second is an elastic material (no dissipation) and the third is an elastic material inserted with a frictional damping mechanism. They proved that the rate of decay depends on the position of each component. When the viscoelastic component is not in the middle of the material, they proved exponential stability of the solution. However, when the viscoelastic part is in the middle of the material, the solution decays polynomially as $t^{-2}$. In 2016, under the assumption that the damping coefficient has a singularity at the interface of the damped and undamped regions and behaves like $x^{\alpha}$ near the interface, it was proven by Liu and Zhang \cite{liu-zhang-2016} that the semigroup corresponding to the system is polynomially or exponentially stable and the decay rate depends on the parameter $\alpha\in (0,1]$. In \cite{Ammari03}, Ammari and al. generalized the  cases  of single elastic string with local Kelvin-Voigt damping (in \cite{liu-liu-2002,Ammari04}). They studied the stability of a tree of elastic strings with local Kelvin-Voigt damping on some of the edges. They proved exponential/polynomial  stability of the system under  the  compatibility  condition  of displacement and strain and the continuity condition of damping coefficients at the vertices of the  tree.

In \cite{HASSINE01}, Hassine considered the longitudinal and transversal vibrations of the transmission Euler-Bernoulli beam with Kelvin-Voigt damping distributed locally on any subinterval of
the region occupied by the beam. He proved that the semigroup associated with the equation for the transversal motion of the beam is exponentially stable, although the semigroup associated with the equation for the longitudinal motion of the beam is polynomially stable of type $t^{-2}$. In \cite{HASSINE02}, Hassine considered a beam and a wave equation coupled on an elastic beam through transmission conditions with locally distributed Kelvin-Voigt damping that acts through one of the two equations only. He proved a polynomial energy decay  rate of type $t^{-2}$ for  both cases where the dissipation acts through the beam equation and through the wave equation. In 2016, Oquendo and Sanez studied the wave equation with internal coupled terms where the Kelvin-Voigt damping is global in one equation and the second equation is conservative. They showed that the semigroup loses speed and decays with the rate $t^{-\frac{1}{4}}$ and they proved that this decay rate is optimal (see \cite{Oquendo-2016}).

Let us mention some of the results that have been established for the case of wave equation with Kelvin-Voigt damping in the multi-dimensional setting. In \cite{Huang-falun}, the author proved that when the Kelvin-Voigt damping div$(d(x)\nabla u_{t})$ is globally distributed, i.e. $d(x)\geq d_0>0$ for almost all $x\in\Omega$, the wave equation generates an analytic semi-group. In \cite{Liu-Rao-2006}, the authors considered the wave equation with local visco-elastic damping distributed around the boundary of $\Omega$. They proved that the energy of the system decays exponentially to zero as t goes to infinity for all usual initial data under the assumption that the damping coefficient satisfies: $d\in C^{1,1}(\Omega)$, $\Delta d\in L^{\infty}(\Omega)$ and $|\nabla d(x)|^2\leq M_0 d(x)$ for almost every $x$ in $\Omega$ where $M_0$ is a positive constant. On the other hand, in \cite{Tebou}, the author studied the stabilization of the wave equation with Kelvin-Voigt damping. He established a polynomial energy decay rate of type $t^{-1}$ provided that the damping region is
localized in a neighborhood of a part of the boundary and verifies certain geometric condition. Also in \cite{Nicaise-Pignotti2016}, under the same assumptions on $d$, the authors established the exponential stability of the wave equation with local Kelvin-Voigt damping localized around a part of the boundary and an extra boundary with time delay where they added an appropriate geometric condition. Later on, in \cite{Ammari_2019}, the wave equation with Kelvin-Voigt damping localized in a subdomain $\omega$ far away from the boundary without any geometric conditions was considered. The authors established a logarithmic energy decay rate for smooth initial data. Further more, in \cite{Rayan2019}, the authors investigate the stabilization of the wave equation with Kelvin-Voigt damping localized via non smooth coefficient in a suitable sub-domain of the whole bounded domain. They proved a polynomial stability result in any space dimension, provided that the damping region satisfies some geometric conditions.

\subsection{Description of the paper} This paper is organized as follows: In Subsection \ref{Section-2.1}, we reformulate the system \eqref{eq1}-\eqref{eq4} into an evolution system and we prove the well-posedness of our system by semigroup approach. In  Subsection \ref{Section-2.2}, using a general criteria of Arendt and Batty, we show the strong stability of our system in the absence of the compactness of the resolvent. In Section \ref{section-3}, we prove that the system lacks exponential stability using two different approaches. The first case is by taking the damping and the coupling terms to be globally defined, i.e $b(x)=b_0>0$ and $c(x)=c_0>0$ and we prove the lack of exponential stability using Borichev-Tomilov results. The second case is by taking only the damping term to be localized and we use the method which was developed by Littman and Markus. In Section \ref{Section-4}, we look for a polynomial decay rate by applying a frequency domain approach combined with a multiplier method  based on the exponential stability of an auxiliary problem, where we establish a polynomial energy decay for smooth solution of type $t^{-1}$ .

%%%%%%%%%%%%%%%%%%%%%%%%%%%%%%%%%%%%%%%%%%%%%%%%%%%%%%%%%%%%%%%%%%%%%%%%%%%%%%%%%%%%%%%%%%%%%%%%%%%%%%%%%%%%%%%%%%%%%%%%%%%%%%%%%%%%%%%%%%%%%%%%%%%%%%%%%%%%%%%%%%%%%%%%%%%%%%%%%%
\section{Well-Posedness and Strong Stability}\label{Section-2}
\noindent In this section, we  study the strong stability of System \eqref{eq1}-\eqref{eq4}. First, using a semigroup approach, we establish well-posedness result of our system.
%%%%%%%%%%%%%%%%%%%%%%%%%%%%%%%%%%%%%%%%%%%%%%%%%%%%%%%%%%%%
%%%%%%%%%%%%%%%%%%%%%%%%%%%%%%%%%%%%%%%%%%%%%%%%%%%%%%%%%%%%
%Subsection
%%%%%%%%%%%%%%%%%%%%%%%%%%%%%%%%%%%%%%%%%%%%%%%%%%%%%%%%%%%%
%%%%%%%%%%%%%%%%%%%%%%%%%%%%%%%%%%%%%%%%%%%%%%%%%%%%%%%%%%%%
\subsection{ Well-Posedness}\label{Section-2.1}
\noindent Firstly, we reformulate System \eqref{eq1}-\eqref{eq4} into an evolution problem in an appropriate Hilbert state space.\\
\noindent The energy of System \eqref{eq1}-\eqref{eq4} is given by
	\begin{equation*}
	E(t)=\frac{1}{2}\int_0^L \left(|u_t|^2+a|u_x|^2+|y_t|^2+|y_x|^2\right)dx.
	\end{equation*}
Let $\left(u,u_{t},y,y_{t}\right)$ be a regular solution of \eqref{eq1}-\eqref{eq4}. Multiplying  \eqref{eq1}, \eqref{eq2}    by $u_t,\ y_t,$  respectively, then  using  the boundary conditions \eqref{eq3}, we get
	\begin{equation*}
	E^\prime(t)=- \int_0^L b(x)|u_{tx}|^2dx,
	\end{equation*}
using the definition of the function $b(x)$, we get 	$E^\prime(t)\leq0$. Thus, System \eqref{eq1}-\eqref{eq4} is dissipative in the sense that its energy is a non-increasing function with respect to the time variable $t$. Let us define the energy space $\mathcal{H}$ by
	\begin{equation*}
	\mathcal{H}=(H_0^1(0,L)\times L^2(0,L))^2.
	\end{equation*}

\noindent The energy space $\mathcal{H}$ is equipped with the inner product defined by
	$$
	\left<U,U_1\right>_\mathcal{H}=\int_{0}^Lv\overline{{v}}_1dx+a\int_{0}^Lu_x(\overline{{u}}_1)_xdx+\int_{0}^Lz\overline{{z}}_1dx+\int_{0}^Ly_x(\overline{{y}}_1)_xdx,
	$$
for  all  $U=\left(u,v,y,z\right)$ and $U_1=\left(u_1,v_1,y_1,z_1\right)$ in $\mathcal{H}$.  We use 
$\|U\|_{\mathcal{H}}$ to denote the corresponding norm.
We define the unbounded linear operator $\mathcal{A}: D\left(\mathcal{A}\right)\subset \mathcal{H}\longrightarrow \mathcal{H}$  by 
\begin{equation*}
D(\mathcal{A})=\left\{\begin{array}{l}
\vspace{0.15cm}\displaystyle
U=(u,v,y,z) \in\mathcal{H};\ y\in H^2\left(0,L\right)\cap H_0^1(0,L)\\ 
\\ \displaystyle
v,z\in H_0^1(0,L),(au_{x}+b(x)v_{x})_{x}\in L^2(0,L)
\end{array}\right\}
\end{equation*}
and for all  $U=\left(u, v,y, z\right)\in D\left(\mathcal{A}\right)$,
$$
\mathcal{A}\left(u, v,y, z\right)=\left(v,(au_{x}+b(x)v_{x})_{x}-c(x)z, z, y_{xx}+c(x)v \right)^{\top}.
$$

\noindent If $U=(u,u_t,y,y_t)$ is the state  of System \eqref{eq1}-\eqref{eq4}, then  this system is transformed into the first order evolution equation on the Hilbert space $\mathcal{H}$ given by
\begin{equation}\label{eq-2.9}
U_t=\mathcal{A}U,\quad
U(0)=U_0,
\end{equation}
where $U_0=(u_0,u_1,y_0,y_1)$.		
%%%%%%%%%%%%%%%%%%%%%%%%%%%%%%%%%%%%%%%%%%%%%%%%%%%%%%%%%%%%
%Proposition
%%%%%%%%%%%%%%%%%%%%%%%%%%%%%%%%%%%%%%%%%%%%%%%%%%%%%%%%%%%%
			\begin{pro}\label{Theorem-2.2}
\noindent The unbounded linear operator $\mathcal{A}$ is m-dissipative in the energy space $\mathcal{H}$.
		\end{pro}
%%%%%%%%%%%%%%%%%%%%%%%%%%%%%%%%%%%%%%%%%%%%%%%%%%%%%%%%%%%%
%proof of Proposition
%%%%%%%%%%%%%%%%%%%%%%%%%%%%%%%%%%%%%%%%%%%%%%%%%%%%%%%%%%%%
		\begin{proof}
			For all $U=(u,v,y,z)\in D\left(\mathcal{A}\right)$, we have 
			\begin{equation*}
			\Re\left(\left<\mathcal{A}U,U\right>_{\mathcal{H}}\right)=- \int_{0}^{L} b(x)|v_{x}|^2dx=-\int _{\alpha _1}^{\alpha _3} b_{0} |v_x|^2dx\leq 0,
			\end{equation*}
			which implies that $\mathcal{A}$ is dissipative. Here $\Re$ is used to denote the real part of a complex number. Now, let $F=(f_1,f_2,f_3,f_4)$, we prove the existence of $U=(u,v,y,z)\in D(\mathcal{A})$, solution of the equation 
			\begin{equation}\label{eq-2.10}
			-\mathcal{A}U=F.
			\end{equation}
			Equivalently, one must consider the  system given by 
			\begin{eqnarray}
			-v&=&f_1,\label{eq-2.11}\\
			-(au_{x}+b(x)v_{x})_{x}+c(x)z&=&f_2,\label{eq-2.12}\\
			-z&=&f_3,\label{eq-2.13}\\
			-y_{xx}-c(x)v&=&f_4,\label{eq-2.14}
			\end{eqnarray}
			with the boundary conditions 
			\begin{equation}\label{eq-2.19}
			u(0)=u(L)=0,\quad\text{and}\quad y(0)=y(L)=0.
			\end{equation}
Let $\left(\varphi,\psi\right)\in H_0^1(0,L)\times H_0^1(0,L)$. Multiplying Equations \eqref{eq-2.12} and \eqref{eq-2.14}  by $\overline{\varphi}$ and $\overline{\psi}$ respectively, integrate over $(0,L)$,  we obtain 
			\begin{eqnarray}
				\int_0^L(au_x+b(x)v_{x})\overline{\varphi}_{x}dx+\int_{0}^{L}c(x)z\overline{\varphi}dx&=&\int_0^L f_2\overline{\varphi} dx,\label{Part1}\\
				\int_0^Ly_x\overline{\psi}_x dx-\int_0^L c(x)v\overline{\psi}dx&=&\int_0^L f_4\overline{\psi} dx\label{Part2}.
			\end{eqnarray}
Inserting Equations \eqref{eq-2.11} and  \eqref{eq-2.13} into \eqref{Part1} and \eqref{Part2}, we get 
	\begin{eqnarray}
				\int_0^Lau_x\overline{\varphi}_{x}dx&=&\int_0^L f_2\overline{\varphi} dx+\int_0^L b(x)(f_{1})_{x}\overline{\varphi}_{x}dx+\int_{0}^{L}c(x)f_{3}\overline{\varphi}dx,\label{Part-11}\\
				\int_0^Ly_x\overline{\psi}_x dx&=&\int_0^L f_4\overline{\psi} dx-\int_0^L c(x)f_{1}\overline{\psi}dx\label{Part-12}.
			\end{eqnarray}
Adding Equations \eqref{Part-11} and \eqref{Part-12}, we obtain
			\begin{equation}\label{Part-1-27}
			a\left((u,y),(\varphi,\psi)\right)=L\left(\varphi,\psi\right),\quad \forall \ (\varphi,\psi)\in H_0^1(0,L)\times H_0^1(0,L),
			\end{equation}
			where
			\begin{equation}\label{Part-1-28}
					a\left((u,y),(\varphi,\psi)\right)=a\int_0^L 	 u_x\overline{\varphi}_xdx+\int_0^L y_x\overline{\psi}_{x}dx		
			\end{equation}
			and 
			\begin{equation}\label{Part-1-29}
			L(\varphi,\psi)=\int_0^L f_2\overline{\varphi} dx+\int_0^L b(x)(f_{1})_{x}\overline{\varphi}_{x}dx+\int_{0}^{L}c(x)f_{3}\overline{\varphi}dx+\int_0^L f_4\overline{\psi} dx-\int_0^L c(x)f_{1}\overline{\psi}dx.
			\end{equation}
 Thanks to \eqref{Part-1-28}, \eqref{Part-1-29} , we have that $a$ is a bilinear continuous coercive form on $\left( H_0^1(0,L)\times H_0^1(0,L)\right)^2$,  and $L$ is a linear continuous form on $H_0^1(0,L)\times H_0^1(0,L)$. Then, using Lax-Milgram theorem, we deduce that there exists $(u,y)\in H_0^1(0,L)\times H_0^1(0,L)$ unique solution of the variational problem \eqref{Part-1-27}. Applying the classical elliptic regularity we deduce that $U=(u,v,y,z)\in D(\AA)$ is the unique solution of \eqref{eq-2.10}. The proof is thus complete.
\end{proof}$\\[0.1in]$
%%%%%%%%%%%%%%%%%%%%%%%%%%%%%%%%%%%%%%%%%%%%%%%%%%%%%%%%%%%%
\noindent From Proposition \ref{Theorem-2.2}, the operator $\AA$ is m-dissipative on $\HH$ and consequently, generates a $C_0-$semigroup of contractions $\left(e^{t\mathcal{A}}\right)_{t\geq0}$ following Lummer-Phillips theorem (see in \cite{LiuZheng01} and \cite{Pazy01}). Then the solution of the evolution Equation \eqref{eq-2.9}	admits the following representation
$$
U(t)=e^{t\AA}U_0,\quad t\geq 0,
$$
\noindent which leads to the well-posedness of \eqref{eq-2.9}. Hence, we have the following result.
%%%%%%%%%%%%%%%%%%%%%%%%%%%%%%%%%%%%%%%%%%%%%%%%%%%%%%%%%%%%
%Theorem
%%%%%%%%%%%%%%%%%%%%%%%%%%%%%%%%%%%%%%%%%%%%%%%%%%%%%%%%%%%%
		\begin{theoreme}\label{Theorem-2.3}
Let $U_0\in \mathcal{H}$ then, problem \eqref{eq-2.9} admits a unique weak solution $U$ satisfies 
		$$U(t)\in C^0\left(\R^+,\mathcal{H}\right).$$
Moreover, if $U_0\in D(\mathcal{A})$ then, problem \eqref{eq-2.9} admits a unique strong solution $U$ satisfies 
$$U(t)\in C^1\left(\R^+,\mathcal{H}\right)\cap C^0(\R^+,D(\mathcal{A})).$$
	\end{theoreme}
%%%%%%%%%%%%%%%%%%%%%%%%%%%%%%%%%%%%%%%%%%%%%%%%%%%%%%%%%%%%
%%%%%%%%%%%%%%%%%%%%%%%%%%%%%%%%%%%%%%%%%%%%%%%%%%%%%%%%%%%%
%Subsection
%%%%%%%%%%%%%%%%%%%%%%%%%%%%%%%%%%%%%%%%%%%%%%%%%%%%%%%%%%%%
%%%%%%%%%%%%%%%%%%%%%%%%%%%%%%%%%%%%%%%%%%%%%%%%%%%%%%%%%%%%
\subsection{Strong Stability}\label{Section-2.2}
%\noindent We introduce here the notions of stability that we encounter in this work.
%\noindent To obtain strong stability of the $C_0$-semigroup $\left(e^{t\mathcal{A}}\right)_{t\geq0}$ we mention the theorem of Arendt and Batty in \cite{Arendt01}. 
\noindent This part is devoted for the proof of the strong stability of the $C_0$-semigroup $\left(e^{t\mathcal{A}}\right)_{t\geq0}$.
\noindent To obtain strong stability of the $C_0$-semigroup $\left(e^{t\mathcal{A}}\right)_{t\geq0}$ we use the theorem of Arendt and Batty in \cite{Arendt01} (see Theorem \ref{arendtbatty} in Appendix).
 %%%%%%%%%%%%%%%%%%%%%%%%%%%%%%%%%%%%%%%%%%%%
       % Theorem 
 %%%%%%%%%%%%%%%%%%%%%%%%%%%%%%%%%%%%%%%%%%%%

\begin{theoreme}\label{strong stability}
The $C_0-$semigroup of contractions $\left(e^{t\mathcal{A}}\right)_{t\geq0}$ is strongly stable in $\HH$; i.e. for all $U_0\in \mathcal{H}$, the solution of \eqref{eq-2.9} satisfies 
$$ 
\lim_{t\to +\infty}\|e^{t\mathcal{A}}U_0\|_{\mathcal{H}}=0.
$$
\end{theoreme}
 
\noindent For the proof of Theorem \ref{strong stability}, according to Theorem \ref{arendtbatty}, we need to prove that the operator $\mathcal{A}$ has no pure imaginary eigenvalues and $\sigma\left(A\right)\cap i\mathbb{R}$ contains only a countable number of continuous spectrum of $\mathcal{A}$. The argument for Theorem \ref{strong stability} relies on the subsequent lemmas.

\begin{lemma}\label{ker}
For $\la\in \R$, we have $i\la I -\mathcal{A}$ is injective i.e. 
$$
\ker\left(i\la I-\mathcal{A}\right)=\{0\},\quad \forall \la\in \R.
$$
\end{lemma}

\begin{proof}
From Proposition \ref{Theorem-2.2}, we have $0\in \rho(\mathcal{A})$. We still need to show the result for $\la\in \R^{\ast}$. Suppose that there exists a real number $\la\neq 0$ and  $U=\left(u,v,y,z\right)\in D(\AA)$, such that
\begin{equation*}
\AA U=i\la U.
\end{equation*}
Equivalently, we have 
\begin{eqnarray}
v&=&i\la u,\label{eq-2.20}\\
(au_{x}+b(x)v_{x})_{x}-c(x)z&=&i\la v,\label{eq-2.21}\\
z&=&i\la y,\label{eq-2.22}\\
y_{xx}+c(x)v&=&i\la z.\label{eq-2.23}
\end{eqnarray}
Next, a straightforward computation gives 
\begin{equation*}
0=\Re\left<i\la U,U\right>_{\HH}=\Re\left<\AA U,U\right>_{\HH}=-\int_0^L b(x)|v_x|^2dx=-\int_{\alpha_{1}}^{\alpha_{3}}b_{0}|v_{x}|^2dx,
\end{equation*}
consequently, we deduce that  
\begin{equation}\label{eq-2.25}
b(x)v_x=0\quad \text{in}\quad (0,L)\quad \text{and}\quad v_{x}=0 \quad \text{in} \quad (\alpha_{1},\alpha_{3}).
\end{equation}
It follows, from Equation \eqref{eq-2.20}, that 
\begin{equation}\label{eq-2.26}
u_x=0\quad \text{in}\quad (\alpha_{1},\alpha_{3}).
\end{equation}
Using Equations \eqref{eq-2.21}, \eqref{eq-2.22}, \eqref{eq-2.25}, \eqref{eq-2.26} and the definition of $c(x)$, we obtain 
\begin{equation}\label{neweq-2.29}
y_x=0\quad \text{in}\quad (\alpha_2,\alpha_3).
\end{equation}
Substituting Equations \eqref{eq-2.20}, \eqref{eq-2.22} in Equations \eqref{eq-2.21}, \eqref{eq-2.23}, and using Equation \eqref{eq-2.25} and the definition of $b(x)$ in \eqref{bc},  we get
\begin{eqnarray}
\la^2u+au_{xx}-i\la c(x)y&=&0,\hspace{1cm}\text{in}\quad (0,L)\label{eq-2.27}\\
\la^2y+y_{xx}+i\la c(x)u&=&0,\hspace{1cm}\text{in}\quad (0,L)\label{eq-2.2.8}
\end{eqnarray}
with the boundary conditions 
\begin{equation}\label{boundaryconditionker}
u(0)=u(L)=y(0)=y(L)=0.
\end{equation}
Our goal is to prove that $u=y=0$ on $(0,L)$. For simplicity, we divide the proof into three steps.\\
\textbf{Step 1.} The aim of this step is to show that $u=y=0$ on $(0,\alpha_3)$. so, using Equation \eqref{eq-2.26}, we have
\begin{equation*}
u_x=0\quad \text{in}\quad (\alpha_1,\alpha_2).
\end{equation*}
Using the above equation and Equation \eqref{eq-2.27} and the fact that $c(x)=0$ on $(\alpha_1,\alpha_2)$, we obtain 
\begin{equation}\label{0ua1a2}
u=0\quad \text{in}\quad (\alpha_1,\alpha_2).
\end{equation}
In fact, system \eqref{eq-2.27}-\eqref{boundaryconditionker} admits a unique solution $(u,y)\in C^1\left([0,L]\right)$, then 
\begin{equation}\label{0uxa1a2}
u(\alpha_1)=u_x(\alpha_1)=0.
\end{equation}
Then, from Equations \eqref{eq-2.27} and \eqref{0uxa1a2} and the fact that $c(x)=0$ on $(0,\alpha_1)$, we get 
\begin{equation}\label{0u0a1}
u=0\quad \text{in}\quad (0,\alpha_1).
\end{equation}
Using Equations \eqref{eq-2.26} and \eqref{0ua1a2} and the fact that $u\in C^1([0,L])$, we get
\begin{equation}\label{0ua2a3}
u=0\quad \text{in}\quad (\alpha_1,\alpha_3).
\end{equation}
Now, using Equations \eqref{eq-2.26}, \eqref{neweq-2.29} and the fact that $c(x)=c_0$ on $(\alpha_2,\alpha_3)$ in Equations \eqref{eq-2.27}, \eqref{eq-2.2.8} , we obtain 
\begin{equation}\label{uiy}
u=\dfrac{i c_0}{\la}y\quad \text{in}\quad (\alpha_2,\alpha_3).
\end{equation}
Using Equation \eqref{0ua2a3} in Equation \eqref{uiy}, we obtain
%Since $(u,y)\in C^1([0,L])$ and the fact that $u=0$ on $(\alpha_1,\alpha_2)$, then using Equations \eqref{eq-2.26} and \eqref{uiy}, we obtain 
\begin{equation}\label{0uya2a3}
u=y=0\quad \text{in}\quad (\alpha_2,\alpha_3).
\end{equation}
Since $y\in C^1([0,L])$, then 
\begin{equation}\label{0yyxa2}
y(\alpha_2)=y_x(\alpha_2)=0.
\end{equation}
So, from Equations \eqref{eq-2.2.8} and \eqref{0yyxa2} and the fact that $c(x)=0$ on $(\alpha_1,\alpha_2)$, we obtain 
\begin{equation}\label{0ya1a2}
y=0\quad \text{in}\quad (\alpha_1,\alpha_2).
\end{equation}
Using the same argument over $(0,\alpha_1)$, we get 
\begin{equation}\label{0y0a1}
y=0\quad \text{in}\quad (0,\alpha_1).
\end{equation}
Hence,  from Equations \eqref{0ua1a2}, \eqref{0u0a1}, \eqref{0ua2a3}, \eqref{0uya2a3}, \eqref{0ya1a2} and \eqref{0y0a1}, we obtain $u=y=0$ on $(0,\alpha_3)$. Consequently, we obtain 
$$
U=0\quad \text{in}\quad (0,\alpha_3).
$$
\textbf{Step 2.} The aim of this step is to show that $u=y=0$ on $(\alpha_3,\alpha_4)$. Using Equation \eqref{0uya2a3}, and the fact that $(u,y)\in C^1([0,L])$, we obtain the boundary conditions 
\begin{equation}\label{0uuxyyxa3}
u(\alpha_3)=u_x(\alpha_3)=y(\alpha_3)=y_x(\alpha_3)=0.
\end{equation}
Combining  Equations \eqref{eq-2.27}, \eqref{eq-2.2.8}, and the fact that $c(x)=c_0$ on $(\alpha_3,\alpha_4)$, we get 
\begin{equation}\label{syst1a3a4}
au_{xxxx}+(a+1)\la^2 u_{xx}+\la^2\left(\la^2-c_0^2\right)u=0.
\end{equation}
The characteristic equation of system \eqref{syst1a3a4} is 
$$
P(r):= ar^4+(a+1)\la^2 r^2+\la^2\left(\la^2-c_0^2\right).
$$ 
Setting 
$$
P_0(m):=am^2+(a+1)\la^2m+\la^2\left(\la^2-c_0^2\right).
$$
The polynomial $P_0$ has two distinct real roots $m_1$ and $m_2$ given by:
$$
m_1=\frac{-\la^2(a+1)-\sqrt{\la^4(a-1)^2+4ac_0^2\la^2}}{2a}\quad \text{and}\quad m_2=\frac{-\la^2(a+1)+\sqrt{\la^4(a-1)^2+4ac_0^2\la^2}}{2a}.
$$
It is clear that $m_1<0$ and the sign of $m_2$ depends on the value of $\la$ with respect to $c_0$. We distinguish the following three cases: $\la^2<c_0^2$, $\la^2=c_0^2$ and $\la^2>c_0^2$.\\
\textbf{Case 1.} If $\la^2<c_0^2$, then $m_2>0$. Setting 
$$
r_1=\sqrt{-m_1}\quad \text{and}\quad r_2=\sqrt{m_2}.
$$
Then $P$ has four simple roots $ir_1$, $-ir_1$, $r_2$ and $-r_2$, and hence the general solution of system \eqref{eq-2.27}, \eqref{eq-2.2.8}, is given by 
$$
\left\{\begin{array}{lll}
u(x)&=&\displaystyle{c_1\sin(r_1x)+c_2\cos(r_1x)+c_3\cosh(r_2x)+c_4\sinh(r_2x)},\\[0.1in]
y(x)&=&\displaystyle{\frac{(\la^2-ar_1^2)}{i\la c_0}\left(c_1\sin(r_1x)+c_2\cos(r_1x)\right)+\frac{(\la^2+ar_2^2)}{i\la c_0}\left(c_3\cosh(r_2x)+c_4\sinh(r_2x)\right)},
\end{array}
\right.
$$ 
where $c_j\in \mathbb{C}$, $j=1,\cdots,4$. In this case, the boundary condition in Equation \eqref{0uuxyyxa3}, can be expressed by 
\begin{equation*}
M_1\begin{pmatrix}
c_1\\ c_2\\ c_3\\ c_4
\end{pmatrix}=0,
\end{equation*}
where 
$$
M_1=\begin{pmatrix}
\sin(r_1\alpha_3)&\cos(r_1\alpha_3)&\cosh(r_2\alpha_3)&\sinh(r_2\alpha_3)\\[0.1in]
r_1\cos(r_1\alpha_3)&-r_1\sin(r_1\alpha_3)&r_2\sinh(r_2\alpha_3)&r_2\cosh(r_2\alpha_3)\\[0.1in]
\displaystyle{\frac{(\la^2-ar_1^2)}{i\la c_0}\sin(r_1\alpha_3)}&\displaystyle{\frac{(\la^2-ar_1^2)}{i\la c_0}\cos(r_1\alpha_3)}&\displaystyle{\frac{(\la^2+ar_2^2)}{i\la c_0}\cosh(r_2\alpha_3)}&\displaystyle{\frac{(\la^2+ar_2^2)}{i\la c_0}\sinh(r_2\alpha_3)}\\[0.1in]
\displaystyle{\frac{(\la^2-ar_1^2)}{i\la c_0}r_1\cos(r_1\alpha_3)}&\displaystyle{-\frac{(\la^2-ar_1^2)}{i\la c_0}r_1\sin(r_1\alpha_3)}&\displaystyle{\frac{(\la^2+ar_2^2)}{i\la c_0}r_2\sinh(r_2\alpha_3)}&\displaystyle{\frac{(\la^2+ar_2^2)}{i\la c_0}r_2\cosh(r_2\alpha_3)}
\end{pmatrix}.
$$
The determinant of $M_1$ is given by
$$
\det(M_1)=\frac{r_1r_2a^2\left(r_1^2+r_2^2\right)^2}{\la^2c_0^2}.
$$ 
System \eqref{eq-2.27}, \eqref{eq-2.2.8} with the boundary conditions \eqref{0uuxyyxa3}, admits only a trivial solution $u=y=0$ if and only if $\det(M_1)\neq 0$, i.e. $M_1$ is invertible. Since, $r_1^2+r_2^2=m_2-m_1\neq 0$, then $\det(M_1)\neq 0$. Consequently, if $\la^2<c_0^2$,
we obtain $u=y=0$ on $(\alpha_3,\alpha_4)$.\\
\textbf{Case 2.} If $\la^2=c_0^2$, then $m_2=0$. Setting 
$$
r_1=\sqrt{-m_1}=\sqrt{\frac{(a+1)c_0^2}{a}}.
$$
Then $P$ has two simple roots $ir_1$, $-ir_1$ and $0$ is a double root. Hence the general solution of System \eqref{eq-2.27}, \eqref{eq-2.2.8} is given by 
$$
\left\{\begin{array}{lll}
u(x)&=&c_1\sin(r_1 x)+c_2\cos(r_1x)+c_3x+c_4,\\[0.1in]
y(x)&=&\displaystyle{\frac{(\la^2-ar_1^2)}{i\la c_0}\left(c_1\sin(r_1x)+c_2\cos(r_1x)\right)+\frac{\la}{ic_0}\left(c_3x+c_4\right)},
\end{array}
\right.
$$
where $c_j\in \mathbb{C}$, for $j=1,\cdots,4$. Also, the boundary condition in Equation \eqref{0uuxyyxa3}, can be expressed by 
\begin{equation*}
M_2\begin{pmatrix}
c_1\\ c_2\\ c_3\\ c_4
\end{pmatrix}=0,
\end{equation*}
where 
$$
M_2=\begin{pmatrix}
\sin(r_1\alpha_3)&\cos(r_1\alpha_3)&\alpha_3 &1\\[0.1in]
r_1\cos(r_1\alpha_3)&-r_1\sin(r_1\alpha_3)&1&0\\[0.1in]
\displaystyle{\frac{(\la^2-ar_1^2)}{i\la c_0}\sin(r_1\alpha_3)}&\displaystyle{\frac{(\la^2-ar_1^2)}{i\la c_0}\cos(r_1\alpha_3)}&\displaystyle{\frac{\la\alpha_3}{ic_0}}&\displaystyle{\frac{\la}{ic_0}}\\
\displaystyle{\frac{(\la^2-ar_1^2)}{i\la c_0}r_1\cos(r_1\alpha_3)}&\displaystyle{-\frac{(\la^2-ar_1^2)}{i\la c_0}r_1\sin(r_1\alpha_3)}&\displaystyle{\frac{\la}{ic_0}}&0
\end{pmatrix}.
$$
The determinant of $M_2$ is given by 
$$
\det(M_2)=\frac{-a^2r_1^5}{\la^2c_0^2}.
$$
Since $r_1=\sqrt{-m_1}\neq 0$, then $\det(M_2)\neq 0$. Thus, System \eqref{eq-2.27}, \eqref{eq-2.2.8} with the boundary conditions \eqref{0uuxyyxa3}, admits only a trivial solution $u=y=0$ on $(\alpha_3,\alpha_4)$.\\
\textbf{Case 3.} If $\la^2>c_0^2$, then $m_2<0$. Setting 
$$
r_1=\sqrt{-m_1}\quad \text{and}\quad r_2=\sqrt{-m_2}.
$$ 
Then $P$ has four simple roots $ir_1$, $-ir_1$, $ir_2$ and $-ir_2$, and hence the general solution of System \eqref{eq-2.27}, \eqref{eq-2.2.8} is given by 
$$
\left\{\begin{array}{lll}
u(x)&=&\displaystyle{c_1\sin(r_1x)+c_2\cos(r_1x)+c_3\sin(r_2x)+c_4\cos(r_2x)},\\[0.1in]
y(x)&=&\displaystyle{\frac{(\la^2-ar_1^2)}{i\la c_0}\left(c_1\sin(r_1x)+c_2\cos(r_1x)\right)+\frac{(\la^2-ar_2^2)}{i\la c_0}\left(c_3\sin(r_2x)+c_4\cos(r_2x)\right)},
\end{array}
\right.
$$ 
where $c_j\in \mathbb{C}$, for $j=1,\cdots,4$. Also, the boundary condition in Equation \eqref{0uuxyyxa3}, can be expressed by 
\begin{equation*}
M_3\begin{pmatrix}
c_1\\ c_2\\ c_3\\ c_4
\end{pmatrix}=0,
\end{equation*}
where
$$
M_3=\begin{pmatrix}
\sin(r_1\alpha_3)&\cos(r_1\alpha_3)&\sin(r_2\alpha_3)&\cos(r_2\alpha_3)\\[0.1in]
r_1\cos(r_1\alpha_3)&-r_1\sin(r_1\alpha_3)&r_2\cos(r_2\alpha_3)&-r_2\sin(r_2\alpha_3)\\[0.1in]
\displaystyle{\frac{(\la^2-ar_1^2)}{i\la c_0}\sin(r_1\alpha_3)}&\displaystyle{\frac{(\la^2-ar_1^2)}{i\la c_0}\cos(r_1\alpha_3)}&\displaystyle{\frac{(\la^2-ar_2^2)}{i\la c_0}\sin(r_2\alpha_3)}&\displaystyle{\frac{(\la^2+ar_2^2)}{i\la c_0}\cos(r_2\alpha_3)}\\[0.1in]
\displaystyle{\frac{(\la^2-ar_1^2)}{i\la c_0}r_1\cos(r_1\alpha_3)}&\displaystyle{-\frac{(\la^2-ar_1^2)}{i\la c_0}r_1\sin(r_1\alpha_3)}&\displaystyle{\frac{(\la^2-ar_2^2)}{i\la c_0}r_2\cos(r_2\alpha_3)}&\displaystyle{-\frac{(\la^2-ar_2^2)}{i\la c_0}r_2\sin(r_2\alpha_3)}
\end{pmatrix}.
$$ 
The determinant of $M_3$ is given by 
$$
\det(M_3)=-\frac{r_1r_2a^2(r_1^2-r_2^2)^2}{\la c_0^2}.
$$
Since $r_1^2-r_2^2=m_2-m_1\neq 0$, then $\det(M_3)\neq 0$. Thus, System \eqref{eq-2.27}-\eqref{eq-2.2.8} with the boundary condition \eqref{0uuxyyxa3}, admits only a trivial solution $u=y=0$ on $(\alpha_3,\alpha_4)$. Consequently, we obtain $U=0$ on $(\alpha_3,\alpha_4)$.\\
\textbf{Step 3.} The aim of this step is to show that $u=y=0$ on $(\alpha_4,L)$. From Equations \eqref{eq-2.27}, \eqref{eq-2.2.8} and the fact that $c(x)=0$ on $(\alpha_4,L)$, we obtain the following system 
\begin{equation}\label{finalsystema4L}
\left\{\begin{array}{lll}
\la^2u+au_{xx}&=&0\quad \text{over}\quad (\alpha_4,L)\\
\la^2y+y_{xx}&=&0\quad \text{over}\quad (\alpha_4,L).
\end{array}
\right.
\end{equation}
Since $(u,y)\in C^1([0,L])$ and the fact that $u=y=0$ on $(\alpha_3,\alpha_4)$, we get 
\begin{equation}\label{uuxyyxa4L}
u(\alpha_4)=u_x(\alpha_4)=y(\alpha_4)=y_x(\alpha_4)=0.
\end{equation}
Finally, it is easy to see that System \eqref{finalsystema4L} admits only a trivial solution on $(\alpha_4,L)$ under the boundary condition \eqref{uuxyyxa4L}.\\
Consequently, we proved that $U=0$ on $(0,L)$. The proof is thus complete.
\end{proof}
%%%%%%%%%%%%%%
%%%%%%%%%%%%%%
\begin{lemma}\label{surjec}
For all $\la\in\R$, we have
\begin{equation*}
R(i\la I-\AA)=\HH.
\end{equation*}
\end{lemma}
\begin{proof}
From Proposition \ref{Theorem-2.2}, we have $0\in \rho(\mathcal{A})$. We still need to show the result for $\la\in \R^{\ast}$. Set
$F = (f_1, f_2, f_3, f_4)\in\HH $, we look for $U = (u, v, y, z)\in D(\AA)$ solution of
\begin{equation}\label{eq-2.30}
(i\la I-\AA)U=F.
\end{equation}
Equivalently, we have 
\begin{eqnarray}
v&=&i\la u-f_1,\label{1eq-2.47}\\
i\la v-(au_{x}+b(x)v_{x})_{x}+c(x)z&=&f_2,\label{2eq-2.47}\\
z&=&i\la y-f_3,\label{3eq-2.47}\\
i\la z-y_{xx}-c(x)v&=&f_4.\label{4eq-2.47}
\end{eqnarray}
\begin{comment}
\noindent We distinguish two cases:\\
\textbf{Case\ 1.} If $\lambda=0$. By eliminating $v$ and $z$ in Equations \eqref{2eq-2.47} and \eqref{4eq-2.47} by $-f_1$ and $-f_3$, we obtain
\begin{eqnarray}
-\left(au_x-b(x)\left(f_1\right)_x\right)_x&=&f_2+c(x)f_3,\label{1eq-2.470}\\
-y_{xx}&=&f_4-c(x)f_1,\label{2eq-2.470}
\end{eqnarray}
with fully Dirichlet boundary conditions. Let $\left(\varphi,\psi\right)\in H_{0}^1(0,L)\times H_0^1(0,L)$,  multiplying Equations \eqref{1eq-2.470} and \eqref{2eq-2.470} by $\bar{\varphi}$ and $\bar{\psi}$, integrate over $(0,L)$ respectively and taking the sum, we get 
\begin{equation}\label{eq1case1}
\int_0^Lau_x\bar{\varphi}_xdx+\int_0^Ly_x\psi_xdx=\int_0^L\left(f_2+c(x)f_3\right)\bar{\varphi}dx+\int_0^Lb(x)\left(f_1\right)_x\bar{\varphi}_xdx+\int_0^L\left(f_4-c(x)f_1\right)\bar{\psi}dx
\end{equation}
The left hand side of \eqref{eq1case1} is a bilinear continuous coercive form on $\left(H_0^1(0,L)\times H_0^1(0,L)\right)^2$, and the right hand side of equation \eqref{eq1case1} is a linear continuous form on $H_0^1(0,L)\times H_0^1(0,L)$. Using Lax-Milgram theorem, we deduce that there exists a unique solution $(u,y)\in H_0^1(0,L)\times H_0^1(0,L)$. Applying the classical elliptic regularity we deduce that $U=(u,v,y,z)\in D(\AA)$ is the unique solution of \eqref{eq-2.30}.
\noindent \textbf{Case\ 2.} If $\lambda\in \R^{\ast}$;\\
\end{comment}
Let $\left(\varphi,\psi\right)\in H_{0}^1(0,L)\times H_0^1(0,L)$, multiplying Equations \eqref{2eq-2.47} and \eqref{4eq-2.47} by $\bar{\varphi}$ and $\bar{\psi}$ respectively and integrate over $(0,L)$, we obtain 
\begin{eqnarray}
\int_0^Li\la v\bar{\varphi}dx+\int_0^Lau_x\bar{\varphi}_xdx+\int_0^Lb(x)v_x\bar{\varphi}_xdx+\int_0^Lc(x)z\bar{\varphi}dx=\int_0^Lf_2\bar{\varphi}dx,\label{5eq-2.47}\\
\int_0^Li\la z\bar{\psi}dx+\int_0^Ly_x\bar{\psi}_xdx-\int_0^Lc(x)v\bar{\psi}dx=\int_0^Lf_4\bar{\psi}dx.\label{6eq-2.47}
\end{eqnarray}
Substituting $v$ and $z$ by $i\la u-f_1$ and $i\la y-f_3$ respectively in Equations \eqref{5eq-2.47}-\eqref{6eq-2.47} and taking the sum, we obtain 
\begin{equation}\label{aL}
a\left((u,y),(\varphi,\psi)\right)={\rm L}(\varphi,\psi),\qquad \forall (\varphi,\psi)\in H_0^1(0,L)\times H_0^1(0,L),
\end{equation}
where 
\begin{equation*}
a\left((u,y),(\varphi,\psi)\right)=a_1\left((u,y),(\varphi,\psi)\right)+a_2\left((u,y),(\varphi,\psi)\right)
\end{equation*}
with 
\begin{equation*}
\left\{\begin{array}{l}
a_1\left((u,y),(\varphi,\psi)\right)=\displaystyle{\int_0^L\left( au_{x}\bar{\varphi}_{x}+ y_{x}\bar{\psi}_{x}\right)dx+i\la \int_0^Lb(x)u_{x}\bar{\varphi}_{x}dx},\\[0.1in]
a_2\left((u,y),(\varphi,\psi)\right)=\displaystyle{-\la^2\int_0^L\left(u\bar{\varphi}+y\bar{\psi}\right)dx+i\la\,\int_0^L c(x)\left(y\bar{\varphi}-u\bar{\psi}\right)dx},
\end{array}
\right.
\end{equation*}
and 
\begin{equation*}
\begin{array}{l}
{\rm L}(\varphi,\psi)=\displaystyle{\int_0^L\left(f_{2}+c(x)f_{3}+i\la f_{1}\right)\bar{\varphi}dx+\int_0^L\left(f_{4}-c(x)f_{1}+i\la f_{3}\right)\bar{\psi} dx+\int_0^Lb(x)\left(f_1\right)_x\bar{\varphi}_xdx.}
\end{array}
\end{equation*}
Let $V=H_0^1(0,L)\times H_0^1(0,L)$ and $V'=H^{-1}(0,L)\times H^{-1}(0,L)$ the dual space of $V$. Let us consider the following operators,
$$
\left\{\begin{array}{llll}
{\rm A}:&V&\rightarrow& V'\\
&(u,y)&\rightarrow &{\rm A}(u,y)
\end{array}
\right.
\left\{\begin{array}{llll}
{\rm A_1}:&V&\rightarrow& V'\\
&(u,y)&\rightarrow &{\rm A_1}(u,y)
\end{array}
\right.
\left\{\begin{array}{llll}
{\rm A_2}:&V&\rightarrow& V'\\
&(u,y)&\rightarrow &{\rm A_2}(u,y)
\end{array}
\right.
$$
such that
\begin{equation}\label{AA1A2aa1a2}
\left\{\begin{array}{ll}
\displaystyle{\left({\rm A}(u,y)\right)(\varphi,\psi)=a\left((u,y),(\varphi,\psi)\right)},&\forall (\varphi,\psi)\in H_0^1(0,L)\times H_0^1(0,L),\\[0.1in]
\displaystyle{\left({\rm A_1}(u,y)\right)(\varphi,\psi)=a_1\left((u,y),(\varphi,\psi)\right)},&\forall (\varphi,\psi)\in H_0^1(0,L)\times H_0^1(0,L),\\[0.1in]
\displaystyle{\left({\rm A_2}(u,y)\right)(\varphi,\psi)=a_2\left((u,y),(\varphi,\psi)\right)},&\forall (\varphi,\psi)\in H_0^1(0,L)\times H_0^1(0,L).
\end{array}
\right.
\end{equation}
Our goal is to prove that ${\rm A}$ is an isomorphism operator. For this aim, we divide the proof into three steps.\\
{\textbf{Step 1.}} In this step, we prove that the operator ${\rm A_1}$ is an isomorphism operator. For this goal, following the second equation  of \eqref{AA1A2aa1a2} we can easily verify that $a_1$ is a bilinear continuous coercive form on $H_0^1(0,L)\times H_0^1(0,L)$. Then, by Lax-Milgram Lemma, the operator ${\rm A_1}$ is an isomorphism.\\

\noindent {\textbf{Step 2.}} In this step, we prove that the operator ${\rm A_2}$ is compact. According to the third equation of \eqref{AA1A2aa1a2}, we have 
$$
\abs{a_2\left((u,y),(\varphi,\psi)\right)}\leq C\|(u,y)\|_{L^2(0,L)}\|(\varphi,\psi)\|_{L^2(0,L)}.
$$ 
Finally, using the compactness embedding from $H_0^1(0,L)$ to $L^2(0,L)$ and the continuous embedding from $L^2(0,L)$ into $H^{-1}(0,L)$ we deduce that $A_2$ is compact.\\

\noindent From steps 1 and 2, we get that the operator ${\rm A=A_1+A_2}$ is a Fredholm operator of index zero. Consequently, by Fredholm alternative, to prove that operator ${\rm A}$ is an isomorphism it is enough to prove that ${\rm A}$ is injective, i.e. $\ker\left\{{\rm A}\right\}=\left\{0\right\}$.\\

\noindent {\textbf{Step 3.}} In this step, we prove that $\ker\{{\rm A}\}=\{0\}$. For this aim, let $\left(\tilde{u},\tilde{y}\right)\in \ker\{{\rm A}\}$, i.e. 
$$
a\left((\tilde{u},\tilde{y}),(\varphi,\psi)\right)=0,\quad \forall \left(\varphi,\psi\right)\in H_0^1(0,L)\times H_0^1(0,L).
$$
Equivalently, we have  
\begin{equation}\label{LUL1}
\begin{split}
-\la^2\int_0^L\left(\tilde{u}\bar{\varphi}+\tilde{y}\bar{\psi}\right)dx+i\la\,  \int_0^L c(x)\left(\tilde{y}\bar{\varphi}-\tilde{u}\bar{\psi}\right)dx+\int_0^L\left( a\tilde{u}_{x}\bar{\varphi}_{x} 
+ \tilde{y}_{x}\bar{\psi}_{x}\right)dx\\+i\la \int_0^Lb(x)\tilde{u}_{x}\bar{\varphi}_{x}dx=0.
\end{split}
\end{equation}
Taking $\varphi=\tilde{u}$ and $\psi=\tilde{y}$ in equation \eqref{LUL1}, we get   
\begin{equation*}
-\la^2\int_{0}^{L}\abs{\tilde{u}}^2dx-\la^2\int_{0}^{L}\abs{\tilde{y}}^2dx+a\int_{0}^{L}\abs{ \tilde{u}_{x}}^2dx+\int_{0}^{L}\abs{\tilde{y}_{x}}^2dx-2\la\Im\left(\int_{0}^{L}c(x)\tilde{y}\bar{\tilde{u}}dx\right)+i\la \int_0^Lb(x)\abs{\tilde{u}_{x}}^2 dx=0.
\end{equation*}
Taking the imaginary part of the above equality, we get 
\begin{equation*}
0=\int_0^Lb(x)\abs{\tilde{u}_{x}}^2 dx,
\end{equation*}
 we get, 
\begin{equation}\label{LUL3}
\tilde{u}_{x}=0,\qquad \quad \text{in}\quad \left(\alpha_1,\alpha_3\right).
\end{equation}
Then, we find that 
$$
\left\{\begin{array}{lll}
-\la^2\tilde{u}-a\tilde{u}_{xx}+i\la c(x)\tilde{y}&=&0,\hspace{1cm}\text{in}\quad (0,L)\\[0.1in]
-\la^2\tilde{y}-a\tilde{y}_{xx}-i\la c(x)\tilde{u}&=&0,\hspace{1cm}\text{in}\quad (0,L)\label{eq-2.28}\\[0.1in]
\tilde{u}_{x}=\tilde{y}_{x}&=&0.\hspace{1cm}\text{in}\quad(\alpha_{2},\alpha_{3}) \label{eq-2.29}
\end{array}
\right.
$$
Therefore, the vector $\tilde{U}$ defined by
$$
\tilde{U}=\left(\tilde{u},i\la \tilde{u},\tilde{y},i\la \tilde{y}\right)
$$
belongs to $D(\mathcal{A})$ and we have 
$$
i\la \tilde{U}-\mathcal{A}\tilde{U}=0.
$$
Hence, $\tilde{U}\in \ker\left(i\la I-\mathcal{A}\right)$, then by Lemma \ref{ker}, we get $\tilde{U}=0$, this implies that $\tilde{u}=\tilde{y}=0$. Consequently, $\ker\left\{A\right\}=\left\{0\right\}$.\\

\noindent Therefore, from step 3 and Fredholm alternative, we get that the operator ${\rm A}$ is an isomorphism. It is easy to see that the operator ${\rm L}$ is continuous from $V$ to $L^2(0,L)\times L^2(0,L)$. Consequently, Equation \eqref{aL} admits a unique solution $(u,y)\in H_0^1(L)\times H_0^1(0,L)$. Thus, using $v=i\la u-f_1$, $z=i\la y-f_3$ and using the classical regularity arguments, we conclude that Equation \eqref{eq-2.30} admits a unique solution $U\in D\left(\mathcal{A}\right)$. The proof is thus complete. 
\end{proof}
\\

\noindent \textbf{Proof of Theorem \ref{strong stability}.}  Using Lemma \ref{ker}, we have that $\mathcal{A}$ has non pure imaginary eigenvalues. According to Lemmas \ref{ker}, \ref{surjec} and with the help of the closed graph theorem of Banach, we deduce that $\sigma(\mathcal{A})\cap i\mathbb{R}=\emptyset$. Thus, we get the conclusion by applying Theorem 
\ref{arendtbatty} of Arendt Batty (see Appendix). The proof of the theorem is thus complete.

\section{Lack of the exponential Stability} \label{section-3}
\noindent In this section, our goal is to show that system \eqref{eq1}-\eqref{eq4} in not exponentially stable.
%%%% Global Kelvin-Voigt
\subsection{Lack of exponential stability with global Kelvin-Voigt damping.} In this part, assume that 
\begin{equation}\label{Cond-Global}
b(x)=b_0>0\quad \text{and}\quad c(x)=c_0,\quad \forall\ x\in (0,L).
\end{equation}
We introduce the following theorem.
\begin{theoreme}\label{Thm. Non-Exp-Global}
{\rm Under hypothesis \eqref{Cond-Global}, for $\varepsilon>0$ small enough, we cannot expect the energy decay rate $\frac{1}{t^{\frac{2}{2-\varepsilon}}}$ for all initial data $U_0\in D(\AA)$ and for all $t>0$}. 
\end{theoreme}
%%%%%%%%% Proof of Theorem 3.1 
\begin{proof}
Following Huang and Pr$^¨u$ss \cite{Huang01,pruss01} (see also Theorem \ref{hp} in the Appendix) it is sufficient to show the existence of a real sequences $(\lambda_n)_n$ with $\lambda_n\rightarrow +\infty$, $(U_n)_n\in D(\AA)$, and $(F_n)_n \subset \HH$ such that $\left(i\lambda_nI-\AA\right)U_n=F_n$ is bounded in $\mathcal{H}$ and $\lambda_n^{-2+\varepsilon}\|U_n\|\rightarrow +\infty$. For this aim, take 
$$
F_n=\left(0,0,0,\sin\left(\frac{n\pi x}{L}\right)\right), U_n=\left(A_n\sin\left(\frac{n\pi x}{L}\right), i\lambda_n A_n\sin\left(\frac{n\pi x}{L}\right), B_n\sin\left(\frac{n\pi x}{L}\right),i\la_nB_n\sin\left(\frac{n\pi x}{L}\right)\right),
$$
where 
$$
\lambda_n=\frac{n\pi}{L},\quad A_n=\frac{iL}{c_0n\pi},\quad B_n=-\frac{inb_0\pi}{c_0^2L}-\frac{a-1}{c_0^2}.
$$
Clearly that $U_n\in D(\AA)$, and $F_n$ is bounded in $\HH$. Let us show that $(i\la_n I-\AA)U_n=F_n$. Detailing $(i\la_n I-\AA)U_n$, we get 
$$
(i\la_n I-\AA)U_n=\left(0,D_{1,n}\sin\left(\frac{n\pi x}{L}\right),0,D_{2,n}\sin\left(\frac{n\pi x}{L}\right)\right),
$$
where 
\begin{equation}\label{D1nD2n}
D_{1,n}=\frac{-\left(L^2\lambda_n^2-an^2\pi^2-i\pi^2b_0\lambda_n n^2\right)A_n}{L^2}+iB_nc_0\lambda_n,\quad \text{and}\quad D_{2,n}=-iA_n c_0\lambda_n+\frac{B_n\left(\pi^2n^2-L^2\lambda_n^2\right)}{L^2}.
\end{equation}
Inserting $\la_n,A_n,B_n$ in $D_{1,n}$ and $D_{2,n}$, we get $D_{1,n}=0$ and $D_{2,n}=1$. Hence we obtain 
$$
\left(i\la_nI-\AA\right)U_n=\left(0,0,0,\sin\left(\frac{n\pi x}{L}\right)\right)=F_n.
$$
Now, we have 
$$
\|U_n\|_{\HH}^2\geq \int_0^L\left|i\la_nB_n\sin\left(\frac{n\pi x}{L}\right)\right|^2dx=\frac{L\la_n^2}{2}\abs{B_n}^2\sim \la_n^4.
$$
Therefore, for $\varepsilon>0$ small enough, we have 
$$
\la_n^{-2+\varepsilon}\|U_n\|_{\mathcal{H}}\sim \la_n^{\varepsilon}\rightarrow +\infty.
$$
Then, we cannot expect the energy decay rate $\frac{1}{t^{\frac{2}{2-\varepsilon}}}$.
\end{proof}
\subsection{Lack of exponential stability with Local Kelvin-Voigt damping.}
In this part, under the equal speed wave propagation condition (i.e. $a=1$), we use the classical method developed  by Littman and Markus in \cite{Littman1988} (see also \cite{CurtainZwart01}), to show that system \eqref{eq1}-\eqref{eq4} with Local Kelvin-Voigt damping and global coupling is not exponentially stable. For this aim, assume that 
\begin{equation}\label{Cond-local}
a=1,\quad b(x)=\left\{\begin{array}{ccc}
0&\text{if}&0<x\leq \frac{1}{2},\\
1&\text{if}&\frac{1}{2}<x\leq 1.
\end{array}
\right.,\quad \text{and}\quad c(x)=c\in \R.
\end{equation}
Our main result in this part is following theorem.  
\begin{theoreme}\label{Thm. Non-Exp-local}
Under condition \eqref{Cond-local}. The semigroup of contractions $\left(e^{t\AA}\right)_{t\geq 0}$generated by the operator $\AA$ is not exponentially stable in the energy space $\HH$.
\end{theoreme}
\noindent For the proof of Theorem  \ref{Thm. Non-Exp-local}, we recall the following definitions: the growth bound 
$\omega_0\left(\mathcal{A}\right)$ and the 
the spectral bound $s\left(\mathcal{A}\right)$ of $\mathcal{A}$
are defined respectively as 
$$
\omega_0\left(\mathcal{A}\right)= \inf\left\{\omega\in \mathbb{R}:\  
\text{ there exists a constant }  M_\omega \text{ such that } \forall\ t\geq0,\ \left\|e^{t\mathcal{A}_1}\right\|_{\mathcal{L}(\mathcal{H}_1)}\leq M_\omega e^{\omega t}\right\}
$$
and 
$$
s\left(\mathcal{A}\right)=\sup\left\{\Re\left(\lambda\right):\ \lambda\in \sigma\left(\mathcal{A}\right) \right\}.
$$
Then, according to Theorem 2.1.6 and Lemma 2.1.11 in \cite{CurtainZwart01}, one has that 
$$
s\left(\mathcal{A}_1\right)\leq \omega_0\left(\mathcal{A}_1\right).
$$
By the previous results, one clearly has that $
s\left(\mathcal{A}\right)\leq 0$ and the theorem would follow if equality holds in the previous inequality. It therefore amounts to 
show the existence of a sequence of eigenvalues of $\mathcal{A}$ whose real parts tend to zero.\\[0.1in]
Since $\mathcal{A}$ is dissipative, we fix $\alpha_0>0$ small enough and we study the asymptotic behavior of the eigenvalues $\lambda$ of 
$\mathcal{A}$ in the strip 
$$S=\left\{\lambda \in \mathbb{C}:-\alpha_0\leq \text{Re}(\lambda)\leq 0\right\}.$$
First, we determine the characteristic equation satisfied by the eigenvalues of $\AA$. For this aim, let $\lambda\in \C^{\ast}$ be an eigenvalue of $\AA$ and let $U=(u,\la u,y,\la y)\in D(\AA)$ be an associated eigenvector. Then, the eigenvalue problem is given by 
\begin{eqnarray}
\la^2 u-\left(1+\lambda\right)u_{xx}+c\la y&=&0,\quad x\in (0,1),\label{LE-1}\\
\la^2 y-y_{xx}-c\la u&=&0,\quad x\in (0,1)\label{LE-2},
\end{eqnarray}
with the boundary conditions 
$$
u(0)=u(1)=y(0)=y(1)=0.
$$
We define 
$$
\left\{\begin{array}{ccc}
u^-(x):=u(x),&y^{-}(x):=y(x)&x\in (0,\frac{1}{2}),\\
u^+(x):=u(x),&y^{+}(x):=y(x)&x\in [\frac{1}{2},1).
\end{array}
\right.
$$
Then, system \eqref{LE-1}-\eqref{LE-2} becomes
\begin{eqnarray}
\la^2u^--u^-_{xx}+c\la y^-&=&0,\quad x\in (0,1/2),\label{NE-Trans1}\\
\la^2y^--y^-_{xx}-c\la u^-&=&0,\quad x\in (0,1/2)\label{NE-Trans2},\\
\la^2u^+-(1+\la)u^+_{xx}+c\la y^+&=&0,\quad x\in [1/2,1)\label{NE-Trans3},\\
\la^2y^+-y_{xx}^+-c\la u^+&=&0,\quad x\in [1/2,1),\label{NE-Trans4}
\end{eqnarray}
with the boundary conditions 
\begin{eqnarray}
u^-(0)=y^-(0)=0,\label{NE-Trans5}\\
u^+(1)=y^+(1)=0,\label{NE-Trans6}
\end{eqnarray}
and the continuity conditions 
\begin{eqnarray}
u^{-}(1/2)=u^+(1/2),\label{NE-CC-1}\\
u_x^-(1/2)=(1+\la)u_x^+(1/2),\label{NE-CC-2}\\
y^-(1/2)=y^+(1/2),\label{NE-CC-3}\\
y_x^-(1/2)=y_x^{+}(1/2).\label{NE-CC-4}
\end{eqnarray}
Here and below, in order to handle, in the case where $z$ is a non zero non-real number, we denote by $\sqrt{z}$ the square root of $z$; i.e., the unique complex number whose square is equal to $z$, that is defined by 
$$
\sqrt{z}=\sqrt{\frac{\abs{z}+\Re(z)}{2}}+i\ \text{sign}(\Im(z))\sqrt{\frac{\abs{z}-\Re(z)}{2}}.
$$
Our aim is to study the asymptotic behavior of the largest eigenvalues $\la$ of $\AA$ in $S$. By taking $\la$ large enough, the general solution of System \eqref{NE-Trans1}-\eqref{NE-Trans2} with boundary condition \eqref{NE-Trans5} is given by 
$$
\left\{\begin{array}{ccl}
u^-(x)&=&d_1\dfrac{\la^2-r_1^2}{c\,\la}\sinh(r_1x)+d_2\dfrac{\la^2-r_2^2}{c\,\la}\sinh(r_2x),\\[0.1in]
y^-(x)&=&d_1\sinh(r_1x)+d_2\sinh(r_2x),
\end{array}
\right.
$$  
and the general solution of System \eqref{NE-Trans1}-\eqref{NE-Trans2} with boundary condition \eqref{NE-Trans6} is given by 
$$
\left\{\begin{array}{ccl}
u^+(x)&=&-D_1\dfrac{\la^2-s_1^2}{c\,\la}\sinh(s_1(1-x))-D_2\dfrac{\la^2-s_2^2}{c\,\la}\sinh(s_2(1-x)),\\[0,1in]
y^+(x)&=&-D_1\sinh(s_1(1-x))-D_2\sinh(s_2(1-x)),
\end{array}\right.
$$
where $d_1,d_2,D_1,D_2\in \C$, 
\begin{equation}\label{r1r2}
r_1=\lambda\sqrt{1+\frac{ic}{\la}},\qquad r_2=\lambda\sqrt{1-\frac{ic}{\lambda}}
\end{equation}
and 
\begin{equation}\label{s1s2}
s_1=\la\sqrt{\frac{1+\frac{2}{\la}+\sqrt{1-\frac{4c^2}{\la^3}-\frac{4c^2}{\la^4}}}{2\left(1+\frac{1}{\la}\right)}},\qquad s_2=\sqrt{\la}\sqrt{\frac{\la+2-\la\sqrt{1-\frac{4c^2}{\la^3}-\frac{4c^2}{\la^4}}}{2\left(1+\frac{1}{\la}\right)}}.
\end{equation}
The boundary conditions in \eqref{NE-CC-1}-\eqref{NE-CC-4}, can be expressed by $M(d_1\ d_2\ D_1\ D_2)^{\top}=0$, where 
$$
M=\left(\begin{array}{cccc}
\sinh(\frac{r_1}{2})&\sinh(\frac{r_2}{2})&\sinh(\frac{s_1}{2})&\sinh(\frac{s_2}{2})\\[0.1in]
r_1\cosh(\frac{r_1}{2})&r_2\cosh(\frac{r_2}{2})&-s_1\cosh(\frac{s_1}{2})&-s_2\cosh(\frac{s_2}{2})\\[0.1in]
r_1^2\sinh(\frac{r_1}{2})&r_2^2\sinh(\frac{r_2}{2})&s_1^2\sinh(\frac{s_1}{2})&s_2^2\sinh(\frac{s_2}{2})\\[0.1in]
r_1^3\cosh(\frac{r_1}{2})&r_2^3\cosh(\frac{r_2}{2})&-s_1(s_1^2-\lambda (\lambda^2-s_1^2))\cosh(\frac{s_1}{2})&-s_2(s_2^2-\lambda (\lambda^2-s_2^2))\cosh(\frac{s_2}{2})
\end{array}\right)
$$
System \eqref{NE-Trans1}-\eqref{NE-CC-4} admits a non trivial solution if and only if $det(M)=0$. Using Gaussian elimination, $det(M)=0$ is equivalent to $det(M_1)=0$, where $M_1$ is given by 
$$
M_1=\left(\begin{array}{cccc}
\sinh(\frac{r_1}{2})&\sinh(\frac{r_2}{2})&\sinh(\frac{s_1}{2})&1-e^{-s_2}\\[0.1in]
r_1\cosh(\frac{r_1}{2})&r_2\cosh(\frac{r_2}{2})&-s_1\cosh(\frac{s_1}{2})&-s_2(1+e^{-s_2})\\[0.1in]
r_1^2\sinh(\frac{r_1}{2})&r_2^2\sinh(\frac{r_2}{2})&s_1^2\sinh(\frac{s_1}{2})&s_2^2(1-e^{-s_2})\\[0.1in]
r_1^3\cosh(\frac{r_1}{2})&r_2^3\cosh(\frac{r_2}{2})&-s_1(s_1^2-\lambda (\lambda^2-s_1^2))\cosh(\frac{s_1}{2})&-s_2(s_2^2-\lambda (\lambda^2-s_2^2))(1+e^{-s_2})
\end{array}\right).
$$
Then, we get 
\begin{equation}\label{DetM1}
det(M_1)=F_1+F_2e^{-s_2},
\end{equation}
where
\begin{equation*}
\begin{array}{rl}
F_1=

&\displaystyle{-s_1 s_2 \left(r_1^2-r_2^2\right)
\left(s_1^2-s_2^2\right)\left(\,\lambda+1\right)
\sinh\left(\frac{r_1}{2} \right)
\sinh\left(\frac{r_2}{2} \right)
\cosh\left(\frac{s_1}{2} \right)
 }\\ \noalign{\medskip}

&\displaystyle{+r_1 s_2 \left(r_2^2-s_1^2\right)
\left((\lambda^2-s_2^2)\, \,\lambda +r_1^2-s_2^2\right)
\cosh\left(\frac{r_1}{2} \right)
\sinh\left(\frac{r_2}{2} \right)
\sinh\left(\frac{s_1}{2} \right)
 }\\ \noalign{\medskip}

&\displaystyle{-r_2 s_2 \left(r_1^2-s_1^2\right)
\left((\lambda^2-s_2^2)\, \,\lambda +r_2^2-s_2^2\right)
\sinh\left(\frac{r_1}{2} \right)
\cosh\left(\frac{r_2}{2} \right)
\sinh\left(\frac{s_1}{2} \right) }\\ \noalign{\medskip}

&\displaystyle{-r_1 r_2 \left(r_1^2-r_2^2\right)
\left(s_1^2-s_2^2\right)
\cosh\left(\frac{r_1}{2} \right)
\cosh\left(\frac{r_2}{2} \right)
\sinh\left(\frac{s_1}{2} \right)
 }\\ \noalign{\medskip}

&\displaystyle{+r_2 s_1 \left(r_1^2-s_2^2\right)
\left((\lambda^2-s_1^2)\, \,\lambda +r_2^2-s_1^2\right)
\sinh\left(\frac{r_1}{2} \right)
\cosh\left(\frac{r_2}{2} \right)
\cosh\left(\frac{s_1}{2} \right)
}\\ \noalign{\medskip}

&\displaystyle{-r_1 s_1 \left(r_2^2-s_2^2\right)
\left((\lambda^2-s_1^2)\, \,\lambda +r_1^2-s_1^2\right)
\cosh\left(\frac{r_1}{2} \right)
\sinh\left(\frac{r_2}{2} \right)
\cosh\left(\frac{s_1}{2} \right) }

\end{array}
\end{equation*}
and
\begin{equation*}
\begin{array}{rl}
F_2=

&\displaystyle{-s_1 s_2 \left(r_1^2-r_2^2\right)
\left(s_1^2-s_2^2\right)\left(\,\lambda+1\right)
\sinh\left(\frac{r_1}{2} \right)
\sinh\left(\frac{r_2}{2} \right)
\cosh\left(\frac{s_1}{2} \right)
 }\\ \noalign{\medskip}

&\displaystyle{+r_1 s_2 \left(r_2^2-s_1^2\right)
\left((\lambda^2-s_2^2)\, \,\lambda +r_1^2-s_2^2\right)
\cosh\left(\frac{r_1}{2} \right)
\sinh\left(\frac{r_2}{2} \right)
\sinh\left(\frac{s_1}{2} \right)
 }\\ \noalign{\medskip}

&\displaystyle{-r_2 s_2 \left(r_1^2-s_1^2\right)
\left((\lambda^2-s_2^2)\, \,\lambda +r_2^2-s_2^2\right)
\sinh\left(\frac{r_1}{2} \right)
\cosh\left(\frac{r_2}{2} \right)
\sinh\left(\frac{s_1}{2} \right) }\\ \noalign{\medskip}

&\displaystyle{+r_1 r_2 \left(r_1^2-r_2^2\right)
\left(s_1^2-s_2^2\right)
\cosh\left(\frac{r_1}{2} \right)
\cosh\left(\frac{r_2}{2} \right)
\sinh\left(\frac{s_1}{2} \right)
 }\\ \noalign{\medskip}

&\displaystyle{-r_2 s_1 \left(r_1^2-s_2^2\right)
\left((\lambda^2-s_1^2)\, \,\lambda +r_2^2-s_1^2\right)
\sinh\left(\frac{r_1}{2} \right)
\cosh\left(\frac{r_2}{2} \right)
\cosh\left(\frac{s_1}{2} \right)
}\\ \noalign{\medskip}

&\displaystyle{+r_1 s_1 \left(r_2^2-s_2^2\right)
\left((\lambda^2-s_1^2)\, \,\lambda +r_1^2-s_1^2\right)
\cosh\left(\frac{r_1}{2} \right)
\sinh\left(\frac{r_2}{2} \right)
\cosh\left(\frac{s_1}{2} \right) }.

\end{array}
\end{equation*}
%%%%%%%%%% Real part of Lambda is bounded 
\begin{lemma}
{\rm Let $\la\in \C$ be an eigenvalue of $\AA$. Then, we have $\Re(\la)$ is bounded}.
\end{lemma}
\begin{proof}
Multiplying equations \eqref{NE-Trans1}-\eqref{NE-Trans4} by $u^-,y^-,u^+,y^+$ respectively, then using the boundary conditions, we get
\begin{equation}\label{Bounded-Re(la)}
\|\la u^-\|^2+\|u_x^-\|^2+\|\la y^-\|^2+\|y_x^-\|^2+\|\la u^+\|^2+\left(1+\Re(\la)\right)\|u_x^+\|^2+\|\la y^+\|^2+\|y_x^+\|^2=0.
\end{equation}
Since the operator $\AA$ is dissipative then the real part of $\lambda$ is negative. It is easy to see that $u^+_x\neq 0$, hence using the fact that $\|U\|_{\HH}=1$ in \eqref{Bounded-Re(la)}, we get that $\Re(\la)$ is bounded below. Therefore, there exists $\alpha>0$, such that 
$$
-\alpha\leq \Re(\la)<0.
$$
\end{proof}
%%%%%%%%% Eigenvalues 
\begin{pro}\label{ev}
{\rm Assume that the condition \eqref{Cond-local} holds. Then there exists $n_0\in \mathbb{N}$ sufficiently large and two sequences $\left(\la_{1,n}\right)_{\abs{n}\geq n_0}$ and $\left(\la_{2,n}\right)_{\abs{n}\geq n_0}$ of simple root of $det(M_1)$ satisfying the following asymptotic behavior}:\\
{\rm \textbf{Case 1.} If $\sin\left(\frac{c}{4}\right)\neq 0$, then 
\begin{equation}\label{1eigenvalue1}
\la_{1,n}=2n\pi i+i\pi -\frac{2\sin^2(\frac{c}{4})(1-i\,sign(n))}{\left(3+\cos(\frac{c}{2})\right)\sqrt{\abs{n}\pi}}+O\left(\frac{1}{n}\right)
\end{equation}
and
\begin{equation}\label{1eigenvalue2}
\la_{2,n}=2n\pi i+i\arccos\left(\cos^2\left(\frac{c}{4}\right)\right)-\frac{\gamma}{\sqrt{\abs{n}\pi}}+i\frac{sign(n)\gamma }{\sqrt{\abs{n}\pi}}+O\left(\frac{1}{n}\right),
\end{equation}
where 
$$
\gamma=\frac{\left(\cos(\frac{c}{2})\sin\left(\frac{\arccos\left(\cos^2(\frac{c}{4})\right)}{2}\right)+\sin\left(\frac{3\arccos\left(\cos^2(\frac{c}{4})\right)}{2}\right)\right)}{4\sqrt{1-cos^4\left(\frac{c}{4}\right)}\cos\left(\frac{\arccos\left(\cos^2(\frac{c}{4})\right)}{2}\right)}.
$$
\textbf{Case 2.} If $\sin\left(\frac{c}{4}\right)= 0$, then 
\begin{equation}\label{2eigenvalue1}
\la_{1,n}=2n\pi i+i\pi+\frac{i\ c^2}{32\pi n}-\frac{(4+i\pi)c^2}{64\pi^2n^2}+O\left(\frac{1}{\abs{n}^{\frac{5}{2}}}\right)
\end{equation}
and 
\begin{equation}\label{2eigenvalue2}
\la_{2,n}=2n\pi i+O\left(\frac{1}{n}\right).
\end{equation}
}
\end{pro}
%%%%%%%%%%%%%%%%%%%%%  
The proof of Proposition \ref{ev}, is divided into two lemmas.
\begin{lemma}\label{Finit-of-Determinant}
{\rm Assume that condition \eqref{Cond-local} holds. Let $\lambda$ be largest eigenvalue of $\AA$, then $\lambda$ is large root of the following asymptotic behavior estimate 
\begin{equation}\label{Flambda}
F(\la):= f_0(\lambda)+\frac{f_1(\la)}{\la^{1/2}}+\frac{f_2(\la)}{8\la}+\frac{f_3(\la)}{8\la^{3/2}}+\frac{f_4(\la)}{128\la^{2}}+O(\la^{-5/2}),
\end{equation}
where 
\begin{equation}\label{coefofFlambda}
\left\{\begin{array}{lll}
f_0(\la)&=&\cosh\left(\frac{3\la}{2}\right)-\cosh\left(\frac{\la}{2}\right)\cos\left(\frac{c}{2}\right),\\[0.1in]
f_1(\la)&=&\sinh\left(\frac{3\la}{2}\right)+\sinh\left(\frac{\la}{2}\right)\cos\left(\frac{c}{2}\right),\\ [0.1in]
f_2(\la)&=&c^2\sinh\left(\frac{3\la}{2}\right)-4\cosh\left(\frac{3\la}{2}\right)+4\left(\cosh\left(\frac{\la}{2}\right)\cos\left(\frac{c}{2}\right)+c\sinh\left(\frac{\la}{2}\right)\sin\left(\frac{c}{2}\right)\right),\\[0.1in]
f_3(\la)&=&-8\sinh\left(\frac{3\la}{2}\right)+c^2\cosh\left(\frac{3\la}{2}\right)-12c\cosh\left(\frac{\la}{2}\right)\sin\left(\frac{c}{2}\right)-8\sinh\left(\frac{\la}{2}\right)\cos\left(\frac{c}{2}\right),\\[0.1in]
f_4(\la)&=&-40c^2\sinh\left(\frac{3\la}{2}\right)+(c^4+72c^2+48)\cosh\left(\frac{3\la}{2}\right)+32c\left(c\ \cos\left(\frac{c}{2}\right)+7\sin\left(\frac{c}{2}\right)\right)\sinh\left(\frac{\la}{2}\right)\\[0.1in]
&&-\left(8c^2+8c^3\sin\left(\frac{c}{2}\right)+16(4c^2+3)\cosh\left(\frac{c}{2}\right)\right)\cosh\left(\frac{\la}{2}\right).
\end{array}
\right.
\end{equation}
}
\end{lemma}
%%%%%%%%%%%%%%%%%%%%% Proof of Lemma \label{Finit-of-Determinant}
\begin{proof}
Let $\la$ be a large eigenvalue of $\AA$, then $\la$ is root of $det(M_1)$. In this Lemma, we give an asymptotic development of the function $det(M_1)$ for large $\la$. First, using the asymptotic expansion in \eqref{r1r2}-\eqref{s1s2}, we get 
\begin{equation}\label{finiter1r2s1s2}
\left\{\begin{array}{l}
r_1=\la+\frac{ic}{2}+\frac{c^2}{8\la}-\frac{ic^3}{16\la^2}+O(\la^{-3}),\ r_2=\la-\frac{ic}{2}+\frac{c^2}{8\la}+\frac{ic^3}{16\la^2}+O(\la^{-3}),\\[0.1in]
s_1=\la-\frac{c^2}{2\la}+O(\la^{-5}),\ s_2=\sqrt{\la}-\frac{1}{2\sqrt{\la}}+\frac{4c^2+3}{8\la^{\frac{3}{2}}}+O\left(\la^{-3/2}\right).
\end{array}\right.
\end{equation}
From \eqref{finiter1r2s1s2}, we get 
\begin{equation}\label{1FiniteF1}
\left\{\begin{array}{l}
2ic\la s_1s_2(s_1^2-s_2^2)(\la+1)=ic\la^{11/2}\left(2-\frac{1}{\la}+\frac{3+4c^2}{4\la^2}+O(\la^{-3})\right),\\[0.1in]
r_1s_2(r_2^2-s_1^2)\left((\la^2-s_2^2)\la+r_1^2-s_2^2\right)=-ic\la^{11/2}\left(1-\frac{1-i\,c}{2\la}+\frac{5c^2+3+14i\,c}{8\la^2}+O(\la^{-3})\right),\\[0.1in]
r_2s_2(r_1^2-s_1^2)\left((\la^2-s_2^2)\la+r_2^2-s_2^2\right)=ic\la^{11/2}\left(1-\frac{1+i\,c}{2\la}+\frac{5c^2+3-14i\,c}{8\la^2}+O(\la^{-3})\right),\\[0.1in]
2i\,c\la r_1r_2\left(s_1^2-s_2^2\right)=ic\la^{11/2}\left(\frac{2}{\sqrt{\la}}-\frac{2}{\la^{3/2}}+O(\la^{-5/2})\right),\\[0.1in]
r_2s_1(r_1^2-s_2^2)((\la^2-s_1^2)\la+r_2^2-s_1^2)=-ic\la^{11/2}\left(\frac{1}{\sqrt{\la}}-\frac{2-3ic}{2\la^{3/2}}+O\left(\la^{-5/2}\right)\right),\\[0.1in]
r_1s_1(r_2^2-s_2^2)((\la^2-s_1^2)\la+r_1^2-s_1^2)=ic\la^{11/2}\left(\frac{1}{\sqrt{\la}}-\frac{2+3ic}{2\la^{3/2}}+O(\la^{-5/2})\right).
\end{array}
\right.
\end{equation}
From equation \eqref{1FiniteF1} and using the fact that $\Re(\la)$ is bounded, we get 
{\small{\begin{equation}\label{F1New}
\begin{array}{llr}
\displaystyle{\frac{F_1}{ic\la^{11/2}}}
=-\bigg[\displaystyle{\left(2-\frac{1}{\lambda}+\frac{4c^2+3}{4\lambda^2}\right)
\sinh\left(\frac{r_1}{2} \right)
\sinh\left(\frac{r_2}{2} \right)
\cosh\left(\frac{s_1}{2} \right)
 }\\ \noalign{\medskip}
 
\displaystyle{\hspace{2.5cm}+\left(1-\frac{1}{2\lambda}+\frac{5 c^2+3}{8\lambda^2} \right)
\left(\cosh\left(\frac{r_1}{2}  \right)
\sinh\left(\frac{r_2}{2} \right)
+\sinh\left(\frac{r_1}{2}  \right)
\cosh\left(\frac{r_2}{2} \right)
\right)\sinh\left(\frac{s_1}{2}\right)
 }\\ \noalign{\medskip}

\displaystyle{\hspace{2.5cm}+\left(\frac{i\, c}{2\lambda}+\frac{7 i\, c}{4\lambda^2} \right)
\left(\cosh\left(\frac{r_1}{2}  \right)
\sinh\left(\frac{r_2}{2} \right)
-\sinh\left(\frac{r_1}{2}  \right)
\cosh\left(\frac{r_2}{2} \right)
\right)\sinh\left(\frac{s_1}{2}\right)}
\\ \noalign{\medskip}

\displaystyle{\hspace{2.5cm}+\left(\frac{2}{\sqrt{\lambda}}-\frac{2}{\lambda^{3/2}}\right)
\cosh\left(\frac{r_1}{2}  \right)
\cosh\left(\frac{r_2}{2} \right)
\sinh\left(\frac{s_1}{2} \right)
 }\\ \noalign{\medskip}
 
\displaystyle{\hspace{2.5cm}+\left(\frac{1}{\sqrt{\lambda}}-\frac{1}{\lambda^{3/2}}\right)
\left(\sinh\left(\frac{r_1}{2}  \right)
\cosh\left(\frac{r_2}{2} \right)
+\cosh\left(\frac{r_1}{2}  \right)
\sinh\left(\frac{r_2}{2} \right)
 \right)\cosh\left(\frac{s_1}{2} \right)}
\\ \noalign{\medskip}

\displaystyle{\hspace{2.5cm}+\left(\frac{3i\, c}{2\lambda^{3/2}} \right)
\left(\sinh\left(\frac{r_1}{2}  \right)
\cosh\left(\frac{r_2}{2} \right)
-\cosh\left(\frac{r_1}{2}  \right)
\sinh\left(\frac{r_2}{2} \right)
 \right)\cosh\left(\frac{s_1}{2} \right)+O\left(\lambda^{-5/2}\right)\bigg]}.

\end{array}
\end{equation}
}}
From equation \eqref{1FiniteF1} and using the fact that $\Re(\la)$ is bounded, we get
\begin{equation}\label{FiniteF2}
\begin{array}{lll}
F_2

=-i\, c\, \lambda^{11/2}\bigg[\displaystyle{2

\sinh\left(\frac{r_1}{2} \right)
\sinh\left(\frac{r_2}{2} \right)
\cosh\left(\frac{s_1}{2} \right)
 }\\ \noalign{\medskip}

\displaystyle{\hspace{2.5cm}+
\left(\cosh\left(\frac{r_1}{2}  \right)
\sinh\left(\frac{r_2}{2} \right)
+\sinh\left(\frac{r_1}{2}  \right)
\cosh\left(\frac{r_2}{2} \right)
\right)\sinh\left(\frac{s_1}{2}\right)+O\left(\lambda^{-1/2}\right)\bigg]}.
\end{array}
\end{equation}
Since the real part of $\sqrt{\la}$ is positive, then

%On the other hand, from \eqref{finiter1r2s1s2}, we get 
$$
\lim_{\abs{\la}\to \infty}\la^{-5/2}e^{-\sqrt{\la}}=0,
$$
hence
\begin{equation}\label{finiteexp-s2}
e^{-\sqrt{\la}}=o(\la^{-5/2}),
\end{equation}
then, 
\begin{equation}\label{F2e-s2}
F_2e^{-s_2}=-ic\la^{11/2}\left(o(\la^{-5/2})\right).
\end{equation}
Inserting \eqref{F1New} and \eqref{F2e-s2}, in \eqref{DetM1}, we get 
\begin{equation*}
{\rm det(M_1)}=-ic\,\la^{11/2}F(\la),
\end{equation*}
where, 
\begin{equation}\label{NewF}
\begin{array}{ll}
F(\lambda)

=\displaystyle{\left(1-\frac{1}{2\lambda}+\frac{4c^2+3}{8\lambda^2}\right)

\left(\cosh\left(\frac{r_1+r_2}{2} \right)
-\cosh\left(\frac{r_1-r_2}{2} \right)\right)
\cosh\left(\frac{s_1}{2} \right)
 }\\ \noalign{\medskip}

\displaystyle{+\left(1-\frac{1}{2\lambda}+\frac{5 c^2+3}{8\lambda^2} \right)
\sinh\left(\frac{r_1+r_2}{2}  \right)
\sinh\left(\frac{s_1}{2}\right)
-\left(\frac{i\, c}{2\lambda}+\frac{7 i\, c}{4\lambda^2}\right)
\sinh\left(\frac{r_1-r_2}{2}  \right)
\sinh\left(\frac{s_1}{2}\right)}
\\ \noalign{\medskip}
\displaystyle{+\left(\frac{1}{\sqrt{\lambda}}-\frac{1}{\lambda^{3/2}} \right)
\left(\cosh\left(\frac{r_1+r_2}{2} \right)
+\cosh\left(\frac{r_1-r_2}{2} \right)\right)
\sinh\left(\frac{s_1}{2} \right)
 }\\ \noalign{\medskip}

\displaystyle{+\left(\frac{1}{\sqrt{\lambda}}-\frac{1}{\lambda^{3/2}}\right)
\sinh\left(\frac{r_1+r_2}{2}  \right)\cosh\left(\frac{s_1}{2} \right)+\left(\frac{3i\, c}{2\lambda^{3/2}} \right)
\sinh\left(\frac{r_1-r_2}{2}  \right)\cosh\left(\frac{s_1}{2} \right)+O\left(\lambda^{-5/2}\right)}.
\end{array}
\end{equation}
Therefore, system \eqref{NE-Trans5}-\eqref{NE-CC-4} admits a non trivial solution if and only if ${\rm det(M_1)=0}$, if and only if the eigenvalues of $\AA$ are roots of the function $F$. Next, from \eqref{finiter1r2s1s2} and the fact that real $\la$ is bounded, we get
\begin{equation}\label{finitechsh}
\left\{\begin{array}{l}
\cosh\left(\frac{r_1+r_2}{2}\right)=\cosh(\la)+\frac{c^2\sinh(\la)}{8\,\la}+\frac{c^4\,\cosh(\la)}{128\la^2}+O(\la^{-3}),\\[0.1in]
\cosh\left(\frac{r_1-r_2}{2}\right)=\cos\left(\frac{c}{2}\right)+\frac{c^3\sin\left(\frac{c}{2}\right)}{16\la^2}+O(\la^{-3}),\\[0.1in]
\sinh\left(\frac{r_1+r_2}{2}\right)=\sinh(\la)+\frac{c^2\,\cosh(\la)}{8\la}+\frac{c^4\,\sinh(\la)}{128\la^2}+O(\la^{-3}),\\[0.1in]
\sinh\left(\frac{r_1-r_2}{2}\right)=i\sin\left(\frac{c}{2}\right)-i\frac{c^3\cos\left(\frac{c}{2}\right)}{16\la^2}+O(\la^{-3}),\\[0.1in]
\sinh\left(\frac{s_1}{2}\right)=\sinh(\frac{\la}{2})-\frac{c^2\cosh\left(\frac{\la}{2}\right)}{4\la^2}+O(\la^{-4}),\\[0.1in]
\cosh(\frac{s_1}{2})=\cosh\left(\frac{\la}{2}\right)-\frac{c^2\sinh\left(\frac{\la}{2}\right)}{4\la^2}+O(\la^{-4}).
\end{array}\right.
\end{equation}
Inserting \eqref{finitechsh} in \eqref{NewF}, we get \eqref{Flambda}.
\end{proof}
%%%%%%%%%%%%%%%%%%%%%%%%%%%%%%%%%%%%%%%%%%%%%
%%%%%%%%%%%%%%%%%%%%%%%%%%%%%%%%%%%%%%%%%%%%%
%%%%%%%%%%%%%%%%%%%%%%%%%%%%%%%%%%%%%%%%%%%%%
\begin{lemma}\label{lemma-eigen}
Under condition \eqref{Cond-local}, there exists $n_0\in \mathbb{N}$ sufficiently large and two sequences $\left(\la_{1,n}\right)_{\abs{n}\geq n_0}$ and $\left(\la_{2,n}\right)_{\abs{n}\geq n_0}$ of simple roots of $F$ satisfying the following asymptotic behavior 
\begin{equation}\label{first-asymptoticbehavior}
\la_{1,n}=2in\pi+i\pi+\epsilon_{1,n}\quad \text{where}\quad \lim_{\abs{n}\to+\infty}\epsilon_{1,n}=0
\end{equation}
and 
\begin{equation}\label{second-asymptoticbehavior}
\la_{2,n}=2n\pi i+i\arccos\left(\cos^2\left(\frac{c}{4}\right)\right)+\epsilon_{2,n}\quad \text{where}\quad \lim_{\abs{n}\to+\infty}\epsilon_{2,n}=0.
\end{equation}
\end{lemma}
%%%%%%%%%%%%%%%%%%%%%%%%%%%%%%%%%%%%%%%%%%%%% 
\begin{proof}
First, we look at the roots of $f_0$. From \eqref{coefofFlambda}, we deduce that $f_0$ can be written as 
\begin{equation}\label{Newf0}
f_0(\la)=2\cosh\left(\frac{\la}{2}\right)\left(\cosh(\la)-\cos^2\left(\frac{c}{4}\right)\right).
\end{equation}
Then, the roots of $f_0$ are given by 
$$
\left\{
\begin{array}{lr}
\mu_{1,n}=2n\pi i+i\pi,&n\in \mathbb{Z},\\
\mu_{2,n}=2n\pi i+i\arccos\left(\cos^2\left(\frac{c}{4}\right)\right),&n\in \mathbb{Z}.
\end{array}\right.
$$
Now, with the help of Rouch\'e's Theorem, we will show that the roots of $F$ are close to $f_0$. Let us start with the first family $\mu_{1,n}$. Let $B_n=B\left((2n+1)\pi i,r_n\right)$ be the ball of centrum $(2n+1)\pi i$ and radius $r_n=\abs{n}^{-\frac{1}{4}}$ and $\la\in \partial\, B_n$; i.e. $\la_n=2n\pi i+i\pi+r_ne^{i\theta}$,\, $\theta\in [0,2\pi[$. Then 
\begin{equation}\label{coshsinh}
\cosh\left(\frac{\la}{2}\right)=\frac{i(-1)^nr_ne^{i\theta}}{2}+O(r_n^2),\quad \text{and}\quad \cosh(\la)=-1+O(r_n^2).
\end{equation}
Inserting \eqref{coshsinh} in \eqref{Newf0}, we get 
$$
f_0(\la)=-i(-1)^nr_ne^{i\theta}\left(1+\cos^2\left(\frac{c}{4}\right)+O(r_n^3)\right).
$$
It follows that there exists a positive constant $C$ such that 
$$
\forall\ \la\in \partial\, B_n,\quad \abs{f_0(\la)}\geq C\,r_n=C\abs{n}^{-\frac{1}{4}}.
$$
On the other hand, from \eqref{Flambda}, we deduce that 
$$
\abs{F(\la)-f_0(\la)}=O\left(\frac{1}{\sqrt{\la}}\right)=O\left(\frac{1}{\sqrt{\abs{n}}}\right).
$$
It follows that, for $\abs{n}$ large enough 
$$
\forall \la\in \partial\, B_n,\quad \abs{F(\la)-f_0(\la)}<\,\abs{f_0(\la)}.
$$
Hence, with the help of Rouch\'e's theorem, there exists $n_0\in \mathbb{N}^{\ast}$ large enough, such that $\forall\, \abs{n}\geq n_0$, the first branch of roots of $F$ denoted by $\la_{1,n}$ are close to $\mu_{1,n}$, that is 
\begin{equation}\label{approach1}
\la_{1,n}=\mu_{1,n}+i\pi+\epsilon_{1,n}\quad \text{where}\quad \lim_{\abs{n}\to+\infty}\epsilon_{1,n}=0.
\end{equation}
Passing to the second family $\mu_{2,n}$. Let $\tilde{B}_n=B\left(\mu_{2,n},r_n\right)$ be the ball of centrum $\mu_{2,n}$ and radius 
$$
r_n:=\left\{\begin{array}{lll}
\frac{1}{\abs{n}^{\frac{1}{8}}}&\text{if}&\sin\left(\frac{c}{4}\right)=0,\\[0.1in]
\frac{1}{\abs{n}^{\frac{1}{4}}}&\text{if}&\sin\left(\frac{c}{4}\right)\neq 0,
\end{array}\right.
$$
such that $\la\in \partial\, \tilde{B}_n$; i.e. $\la_n=\mu_{2,n}+r_ne^{i\theta}$,\, $\theta\in [0,2\pi[$. Then, 
\begin{equation*}
\cosh(\la)-\cos^2\left(\frac{c}{4}\right)=\cosh\left(2n\pi i+i\, \arccos\left(\cos^2\left(\frac{c}{4}\right)+r_ne^{i\theta}\right)\right)-\cos^2\left(\frac{c}{4}\right).
\end{equation*}
It follow that, 
\begin{equation}\label{f0case2}
\cosh(\la)-\cos^2\left(\frac{c}{4}\right)=i\, r_n\sqrt{1-\cos^4\left(\frac{c}{4}\right)}e^{i\theta}+\frac{r_n^2\cos^2\left(\frac{c}{4}\right)e^{2i\theta}}{2}+O(r_n^3),
\end{equation}
and 
\begin{equation}\label{1f0case2}
\cosh\left(\frac{\la}{2}\right)=(-1)^n\cos\left(\frac{\arccos\left(\cos^2\left(\frac{c}{4}\right)\right)}{2}\right)+\frac{ir_ne^{i\theta}(-1)^n}{2}\sin\left(\frac{\arccos\left(\cos^2\left(\frac{c}{4}\right)\right)}{2}\right)+O(r_n^2).
\end{equation}
Inserting \eqref{f0case2} and \eqref{1f0case2} in \eqref{Newf0}, we get 
\begin{equation}\label{1Newf0case2}
f_0(\la)=R_1\,e^{i\theta}r_n+R_2\,e^{2i\theta}r_n^2+O(r_n^3),
\end{equation}
where 
\begin{equation}\label{R1R2}
\left\{\begin{array}{l}
R_1=i\,(-1)^n\sqrt{1-\cos^4\left(\frac{c}{4}\right)}\cos\left(\frac{\arccos\left(\cos^2\left(\frac{c}{4}\right)\right)}{2}\right),\\[0.1in]
R_2=-(-1)^n\sqrt{1-\cos^4\left(\frac{c}{4}\right)}\sin\left(\frac{\arccos\left(\cos^2\left(\frac{c}{4}\right)\right)}{2}\right)+(-1)^n\cos^2\left(\frac{c}{4}\right)\cos\left(\frac{\arccos(\cos^2\left(\frac{c}{4}\right))}{2}\right).
\end{array}\right.
\end{equation}
We distinguish two cases:\\
\textbf{Case 1.} If $\sin\left(\frac{c}{4}\right)=0$, then 
$$
R_1=0\quad \text{and}\quad R_2=(-1)^n\neq 0.
$$
It follows that there exists a positive constant $C$ such that 
$$
\forall \la\in \partial\, \tilde{B}_n,\quad \abs{f_0(\la)}\geq C\, r_n^2=C\abs{n}^{-\frac{1}{4}}. 
$$
\textbf{Case 2.} If $\sin\left(\frac{c}{4}\right)\neq 0$, then $R_1\neq 0$. It follows that, there exists a positive constant $C$ such that 
$$
\forall \la\in \partial\, \tilde{B}_n,\quad \abs{f_0(\la)}\geq C\, r_n=C\abs{n}^{-\frac{1}{4}}. 
$$
On the other hand, from \eqref{Flambda}, we deduce that 
$$
\abs{F(\la)-f_0(\la)}=O\left(\frac{1}{\sqrt{\la}}\right)=O\left(\frac{1}{\sqrt{\abs{n}}}\right).
$$
In both cases, for $\abs{n}$ large enough, we have 
$$
\forall\, \la\in \partial \tilde{B}_n,\quad \abs{F(\la)-f_0(\la)}<\abs{f_0(\la)}.
$$
Hence, with the help of Rouch\'e's Theorem, there exists $n_0\in \mathbb{N}^{\ast}$ large enough, such that $\forall \abs{n}\geq n_0$, the second branch of roots of $F$, denoted by $\la_{2,n}$ are close to $\mu_{2,n}$ that is defined in equation \eqref{second-asymptoticbehavior}. The proof is thus complete.
\end{proof}
%%%%%%%%%%%%%%%%%%%%%%%%%%%%%%%%%%%%%%%%%%%%%

\noindent We are now in position to conclude the proof of Proposition \ref{ev}.\\
\textbf{Proof of Proposition \ref{ev}.} The proof is divided into two steps.\\

\noindent \textbf{Calculation of $\epsilon_{1,n}$}. From \eqref{approach1},  we have
\begin{equation}\label{coefla1}
\left\{
\begin{array}{ll}
\displaystyle{\cosh\left(\frac{3\lambda_{1,n}}{2}\right)=-i\,(-1)^n\, \sinh\left(\frac{3\epsilon_{1,n}}{2}\right),\ \sinh\left(\frac{3\lambda_{1,n}}{2}\right)=-i\,(-1)^n\, \cosh\left(\frac{3\epsilon_{1,n}}{2}\right) ,}

\\ \noalign{\medskip}

\displaystyle{\cosh\left(\frac{\lambda_{1,n}}{2}\right)=i\,(-1)^n\, \sinh\left(\frac{\epsilon_{1,n}}{2}\right),\ \sinh\left(\frac{\lambda_{1,n}}{2}\right)=i\,(-1)^n\, \cosh\left(\frac{\epsilon_{1,n}}{2}\right) ,}

\\ \noalign{\medskip}

\displaystyle{\frac{1}{\lambda_{1,n}}= -\frac{i}{2\pi n}+\frac{i}{4\pi n^2}+O\left(\epsilon_{1,n}\, n^{-2}\right)+O\left(n^{-3}\right),\ \frac{1}{\lambda^2_{1,n}}= -\frac{1}{4\pi^2 n^2}+O\left(n^{-3}\right)}

 \\ \noalign{\medskip}
\displaystyle{\frac{1}{\sqrt{\lambda_{1,n}}}=\frac{1-i\sign(n)}{2\sqrt{\pi |n|}}+\frac{i-\sign(n)}{8\sqrt{\pi |n|^3}} +O\left(\epsilon_{1,n}\, |n|^{-3/2}\right)+O\left(|n|^{-5/2}\right)},

 \\ \noalign{\medskip}

\displaystyle{\frac{1}{\sqrt{\lambda^3_{1,n}}}=\frac{-1-i\sign(n)}{4\sqrt{\pi^3 |n|^3}}+O\left(|n|^{-5/2}\right)\, ,\frac{1}{\sqrt{\lambda^5_{1,n}}}=O\left(|n|^{-5/2}\right)}.

\end{array}\right.
\end{equation}
On the other hand, since $\displaystyle{\lim_{|n|\to+\infty}\epsilon_{1,n}=0}$, we have the asymptotic expansion
\begin{equation}\label{52}
\left\{
\begin{array}{ll}
\displaystyle{\sinh\left(\frac{3\epsilon_{1,n}}{2}\right)=\frac{3\epsilon_{1,n}}{2}+O(\epsilon_{1,n}^3) },\ \displaystyle{\cosh\left(\frac{3\epsilon_{1,n}}{2}\right)=1+\frac{9\epsilon_{1,n}}{8}+O(\epsilon_{1,n}^4)},\\ \noalign{\medskip}
\displaystyle{\sinh\left(\frac{\epsilon_{1,n}}{2}\right)=\frac{\epsilon_{1,n}}{2}+O(\epsilon_{1,n}^3) },\ \displaystyle{\cosh\left(\frac{\epsilon_{1,n}}{2}\right)=1+\frac{\epsilon_{1,n}}{8}+O(\epsilon_{1,n}^4)}.
\end{array}
\right.
\end{equation}
Inserting \eqref{52} in \eqref{coefla1}, we get
\begin{equation}\label{53}
\left\{
\begin{array}{ll}
\displaystyle{\cosh\left(\frac{3\lambda_{1,n}}{2}\right)=-\frac{3 i\,(-1)^n \epsilon_{1,n}}{2}+O(\epsilon_{1,n}^3)  ,\ \sinh\left(\frac{3\lambda_{1,n}}{2}\right)=-i\,(-1)^n -  \frac{9i\,(-1)^n\,\epsilon_{1,n}}{8}+O(\epsilon_{1,n}^4) ,}

\\ \noalign{\medskip}

\displaystyle{\cosh\left(\frac{\lambda_{1,n}}{2}\right)= \frac{i\,(-1)^n\,\epsilon_{1,n}}{2}+O(\epsilon_{1,n}^3),\ \sinh\left(\frac{\lambda_{1,n}}{2}\right)=i\,(-1)^n +  \frac{i\,(-1)^n\,\epsilon_{1,n}}{8}+O(\epsilon_{1,n}^4) ,}

\\ \noalign{\medskip}

\displaystyle{\frac{1}{\lambda_{1,n}}= -\frac{i}{2\pi n}+\frac{i}{4\pi n^2}+O\left(\epsilon_{1,n}\, n^{-2}\right)+O\left(n^{-3}\right),\ \frac{1}{\lambda^2_{1,n}}= -\frac{1}{4\pi^2 n^2}+O\left(n^{-3}\right)}

 \\ \noalign{\medskip}
\displaystyle{\frac{1}{\sqrt{\lambda_{1,n}}}=\frac{1-i\sign(n)}{2\sqrt{\pi |n|}}+\frac{i-\sign(n)}{8\sqrt{\pi |n|^3}} +O\left(\epsilon_{1,n}\, |n|^{-3/2}\right)+O\left(|n|^{-5/2}\right)},

 \\ \noalign{\medskip}

\displaystyle{\frac{1}{\sqrt{\lambda^3_{1,n}}}=\frac{-1-i\sign(n)}{4\sqrt{\pi^3 |n|^3}}+O\left(|n|^{-5/2}\right)\, ,\frac{1}{\sqrt{\lambda^5_{1,n}}}=O\left(|n|^{-5/2}\right)}.

\end{array}\right.
\end{equation}
Inserting \eqref{53} in \eqref{Flambda}, we get
\begin{equation}\label{epsilon1}
\begin{array}{ll}

\displaystyle{\frac{\epsilon_{1,n }}{2}\left(3+\cos\left(\frac{c}{2}\right)\right)\left(1+\frac{i}{4\pi\, n}\right)+\frac{\left(1-i\sign(n)\right)\left(1-\cos\left(\frac{c}{2}\right)\right)}{2\sqrt{\pi\, |n|}}+\frac{i\, c\left(4\sin\left(\frac{c}{2}\right)-c\right)}{16\pi n}}

\\ \noalign{\medskip}

\hspace{1cm}\displaystyle{-\frac{\left(2+i\pi\right)\left(1+i\sign(n)\right)\left(1-\cos\left(\frac{c}{2}\right)\right)}{8\sqrt{\pi^3\, |n|^3}} +\frac{4c\left(7-2i\pi\right)\sin\left(\frac{c}{2}\right)+c^2\, (2i\pi+5+4 \cos\left(\frac{c}{2}\right))}{64\pi^2n^2} }
\\ \noalign{\medskip}

\hspace{1cm}\displaystyle{
+O\left(|n|^{-5/2}\right)+O\left(\epsilon_{1,n }\, |n|^{-3/2}\right)+O\left(\epsilon_{1,n }^2\, |n|^{-1/2}\right)+O\left(\epsilon_{1,n }^3\right)=0}.

\end{array}
\end{equation}
We distinguish two cases. \\
\textbf{Case 1.} If $\sin\left(\frac{c}{4}\right)\neq 0$, then $\displaystyle{1-\cos\left(\frac{c}{2}\right)=2\sin^2\left(\frac{c}{4}\right)\neq 0}$, then from \eqref{epsilon1}, we get 
$$
\frac{\epsilon_{1,n}}{2}\left(3+\cos\left(\frac{c}{2}\right)\right)+\frac{\sin^2\left(\frac{c}{4}(1-i\sign(n))\right)}{\sqrt{\abs{n}\pi}}+O(\epsilon_{1,n}^3)+O(\abs{n}^{-1/2}\epsilon_{1,n}^2)+O(n^{-1})=0,
$$
hence, we get
\begin{equation}\label{epsilon1,n}
\epsilon_{1,n}=-\frac{2\sin^2\left(\frac{c}{4}\right)(1-i\sign(n))}{\left(2+\cos\left(\frac{c}{2}\right)\right)}+O(n^{-1}).
\end{equation}
Inserting \eqref{epsilon1,n} in \eqref{approach1}, we get \eqref{2eigenvalue1}.\\
%%%%%%%%%%%
\textbf{Case 2.} If $\sin\left(\frac{c}{4}\right)=0$, 
$$
1-\cos\left(\frac{c}{2}\right)=2\sin^2\left(\frac{c}{4}\right)=0,\, \, \sin\left(\frac{c}{2}\right)=2\sin\left(\frac{c}{4}\right)\cos\left(\frac{c}{4}\right)=0,
$$ 
then, from \eqref{epsilon1}, we get 
\begin{equation}\label{newepsilon1}
\begin{split}
2\epsilon_{1,n}\left(1+\frac{i}{4\pi n}\right)-\frac{i\, c^2}{16\pi n}+\frac{c^2(2i\pi +9)}{64\pi^2n^2}+O\left(\abs{n}^{-5/2}\right)+O\left(\epsilon_{1,n}\abs{n}^{-3/2}\right)\\ +O\left(\epsilon_{1,n}^2\abs{n}^{-1/2}\right)+O\left(\epsilon_{1,n}^3\right)=0.
\end{split}
\end{equation}
By a straightforward calculation in equation \eqref{newepsilon1}, we get 
\begin{equation}\label{1newepsilon1}
\epsilon_{1,n}=\frac{i\, c^2}{32\pi n}-\frac{(4+i\, \pi)c^2}{64\pi^2n^2}+O\left(\abs{n}^{-5/2}\right).
\end{equation}
Inserting \eqref{1newepsilon1} in \eqref{approach1}, we get \eqref{1eigenvalue2}.\\
%%%%%%%%%%%%%%%%%%%%%%%%%%%%%%
\textbf{Calculation of $\epsilon_{2,n}$}. From \eqref{second-asymptoticbehavior}, we have 
\begin{equation}\label{newestimation1}
\frac{1}{\sqrt{\la_{2,n}}}=\frac{1-i\sign(n)}{2\sqrt{\abs{n}\pi}}+O\left(\abs{n}^{-3/2}\right)\quad \text{and}\quad \frac{1}{\la_{2,n}}=O(n^{-1}).
\end{equation}
Inserting \eqref{second-asymptoticbehavior} and \eqref{newestimation1} in 
\eqref{Flambda}, we get 
\begin{equation}\label{newestimation2}
\begin{array}{l}
\displaystyle{\cosh\left(\frac{\la_{2,n}}{2}\right)\left(\cosh(\la_{2,n})-\cos^2\left(\frac{c}{4}\right)\right)}\\[0.1in]
\displaystyle{+\frac{(1-i\sign(n))\left(\sinh\left(\frac{3\la_{2,n}}{2}\right)+\sinh\left(\frac{\la_{2,n}}{2}\right)\cos\left(\frac{c}{2}\right)\right)}{4\sqrt{\abs{n}\pi}}+O(n^{-1})=0}.
\end{array}
\end{equation}
On the other hand, we have 
\begin{equation}\label{newestimation3}
\begin{array}{lll}
\cosh(\la_{2,n})-\cos^2\left(\frac{c}{4}\right)&=&\cosh\left(2n\pi i+i\, \arccos\left(\cos^2\left(\frac{c}{4}\right)\right)+\epsilon_{2,n}\right)-\cos^2\left(\frac{c}{4}\right)\\[0.1in]
&=&\cos^2\left(\frac{c}{4}\right)\cosh(\epsilon_{2,n})+i\sqrt{1-\cos^4\left(\frac{c}{4}\right)}\sinh(\epsilon_{2,n})-\cos^2\left(\frac{c}{4}\right)\\[0.1in]
&=&i\,\epsilon_{2,n}\sqrt{1-\cos^4\left(\frac{c}{4}\right)}+O(\epsilon_{2,n}^2),
\end{array}
\end{equation}
and 
\begin{equation}\label{newestimation4}
\left\{\begin{array}{lll}
\cosh\left(\frac{\la_{2,n}}{2}\right)&=&(-1)^n\cos\left(\frac{\arccos\left(\cos^2\left(\frac{c}{4}\right)\right)}{2}\right)+O(\epsilon_{2,n}),\\[0.1in]
\sinh\left(\frac{\la_{2,n}}{2}\right)&=&i(-1)^n\sin\left(\frac{\arccos\left(\cos^2\left(\frac{c}{4}\right)\right)}{2}\right)+O(\epsilon_{2,n}),\\[0.1in]
\sinh\left(\frac{3\la_{2,n}}{2}\right)&=&i(-1)^n\sin\left(\frac{3\arccos\left(\cos^2\left(\frac{c}{4}\right)\right)}{2}\right)+O(\epsilon_{2,n}).
\end{array}\right.
\end{equation}
Inserting \eqref{newestimation3} and \eqref{newestimation4} in \eqref{newestimation2}, we get 
\begin{equation}\label{newestimation5}
\begin{array}{l}
\epsilon_{2,n}\cos\left(\frac{\arccos\left(\cos^2\left(\frac{c}{4}\right)\right)}{2}\right)\sqrt{1-\cos^4\left(\frac{c}{4}\right)}+O\left(\frac{\epsilon_{2,n}}{n}\right)+O\left(\frac{1}{n}\right)\\[0.1in]
\displaystyle{+\frac{(1-i\,\sign(n))\left(\sin\left(3\frac{\arccos\left(\cos^2\left(\frac{c}{4}\right)\right)}{2}\right)+\cos\left(\frac{c}{2}\right)\sin\left(\frac{\arccos\left(\cos^2\left(\frac{c}{4}\right)\right)}{2}\right)\right)}{4\sqrt{\abs{n}\pi}}=0}. 
\end{array}
\end{equation}
We distinguish two cases. \\
\textbf{Case 1.}  If $\sin\left(\frac{c}{4}\right)\neq 0$, then from \eqref{newestimation5}, we get 
\begin{equation}\label{newestimation6}
\epsilon_{2,n}=-\frac{\left(\cos(\frac{c}{2})\sin\left(\frac{\arccos\left(\cos^2(\frac{c}{4})\right)}{2}\right)+\sin\left(\frac{3\arccos\left(\cos^2(\frac{c}{4})\right)}{2}\right)\right)(1-i\, \sign(n))}{4\sqrt{1-cos^4\left(\frac{c}{4}\right)}\cos\left(\frac{\arccos\left(\cos^2(\frac{c}{4})\right)}{2}\right)\sqrt{\pi \abs{n}}}+O(n^{-1}).
\end{equation}
Inserting \eqref{newestimation6} in \eqref{second-asymptoticbehavior}, we get \eqref{1eigenvalue2}.\\
\textbf{Case 2.} If $\sin\left(\frac{c}{4}\right)=0$, we get 
\begin{equation}\label{newestimation7}
\epsilon_{2,n}=O(n^{-1}).
\end{equation}
Inserting \eqref{newestimation7} in \eqref{second-asymptoticbehavior}, we get \eqref{2eigenvalue2}. Thus, the proof is complete.
%%%%%%%%%%%%%%%%

\noindent \textbf{Proof of Theorem \ref{Thm. Non-Exp-local}.} From Proposition \ref{ev}, the operator $\AA$ has two branches of eigenvalues such that the real parts tending to zero. Then the energy corresponding to the first and second branch of eigenvalues is not exponentially decaying. Then the total energy of the wave equations with local Kelvin-Voigt damping with global coupling are not exponentially stable in the equal speed case.

%%%%%%%%%%%%%%%%%%%%%%%%%%%%%%%%%%%%%%%%%%%%%%%%%%%%%%%%%%%%%%%%%%%%%%%%%%%%%%%%%%%%%%%%%%%%%%%%%%%%%%%%%%%%%%%%%%%%%%%%%%%%%%%%%%%%%%%%%%%%%%%%%%%%%%%%%%%%%%%%%%%%%%%%%%%%%%%%%%%%%%
\section{Polynomial Stability}\label{Section-4}
%%%%%%%%%%%%%%%%%%%%%%%%%%%%%%%%%%%%%%%%%%%%%%%%%%%%%%%%%%%%%%%%%%%%%%%%%%%%%%%%%%%%%%%%%%%%%%%%%%%%%%%%%%%%%%%%
\noindent From  Section \ref{section-3},  System \eqref{eq1}-\eqref{eq4} is not uniformly (exponentially) stable, so we look for a polynomial decay rate.
%%%%%%%%%%%%%%%%%%%%%%%%%%%%%%%%%%%%%%%%%%%%%%%%%%
                % Theorem %
%%%%%%%%%%%%%%%%%%%%%%%%%%%%%%%%%%%%%%%%%%%%%%%%%%
\noindent As the condition $ i\mathbb{R}\subset \rho(\mathcal{A})$ is already checked in Lemma \ref{ker}, following Theorem \ref{bt}, it remains to prove that condition \eqref{h1} holds. This is made with the help of a specific  multiplier and by using the exponential decay of an auxiliary problem.
%Firstly, like as \cite{Nicaise1,Nicaise2},
%Define the auxiliary space $\HH_{a}=\left(H^1_0(0,L)\times L^2(0,L)\right)^2$, and the auxiliary unbounded linear operator $\mathcal{A}_{a}$ in $\HH_{a}$ by
%\vspace{0.1cm}
%We introduce the following condition:\\
%\vspace{0.1cm}
%(H): the problem \eqref{aux-prb} is uniformly stable in the energy space $\HH_a$.\\
Our main result in this section is the following theorem.
%%%%%%%%%%%%%%%%%%%%%%%%%%%%%%%%%%%%%%%%%%%%%%%%%%
                % Theorem %
%%%%%%%%%%%%%%%%%%%%%%%%%%%%%%%%%%%%%%%%%%%%%%%%%%
\begin{theoreme}\label{Theorem-4.1}
{\rm There exists a constant $c>0$ independent of $U_0$, such that the energy of system \eqref{eq1}-\eqref{eq4} satisfies the following estimation:
\begin{equation}\label{EnergyGeneral}
E(t)\leq \frac{c}{t}\|U_0\|^2_{D(\AA)},\quad \forall t>0,\ \forall U_0\in D(\AA).
\end{equation}}
\end{theoreme}
\noindent According to Theorem \ref{bt}, by taking $\ell=2$, the polynomial energy decay \eqref{EnergyGeneral} holds if the following conditions
\begin{equation}\label{H1}\tag{H1}
i\R \subset\rho(A),
\end{equation}
and
\begin{equation}\label{H2}\tag{H2}
\sup_{\lambda\in\mathbb{R}}\left\|\left(i\lambda I-\mathcal{A}\right)^{-1}\right\|_{\mathcal{L}\left(\mathcal{H}\right)}=O\left(|\lambda|^2\right),
\end{equation}
are satisfied. Condition \eqref{H1} is already proved in Lemma \ref{ker}. We will prove condition \eqref{H2} using an argument of contradiction. For this purpose, suppose that \eqref{H2} is false, then there exists $\left\{(\lambda_n,U_n=\left(u_n,v_n,y_n,z_n\right))\right\}_{n\geq 1}\subset \mathbb{R}\times D\left(\mathcal{A}\right)$ and 
\begin{equation}\label{eq-4.2}
\lambda_n\to+\infty,\ \ \|U_n\|_{\mathcal{H}}=1,
\end{equation}
%%%%%%%%%%%%%%%%%%%%%%%%%%%%%%%%%%%%%%%%%%%%%%%%%%
such that
\begin{equation}\label{eq-4.3}
\lambda_n^2\left(\ i\lambda_n U_n-\mathcal{A}U_n\right)=\left(f_{1,n},g_{1,n},f_{2,n},g_{2,n}\right):=F_n\to 0\ \text{ in } \mathcal{H}.
\end{equation}
%%%%%%%%%%%%%%%%%%%%%%%%%%%%%%%%%%%%%%%%%%%%%%%%%%%%%%%%%%%%
%For simplicity, we replace $\la_{n}$ by $\la$; $U_{n}=(u_{n},v_{n},y_n,z_{n})$ by $U=(u,v,y,z)$ and   $F_{n}=\la_{n}^{2}(i\la_{n}I-\mathcal{A})U_{n}=(f_{1,n},g_{1,n},f_{2,n},g_{2,n})$ by $F=(f_{1},g_{1},f_{2},g_{2})$. Detailing \eqref{eq-4.3}, we obtain 
For simplicity, we drop the index $n$. Detailing Equation \eqref{eq-4.3}, we obtain 
\begin{eqnarray}
i\la u-v&=&\lambda^{-2} f_1\longrightarrow 0\ \ \text{in}\ \ H_0^1(0,L),\label{eq-4.4}\\ \noalign{\medskip}
i\la v-(au_{x}+b(x)v_{x})_{x}+c(x)z&=&\lambda^{-2} g_1\longrightarrow 0\ \ \text{in}\ \ L^2(0,L),\label{eq-4.5}\\ \noalign{\medskip}
i\la y-z&=&\lambda^{-2} f_2\longrightarrow 0\ \ \text{in}\ \ H_0^1(0,L),\label{eq-4.6}\\ \noalign{\medskip}
i\la z- y_{xx}-c(x)v&=&\lambda^{-2} g_1\longrightarrow 0\ \ \text{in}\ \ L^2(0,L).\label{eq-4.7}
\end{eqnarray}
Here we will check the condition \eqref{H2} by finding a contradiction with \eqref{eq-4.2} such as $\left\|U\right\|_{\HH}=o(1)$. For clarity, we divide the proof into several lemmas.
%%%%%%%%%%%%%%%%%%%%%%%%%%%%%%%%%%%%%%%%%%%%%%%%%%%%%%%%%%%%
By taking the inner product of \eqref{eq-4.3} with $U$ in $\mathcal{H}$, we remark that
\begin{equation*}
\int _0^L b(x)\left|v_{x}\right|^2dx=-\Re\left(\left<\AA U,U\right>_{\HH}\right)=\Re\left(\left<\left(i\la I-\AA\right)U,U\right>_{\HH}\right)=o\left(\lambda^{-2}\right).
\end{equation*}
Then,
\begin{equation}\label{eq-4.9}
\int _{\alpha_1}^{\alpha_{3}}\left|v_{x}\right|^2dx=o\left(\lambda^{-2}\right).
\end{equation}

%%%%%%%%%%%%%%%%%%%%%%%%%%%%%%%%%%%%%%%%%%%%%%%%%%%%%%%%%%%%
%Lemma
%%%%%%%%%%%%%%%%%%%%%%%%%%%%%%%%%%%%%%%%%%%%%%%%%%%%%%%%%%%%
\begin{rem}\label{remark-4.1} 
\noindent {\rm Since $v$ and $z$ are uniformly bounded in $L^2(0,L)$, then from equations \eqref{eq-4.4} and \eqref{eq-4.6}, the solution $(u,v,y,z)\in D(\AA)$ of \eqref{eq-4.4}-\eqref{eq-4.7} satisfies the following asymptotic behavior estimation
\begin{eqnarray}
\|u\|&=&O\left(\lambda^{-1}\right),\label{eq-4.10}\\ \noalign{\medskip}
\|y\|&=&O\left(\lambda^{-1}\right).\label{eq-4.11}
\end{eqnarray}  
\noindent Using equation \eqref{eq-4.4}, and equation \eqref{eq-4.9} we get
\begin{equation}\label{eq-4.12}
\int _{\alpha_1}^{\alpha_{3}}\left|u_{x}\right|^2dx=o\left(\lambda^{-4}\right).
\end{equation}}
\end{rem}
%%%%%%%%%%%%%%%%%%%%%%%%%%%%%%%%%%%%%%%%%%%%%%%%%%%%%%%%%%%%%%%%%%%%%%%%%%%%%%%%%%%%%%%%%%%%%%%%%%%%%%%%%%%%%%%%%%%%%%%%%%%%%%%%%%%%%%%%%%%%%%%%%%%%%%%%%%%%%%%%%%%%%%%%%%%%%       %%%%%%%%%%%%%%%%%%%%%%%%%%%%%%%

%%%%%%%%%%%%%%%%%%%%%%%%%%%%%%%%%%%%%%%%%%%%%%%%%%%%%%%%%%%%%%%%%%%%%%%%%%%%%%%%%%%%%%%%%%%%%%%%%%                %%%%%%%%%%%%%%%%%%%%%%%%%%%%%%%%5

\begin{lemma}\label{lem.2}
{\rm Let $\varepsilon<\frac{\alpha_3-\alpha_1}{4}$,\ the solution $(u,v,y,z)\in D(\AA$) of the system \eqref{eq-4.4}-\eqref{eq-4.7} satisfies the following estimation

\begin{equation}\label{est-4.1}
\int_{\alpha_{1}+\varepsilon}^{\alpha_{3}-\varepsilon}\left|v\right|^2 dx=o(1)\quad \text{and}\quad \int_{\alpha_{1}+\varepsilon}^{\alpha_{3}-\varepsilon}\abs{\la u}^2dx=o(1).
\end{equation}}
\end{lemma}

\begin{proof}
We define the function $\rho\in C_{0}^{\infty}(0,L)$ by
\begin{equation}
\rho(x)=\left\{\begin{array}{ccc}
1&\text{if}&x\in (\alpha_{1}+\epsilon,\alpha_{3}-\epsilon),\\
0&\text{if}& x\in (0,\alpha_{1})\cup(\alpha_{3},L),\\
0\leq\rho \leq 1&& elsewhere.
\end{array}\right.
\end{equation}
Multiply equation \eqref{eq-4.5} by $\dfrac{1}{\lambda}\rho\bar{v}$, integrate over $(0,L)$, using the fact that $\|g_1\|_{L^2(0,L)}=o(1)$ and $v$ is uniformly bounded in $L^2(\Omega)$, we get
\begin{equation}\label{est-41.2}
 \int_0^L i\rho \left|v\right|^2 dx+\dfrac{1}{\lam}\int_0^L (au_{x}+b(x)v_{x})\left(\rho^{\prime} \bar{v}+\rho\bar{v}_{x}\right)dx+\dfrac{1}{\lam}\int_0^L  c(x)z\rho\bar{v}dx=o(\lam ^{-3}).
\end{equation}
Using Equation \eqref{eq-4.9}, Remark \ref{remark-4.1} and the fact that $v$ and $z$ are uniformly bounded in $L^2(\Omega)$, we get 
\begin{equation}\label{1est-41.2}
\dfrac{1}{\lam}\int_0^L (au_{x}+b(x)v_{x})\left(\rho^{\prime} \bar{v}+\rho\bar{v}_{x}\right)dx=o(\la^{-2})\quad \text{and}\quad \dfrac{1}{\lam}\int_0^L  c(x)z\rho\bar{v}dx=o(1).
\end{equation}
Inserting Equation \eqref{1est-41.2} in Equation \eqref{est-41.2}, we obtain 
\begin{equation}\label{firstfirstfirst}
\int_0^L i\rho\left|v\right|^2dx=o(1).
\end{equation}
Hence, we obtain the first estimation in Equation \eqref{est-4.1}. Now, multiplying Equation \eqref{eq-4.4} by $\lam \rho\bar{u}$ integrate over $(0,L)$ and using the fact that $\|f_1\|_{H_0^1(\Omega)}=o(1)$ and Remark \ref{remark-4.1}, we get 
\begin{equation*}
\int_0^L i\rho\left|\lam u\right|^2dx-\int_0^L\rho\lam v \bar{u}dx=o(\lam^{-2}).
\end{equation*}
Using Equation \eqref{firstfirstfirst}, we get 
\begin{equation*}
\int_0^L i\rho\left|\lam u\right|^2dx=o(1).
\end{equation*}
Then, we obtain the desired second estimation in Equation \eqref{est-4.1}.
\end{proof}
%%%%%%%%%%%%%%%%%%%%%%%%%%%%%%%%%%%%%%%%%%%%
%%%%%%%%%%%%%%%%%%%%%%%%%%%%%%%%%%%%%%%%%%%%
$\newline$ \\
\noindent Inserting equations \eqref{eq-4.4} and \eqref{eq-4.6} respectively in equations \eqref{eq-4.5} and \eqref{eq-4.7}, we get
\begin{eqnarray}
\lam ^{2}u+(au_{x}+b(x)v_{x})_{x}-i\lam c(x)y&=&F_1,\label{eq-4.13}\\
\lam ^{2}y+y_{xx}+i\lam c(x)u&=&F_2,\label{eq-4.14}
\end{eqnarray}
where 
\begin{equation}\label{F1F2}
F_1=-\lam^{-2}g_{1}-i\lam^{-1}f_{1}-c(x)\lam^{-2}f_{2}\quad \text{and}\quad F_2=-\lam^{-2}g_{2}-i\lam^{-1}f_{2}+c(x)\lam^{-2}f_{1}.
\end{equation}
%%%%%%%%%%%%%%%%%%%%%%
%%%%%%%% Estimation on \la y and z
%%%%%%%%%%%%%%%%%%%%%%
\begin{lemma}\label{lem-4.2}
{\rm Let $\varepsilon<\frac{\alpha_3-\alpha_1}{4}$, the solution $(u,v,y,z)\in D(\AA$) of the system \eqref{eq-4.4}-\eqref{eq-4.7} satisfies the following estimation
\begin{equation}\label{est-4.2}
\int_{\alpha_{2}}^{\alpha_{3}-2\varepsilon}\left|\lam y\right|^2 dx=o(1)\quad\mbox{and}\quad \int_{\alpha_{2}}^{\alpha_{3}-2\varepsilon}\left|z\right|^{2}dx=o(1).
\end{equation}}
\end{lemma}
%%%%%%%%%%%%%%%%%%%%%%
\begin{proof}
We define the function $\zeta\in C_{0}^{\infty}(0,L)$ by
\begin{equation}
\zeta(x)=\left\{\begin{array}{ccc}
1&\text{if}&x\in (\alpha_{1}+2\varepsilon,\alpha_{3}-2\varepsilon),\\
0&\text{if}& x\in (0,\alpha_{1}+\varepsilon)\cup(\alpha_{3}-\varepsilon,L),\\
0\leq\zeta\leq 1&& elsewhere.
\end{array}\right.
\end{equation}
Multiply equations \eqref{eq-4.13} by $\lam\zeta\bar{y}$ and \eqref{eq-4.14} by $\lam\zeta\bar{u}$ respectively, integrate over $(0,L)$, using Remark \ref{remark-4.1} and the fact that $\|F\|_{\mathcal{H}}=\|(f_1,g_1,f_2,g_2)\|_{\mathcal{H}}=o(1)$, we get
\begin{equation}\label{eq-4.15}
\int_0^L\lam^{3}\zeta u\bar{y}dx-\int_0^L\lam\left(au_{x}+b(x)v_{x}\right)(\zeta^{\prime}\bar{y}+\zeta\bar{y}_{x})dx-i\int_0^L c(x)\zeta(x)\left|\lam y \right|^2 dx=o(\la^{-1})
\end{equation}
and
\begin{equation}\label{eq-4.16}
\int_0^L\lam^{3}\zeta y\bar{u}dx-\int_0^L\lam y_{x}\zeta^{\prime}\bar{u}_{x}dx-\int_0^L\lam y_{x}\zeta\bar{u}_{x}dx+i\int_0^L c(x)\zeta(x)\left|\lam u \right|^2 dx=o(\la^{-1}).
\end{equation}
Using Remark \ref{remark-4.1}, Lemma \ref{lem.2} and the fact that $y_x$ is uniformly bounded in $L^2(0,L)$, we get 
\begin{equation}\label{1eq-4.16}
\int_0^L\lam\left(au_{x}+b(x)v_{x}\right)(\zeta^{\prime}\bar{y}+\zeta\bar{y}_{x})dx=o(1),\quad -\int_0^L\lam y_{x}\zeta^{\prime}\bar{u}_{x}dx=o(1)\quad \text{and}\quad  \int_0^L\lam y_{x}\zeta\bar{u}_{x}dx=o(1).
\end{equation}
Using Lemma \ref{lem.2}, we have that
\begin{equation}\label{2eq-4.16}
\int_0^L c(x)\zeta\left|\lam u \right|^2 dx=o(1).
\end{equation}
Inserting Equations \eqref{1eq-4.16} and \eqref{2eq-4.16} in Equations \eqref{eq-4.15} and \eqref{eq-4.16}, and summing the result by taking the imaginary part, and using the definition of the functions $c$ and $\zeta$, we get the first estimation of Equation \eqref{est-4.2}.\\
%Inserting Equation \eqref{1eq-4.16} in Equations \eqref{eq-4.15} and \eqref{eq-4.16}, and summing the result by taking the imaginary part, we obtain 
%\begin{equation}
%-\int_0^L c(x)\zeta\left|\lam y \right|^2 dx+\int_0^L c(x)\zeta\left|\lam u \right|^2 dx=o(1).
%\end{equation}
%Then, using Lemma \ref{lem.2}, we obtain 
%\begin{equation}\label{2eq-4.16}
%-\int_0^L c(x)\zeta\left|\lam y \right|^2 dx=o(1).
%\end{equation}
%By the definition of the function $c$ and $\zeta$, we get the first estimation in Equation \eqref{est-4.2}.\\
Now, multiplying equation \eqref{eq-4.6} by $\bar{z}$, integrating over $(\alpha_2,\alpha_3-2\varepsilon)$ and using the fact that $\|f_2\|_{H_0^1(0,L)}=o(1)$ and $z$ is uniformly bounded in $L^2(0,L)$, in particular in $L^2(\alpha_2,\alpha_3-2\varepsilon)$, we get
\begin{equation*}
\int_{\alpha_2}^{\alpha_3-2\varepsilon} i\lam y\bar{z}dx-\int_{\alpha_2}^{\alpha_3-2\varepsilon}\left|z\right|^{2}dx=o(\lam^{-2}).
\end{equation*}
Then, using the first estimation of Equation \eqref{est-4.2}, we get the second desired estimation of Equation \eqref{est-4.2}.
\end{proof}
\newpage
%%%%%%%%%%%%%%%%%%%%%%%%%%%%%%%%%%%%%%%%%%%%
%%%%%%%%%%%%%%%%%%%%%% Auxiliary Problem
%%%%%%%%%%%%%%%%%%%%%%%%%%%%%%%%%%%%%%%%%%%%

Now, like as \cite{Rayan2019}, we will construct a new multiplier satisfying some ordinary differential systems. 
\begin{lemma}
{\rm Let $0<\alpha_1<\alpha_2<\alpha_3<\alpha_4<L$ and suppose that  $\varepsilon<\frac{\alpha_3-\alpha_1}{4}$, and $c(x)$ the function defined in Equation \eqref{bc}. Then, for any $\la\in \R$, the solution $\left(\varphi,\psi\right)\in ((H^2(0,L)\cap H^1_0(0,L))^2$ of system 
\begin{equation}\label{aux2}
\left\{
\begin{array}{l}
\lam ^{2}\varphi+a\varphi_{xx}-i\lam\left({\mathds{1}}_{(\alpha_{2},\alpha_{3}-2\varepsilon)}\right)(x)\varphi-i\lam c(x)\psi=u,\hspace{0.5cm}x\in (0,L)\\ \\
\lam ^{2}\psi+\psi_{xx}-i\lam\left({\mathds{1}}_{(\alpha_{2},\alpha_{3}-2\varepsilon)}\right)(x)\psi+i\lam c(x)\varphi=y, \hspace{0.7cm}x\in (0,L)\\ \\
\varphi(0)=\varphi(L)=0,\\ \\
\psi(0)=\psi(L)=0,
\end{array}\right.
\end{equation}
satisfies the following estimation
\begin{equation}\label{bdd}
\|\lam\varphi\|_{L^{2}(0,L)}^{2}+\|\varphi_{x}\|_{L^{2}(0,L)}^{2}+\|\lam\psi\|_{L^{2}(0,L)}^{2}+\|\psi_{x}\|_{L^{2}(0,L)}^{2}\leq M\left(\| u\|_{L^{2}(0,L)}^{2}+\| y\|_{L^{2}(0,L)}^{2}\right).
\end{equation}}
\end{lemma}
\begin{proof}
%the exponential stability of system \eqref{aux-prb} implies that the resolvent of the auxiliary operator $\mathcal{A}_a$ defined by 
Following Theorem \ref{auxiliary-problem-Exp}, the exponential stability of System \eqref{aux-prb}, proved in the Appendix, implies that the resolvent of the auxiliary operator $\mathcal{A}_a$ defined by \eqref{dofao}-\eqref{auxiliary-operator}  is uniformly bounded on the imaginary axis i.e. there exists $M>0$ such that
\begin{equation}\label{expo}
\sup_{\la\in \R}\|\left(i\lam I-\mathcal{A}_{a}\right)^{-1}\|_{\mathcal{L}\left(\mathcal{H}_a\right)}\leq M<+\infty
\end{equation}
where $\mathcal{H}_a=\left(H_0^1(0,L)\times L^2(0,L)\right)^2$.
%%%%%%
\begin{comment}
\begin{equation*}
D(\mathcal{A}_a)=\left((H^2(0,L)\cap H^1_0(0,L))\times H^1_0(0,L)\right)^2,
\end{equation*}
and
\begin{equation*}
\mathcal{A}_a(\varphi,\eta,\psi,\xi)=(\eta,a\varphi_{xx}-\chi_{(\alpha_{2},\alpha_{3}-2\epsilon)}\eta-c(x)\xi,\xi,\psi_{xx}-\chi_{(\alpha_{2},\alpha_{3}-2\epsilon)}\xi+c(x)\eta)^{T}.
\end{equation*}
contains $i\R$ and the resolvent $\left(i\la I-\mathcal{A}_a\right)$ of $\mathcal{A}_a$ is uniformly boundary on the imaginary axis. Consequently, there exists a positive constant $M>0$ such that  
\end{comment}
%%%%
Now, since $(u,y)\in H^1_0(0,L)\times H^1_0(0,L)$, then $(0,-u,0,-y)$ belongs to $\HH_a$, and from \eqref{expo}, there exists $(\varphi,\eta,\psi,\xi)\in D(\mathcal{A}_a)$ such that $\left(i\lam I-\mathcal{A}_{a}\right)(\varphi,\eta,\psi,\xi)=(0,-u,0,-y)^{\top}$ $i.e.$

\begin{eqnarray}
i\la \varphi-\eta&=&0,\label{x1}\\ \noalign{\medskip}
i\la \eta-a\varphi_{xx}+\left({\mathds{1}}_{(\alpha_{2},\alpha_{3}-2\varepsilon)}\right)(x)\eta+c(x)\xi&=&-u,\label{x2} \\ \noalign{\medskip}
i\la \psi-\xi&=&0,\label{x3}\\ \noalign{\medskip}
i\la \xi- \psi_{xx}+\left({\mathds{1}}_{(\alpha_{2},\alpha_{3}-2\varepsilon)}\right)(x)\xi-c(x)\eta &=&-y,\label{x4}
\end{eqnarray}
such that
\begin{equation}\label{x4.1}
\|(\varphi,\eta,\psi,\xi)\|_{\HH_a}\leq M\left(\|u\|_{L^{2}(0,L)}+\|y\|_{L^{2}(0,L)}\right).
\end{equation}
From equations \eqref{x1}-\eqref{x4.1}, we deduce that $(\varphi,\psi)$ is a solution of \eqref{aux2} and we have 
\begin{equation*}
\|\lam\varphi\|_{L^{2}(0,L)}^{2}+\|\varphi_{x}\|_{L^{2}(0,L)}^{2}+\|\lam\psi\|_{L^{2}(0,L)}^{2}+\|\psi_{x}\|_{L^{2}(0,L)}^{2}\leq M\left(\| u\|_{L^{2}(0,L)}^{2}+\| y\|_{L^{2}(0,L)}^{2}\right).
\end{equation*}
Then, we get our desired result.
\end{proof}
\begin{rem}{\rm There was no reference found for the proof of the exponential stability of System \eqref{aux-prb} when the coefficients of the damping and the coupling are both non smooth. For this, we give the proof of the exponential stability of System \eqref{aux-prb} in Theorem \ref{auxiliary-problem-Exp} (see Subsection \ref{EXP-AUX} in Appendix section).}
\end{rem}
%%%%%%%%%%%%%%%%%%%%%%%%%%%%%%%%%%%%%%%%%%%%
%%%%%%%%%%%%%%%%%%%%%% Auxiliary Problem
%%%%%%%%%%%%%%%%%%%%%%%%%%%%%%%%%%%%%%%%%%%%
\begin{lemma}\label{final1}
{\rm Let $\varepsilon<\frac{\alpha_3-\alpha_1}{4}$. Then, the solution $(u,v,y,z)\in D(\mathcal{A})$ of \eqref{eq-4.4}-\eqref{eq-4.7} satisfies the following asymptotic behavior estimation
\begin{equation}\label{lamu}
\int_0^L \left|\lam u\right|^2dx=o(1),
\end{equation}
and 
\begin{equation}\label{lamy}
\int_0^L \left|\lam y\right|^2dx=o(1).
\end{equation}
}
\end{lemma}

\begin{proof}
The proof of this Lemma is divided into two steps.\\
\textbf{Step 1.}\\
Multiplying equation \eqref{eq-4.13} by $\lam^{2}\bar{\varphi}$, integrate over $(0,L)$, and using Equation \eqref{bdd} and the facts that $u$ is uniformly bounded in $L^2(0,L)$ and $\|F\|_{\mathcal{H}}=\|(f_1,g_1,f_2,g_2)\|_{\mathcal{H}}=o(1)$, we get
\begin{equation}\label{1x7}
\int_0^L \left(\lam^2\bar{\varphi}+a\bar{\varphi}_{xx}\right)\lam^{2}udx-\int_0^L\lam^{2}b(x)v_{x}\bar{\varphi}_{x}dx-\int_0^L i\lam c(x)y\bar{\varphi}dx=o(\lam^{-1}).
\end{equation}
Using Equations \eqref{eq-4.9} and \eqref{bdd}, we get 
\begin{equation}\label{2x7}
\int_0^L\lam^{2}b(x)v_{x}\bar{\varphi}_{x}dx=o(1).
\end{equation}
Combining Equations \eqref{1x7} and \eqref{2x7}, we obtain
\begin{equation}\label{x7}
\int_0^L \left(\lam^2\bar{\varphi}+a\bar{\varphi}_{xx}\right)\lam^{2}udx-\int_0^L i\lam^3 c(x)y\bar{\varphi}dx=o(1).
\end{equation}
From System \eqref{aux2}, we have
\begin{equation}\label{bar1}
\lam^2\bar{\varphi}+a\bar{\varphi}_{xx}=-i\lam\left({\mathds{1}}_{(\alpha_{2},\alpha_{3}-2\varepsilon)}\right)(x)\bar{\varphi}-i\lam c(x)\bar{\psi}+\bar{u}.
\end{equation}
Substituting \eqref{bar1} in \eqref{x7}, we get
\begin{equation}\label{1eq11}
\int_0^L \left|\lam u\right|^2dx-\int_0^L i\lam^3 \left({\mathds{1}}_{(\alpha_{2},\alpha_{3}-2\varepsilon)}\right)(x)u\bar{\varphi}dx-\int_0^L i\lam^3 c(x)\bar{\psi}udx-\int_0^L i\lam^3 c(x)y\bar{\varphi}dx=o(1).
\end{equation}
Using Remark \ref{remark-4.1}, Lemma \ref{lem.2} and Equation \eqref{bdd}, we obtain
\begin{equation}\label{2eq11}
\int_0^L i\lam^3 \left({\mathds{1}}_{(\alpha_{2},\alpha_{3}-2\varepsilon)}\right)(x)u\bar{\varphi}dx=o(1).
\end{equation}
Inserting Equation \eqref{2eq11} in Equation  \eqref{1eq11}, we get
\begin{equation}\label{eq11}
\int_0^L \left|\lam u\right|^2dx-\int_0^L i\lam^3 c(x)\bar{\psi}udx-\int_0^L i\lam^3 c(x)y\bar{\varphi}dx=o(1).
\end{equation}
%%%%%%%%%%%%%%%%%%%%%%%%%%%%%%%%%%%%%%%%%%%%%%%%%%%%%%%%%%%%%%%%%%%%%%%%%%%%%%%%%%%%%%%%%%%%%%%%%%
\textbf{Step 2.}\\
Multiplying equation \eqref{eq-4.14} by $\lam^{2}\bar{\psi}$, integrate over $(0,L)$, and using Equation \eqref{bdd} and the facts that $y$ is uniformly bounded in $L^2(0,L)$ and $\|F\|_{\mathcal{H}}=\|(f_1,g_1,f_2,g_2)\|_{\mathcal{H}}=o(1)$, we get
\begin{equation}\label{x8}
\int_0^L\left(\la^2\bar{\psi}+\bar{\psi}_{xx}\right)\la^2ydx+\int_0^L i\lam c(x)u\bar{\psi}dx=o(\lam^{-1}).
\end{equation}
From System \eqref{aux2}, we have
\begin{equation}\label{bar2}
\lam^2\bar{\psi}+a\bar{\psi}_{xx}=-i\left({\mathds{1}}_{(\alpha_{2},\alpha_{3}-2\varepsilon)}\right)(x)\bar{\psi}+i\lam c(x)\bar{\varphi}+\bar{y}.
\end{equation}
Substituting \eqref{bar2} in \eqref{x8}, we get
\begin{equation}\label{1eq12}
\int_0^L \left|\lam y\right|^2dx-\int_0^L i\lam^3 \left({\mathds{1}}_{(\alpha_{2},\alpha_{3}-2\varepsilon)}\right)(x)y\bar{\psi}dx+\int_0^L i\lam^3 c(x)\bar{\varphi}ydx+\int_0^L i\lam^3 c(x)u\bar{\psi}dx=o(\lam^{-1}).
\end{equation}
Using Remark \ref{remark-4.1}, Lemma \ref{lem-4.2} and Equation \eqref{bdd}, we obtain
\begin{equation}\label{2eq12}
\int_0^L i\lam^3 \left({\mathds{1}}_{(\alpha_{2},\alpha_{3}-2\varepsilon)}\right)(x)y\bar{\psi}dx=o(1).
\end{equation}
Inserting Equation \eqref{2eq12} in Equation \eqref{1eq12}, we get 
%Using Lemma  and equation \eqref{bdd}, we get
\begin{equation}\label{eq12}
\int_0^L \left|\lam y\right|^2dx+\int_0^L i\lam^3 c(x)\bar{\varphi}ydx+\int_0^L i\lam^3 c(x)u\bar{\psi}dx=o(1).
\end{equation}
Finally, summing up equations \eqref{eq11} and \eqref{eq12} we get
\begin{equation*}
\int_0^L \left|\lam u\right|^2dx=o(1)\quad\mbox{and}\quad\int_0^L \left|\lam y\right|^2dx=o(1).
\end{equation*}
Hence,
\begin{equation}
\int_0^L \left|v\right|^2dx=o(1) \quad \text{and} \quad \int_0^L \left|z\right|^2dx=o(1).
\end{equation}
Then, the proof has been completed. 
\end{proof}

\begin{lemma}\label{final2}
{\rm The solution $(u,v,y,z)\in D(\mathcal{A})$ of the \eqref{eq-4.4}-\eqref{eq-4.7}  satisfies the following asymptotic behavior estimations
\begin{equation}
\int_0^L \left| u_x\right|^2dx=o(1)\quad\mbox{and}\quad\int_0^L \left| y_x\right|^2dx=o(1).
\end{equation}}
\end{lemma}

\begin{proof}
Multiplying \eqref{eq-4.13} by $\bar{u}$ integrate over $(0,L)$, using the fact that $\|F\|_{\mathcal{H}}=\|(f_1,g_1,f_2,g_2)\|_{\mathcal{H}}=o(1)$ and $u$ is uniformly bounded in $L^2(0,L)$, we get
\begin{equation}
\int_0^L \left|\lam u\right|^2dx-\int_0^L a\left| u_x\right|^2dx-\int_0^L  b(x)v_x \bar{u}_{x}dx-\int_0^L i\lam c(x)y\bar{u}dx=o(\lam^{-2}).
\end{equation}
Using equations \eqref{eq-4.9} and \eqref{lamu}, we get
\begin{equation*}
\int_0^L \left| u_x\right|^2dx=o(1).
\end{equation*}
Similarly, multiply \eqref{eq-4.14} by $\bar{y}$ and integrate, we get
\begin{equation*}
\int_0^L \left| y_x\right|^2dx=o(1).
\end{equation*}
The proof has been completed.
\end{proof}

\noindent \textbf{Proof of Theorem \ref{Theorem-4.1}.}. 
Consequently, from the results of Lemmas \ref{final1} and \ref{final2}, we obtain  
\begin{equation*}
\int_0^L\left(|v|^2+|z|^2+a\,|u_x|^2+ |y_x|^2\right)dx
=o\left(1\right).
\end{equation*}
Hence $\|U\|_{\HH}=o(1)$, which contradicts \eqref{eq-4.2}. Consequently, condition ${\rm (H2)}$ holds. This implies, from Theorem \ref{bt}, the energy decay estimation \eqref{EnergyGeneral}. The proof is thus complete.

%%%%%%%%%%%%%%%%%%%%%%%%%%%%%%%%%%%%%%%%%%%%%%%%%%%%%%%%%%%%%%%%%%%%%%%%%%%%%%%%%%%%%%%%%%%%%%%%%%%%%%%%%%%%%%%%%%%%%%%%%%%%%%%%%%%%%%%%%%%%%%%%%%%%%%%%%%%%%%%%%%%%%%%%%%%%%%%%%%%%%%
\section{Appendix}\label{Appendix}
%%%%%%%%%%%%%%%%%%%%%%%%%%%%%%%%%%%%%%%%%%%%%%%%%%%%%%%%%%%%%%%%%%%%%%%%%%%%%%%%%%%%%%%%%%%%%%%%%%%%%%%%%%%%%%%%
\subsection{Exponential stability of locally coupled wave equations with non-smooth coefficients}\label{EXP-AUX}
\noindent We consider the following auxiliary problem, 
\begin{equation}\label{aux-prb}
\left\{\begin{array}{l}
\varphi_{tt}-a\varphi_{xx}+\left({\mathds{1}}_{(\alpha_{2},\alpha_{3}-2\varepsilon)}\right)(x) \varphi_{t}+c(x)\psi_{t}=0,  \hspace{1cm}(x,t)\in (0,L)\times\R^+,\\ \\
\psi_{tt}-\psi_{xx}+\left({\mathds{1}}_{(\alpha_{2},\alpha_{3}-2\varepsilon)}\right)(x) \psi_{t}-c(x)\varphi_{t}=0, \hspace{1.2cm}(x,t)\in (0,L)\times\R^+,\\  \\
\varphi(0,t)=\varphi(L,t)=0,\hspace{5cm} t>0,\\ \\
\psi(0,t)=\psi(L,t)=0,\hspace{5cm} t>0.
\end{array}\right.
\end{equation}
Since, we have a system of coupled wave equations with two interior damping acting on a part of the interval $(0,L)$, then system \eqref{aux-prb} is exponentially stable in the associated energy space $\mathcal{H}_{a}=\left(H_0^1(0,L)\times L^2(0,L)\right)^2$. In this section, our aim is to show that the auxiliary problem \eqref{aux-prb} is uniformly stable. The energy of System  \eqref{aux-prb} is given by 
\begin{equation*}
E_a(t)=\frac{1}{2}\left(\int_0^L\abs{\varphi_t}^2+a\abs{\varphi_x}^2+\abs{\psi_t}^2+\abs{\psi_x}^2dx\right)
\end{equation*}
and by a straightforward calculation, we have 
\begin{equation*}
\frac{d}{dt}E_a(t)=-\int_0^L\left({\mathds{1}}_{(\alpha_{2},\alpha_{3}-2\varepsilon)}\right)(x)\abs{\varphi_t}^2dx-\int_0^L\left({\mathds{1}}_{(\alpha_{2},\alpha_{3}-2\varepsilon)}\right)(x)\abs{\psi_t}^2dx\leq 0.
\end{equation*}
Thus, System \eqref{aux-prb} is dissipative in the sense that its energy is a non-increasing function with respect to the time variable $t$. The auxiliary energy Hilbert space of Problem \eqref{aux-prb} is given by 
\begin{equation*}
\mathcal{H}_a=\left(H_0^1(0,L)\times L^2(0,L)\right)^2.
\end{equation*}
We denote by $\eta=\varphi_t$ and $\xi=\psi_t$.
The auxiliary energy space $\mathcal{H}_a$ is endowed with the following norm
$$
\|\Phi\|_{\mathcal{H}_a}^2=\|\eta\|^2+a\|\varphi_x\|^2+\|\xi\|^2+\|\psi_x\|^2,
$$
where $\|\cdot\|$ denotes the norm of $L^2(0,L)$. We define the unbounded linear operator $\mathcal{A}_a$ by 
\begin{equation}\label{dofao}
D(\mathcal{A}_a)=\left((H^2(0,L)\cap H^1_0(0,L))\times H^1_0(0,L)\right)^2,
\end{equation}
and
\begin{equation}\label{auxiliary-operator}
\mathcal{A}_a(\varphi,\eta,\psi,\xi)=(\eta,a\varphi_{xx}-\left({\mathds{1}}_{(\alpha_{2},\alpha_{3}-2\varepsilon)}\right)(x)\eta-c(x)\xi,\xi,\psi_{xx}-\left({\mathds{1}}_{(\alpha_{2},\alpha_{3}-2\varepsilon)}\right)(x)\xi+c(x)\eta)^{\top}.
\end{equation}
If $\Phi=\left(\varphi,\psi,\eta,\xi\right)$ is the state of System \eqref{aux-prb}, then this system is tranformed into a first order evolution equation on the auxiliary Hilbert space $\mathcal{H}_a$ given by
\begin{equation*}
\Phi_t=\mathcal{A}_a\Phi,\quad \Phi(0)=\Phi_0,
\end{equation*}
where $\Phi_0=\left(\varphi_0,\eta_0,\psi_0,\xi_0\right)$. It is easy to see that $\mathcal{A}_a$ is m-dissipative and generates a $C_0-$semigroup of contractions $\left(e^{t\mathcal{A}_a}\right)_{t\geq 0}$.  
%%%%%%%%%%%%%
\begin{theoreme}\label{strong stability-aux}
The $C_0-$semigroup of contractions $(e^{t\mathcal{A}_{a}})_{t\geq 0}$ is strongly stable on $\HH_{a}$, i.e. for all $U_0\in \HH_{a}$,$\displaystyle{\lim_{t\to +\infty}\|e^{t\AA_{a}}U_0\|_{\HH_{a}}=0}$ .
\end{theoreme}
\begin{proof}
Following  Arendt and Batty Theorem in  \cite{Arendt01}, we have to prove the following two conditions 
 %%%%%%%%%%%%%%%%
 \begin{enumerate}
 \item[1.]  $\mathcal{A}$ has no pure imaginary eigenvalues,
  \item[2.]  $\sigma\left(\mathcal{A}\right)\cap i\mathbb{R}$ is countable.
 \end{enumerate}
In order to prove these two conditions we proceed with the same argument of subsection \ref{Section-2.2} and we reach the desired result.  
\end{proof}

\noindent Now, we present the main result of this section
\begin{theoreme}\label{auxiliary-problem-Exp}
The $C_0-$semigroup of contractions $\left(e^{t\mathcal{A}_a}\right)_{t\geq 0}$ is exponentially  stable, i.e. there exists constants $M\geq 1$ and $\tau>0$ independent of $\Phi_0$ such that 
$$
\left\|e^{t\mathcal{A}_a}\Phi_0\right\|_{\mathcal{H}_a}\leq Me^{-\tau t}\|\Phi_0\|_{\mathcal{H}_a},\qquad t\geq 0.
$$
\end{theoreme}
\noindent According to Huang \cite{Huang01} and Pruss \cite{pruss01}, we have to check if the following conditions hold:
\begin{equation}\label{H3}\tag{${\rm H3}$}
 i\mathbb{R}\subseteq \rho\left(\mathcal{A}_a\right)
\end{equation}
and 
\begin{equation}\label{H4}\tag{${\rm H4}$}
\displaystyle{\sup_{\la\in \R}\|\left(i\la I-\mathcal{A}_a\right)^{-1}\|_{\mathcal{L}\left(\mathcal{H}_a\right)}=O(1).}
\end{equation}
%%%%%%%%
\begin{comment}
\begin{theoreme}\label{HP}
Let $\left(S(t)\right)_{t\geq 0}$ be a $C_0-$semigroup of contractions on $H$ and $A$ be its infinitesimal generator. Then, $\left(S(t)\right)_{t\geq 0}$ is exponentially stable if and only if 
$$
i\mathbb{R}\subseteq \rho(A)\quad \text{and}\quad \displaystyle{\lim_{\la\in \R}\sup_{\abs{\la}\to+\infty}\|(i\la I-A)^{-1}\|_{\mathcal{L}\left(\mathcal{H}\right)}}<+\infty.
$$
\end{theoreme}
\end{comment}
%%%%%%%%%%%%%
\noindent By using the same argument of Lemma \ref{ker}, the operator $\mathcal{A}_a$ has no pure imaginary eigenvalues. Then, condition \eqref{H3} holds. We will prove condition \eqref{H4} using an argument of contradiction. Indeed, suppose there exists 
\begin{equation*}
\left\{\left(\la_n,\Phi_n=\left(\varphi_n,\eta_n,\psi_n,\xi_n\right)\right)\right\}_{n\geq 1}\subset \R_+^{\ast}\times D\left(\mathcal{A}_a\right)
\end{equation*}
such that 
\begin{equation}\label{Phin}
\la_n\to +\infty\quad \text{and}\quad \|\Phi_n\|_{\mathcal{H}_a}=1
\end{equation}
and there exists a sequence $F_n=\left(f_{1,n},f_{2,n},f_{3,n},f_{4,n}\right)\in \mathcal{H}_a$ such that 
\begin{equation}\label{exp-app1}
\left(i\la_nI-\mathcal{A}_a\right)\Phi_n=F_n\to 0\quad \text{in}\quad \mathcal{H}_a.
\end{equation}
Detailing \eqref{exp-app1}, we get the following system  
\begin{eqnarray}
i\la \varphi_n-\eta_n &=&f_{1,n}\quad \text{in}\quad H_0^1(0,L),\label{exp-app2}\\
i\la \eta_n-a\left(\varphi_n\right)_{xx}+\left({\mathds{1}}_{(\alpha_{2},\alpha_{3}-2\varepsilon)}\right)(x)\eta_n+c(x)\xi_n &=&f_{2,n}\quad \text{in}\quad L^2(0,L),\label{exp-app3}\\
i\la \psi_n-\xi_n &=&f_{3,n}\quad \text{in}\quad H_0^1(0,L),\label{exp-app4}\\
i\la \xi_n-\left(\psi_n\right)_{xx}+\left({\mathds{1}}_{(\alpha_{2},\alpha_{3}-2\varepsilon)}\right)(x)\xi_n-c(x)\eta_n &=&f_{4,n}\quad\text{in}\quad L^2(0,L)\label{exp-app5}. 
\end{eqnarray}
In what follows, we will check the condition ${\rm (H4)}$ by finding a contradiction with \eqref{Phin} such as $\|\Phi_n\|_{\mathcal{H}_a} = o(1)$. For clarity,
we divide the proof into several lemmas. From now on, for simplicity, we drop the index n.
%%%%%%%%%%%%%%%%%%%%%
% First Estimation
%%%%%%%%%%%%%%%%%%%%%
\begin{lemma}\label{First-Estimation-aux}
The solution $\left(\varphi,\eta,\psi,\xi\right)\in D\left(\mathcal{A}_a\right)$ of Equations \eqref{exp-app2}-\eqref{exp-app5} satisfies the following asymptotic behavior estimation
\begin{equation*}
\int_{\alpha_2}^{\alpha_3-2\varepsilon}\abs{\eta}^2dx=o(1)\quad \text{and}\quad \int_{\alpha_2}^{\alpha_3-2\varepsilon}\abs{\xi}^2dx=o(1).
\end{equation*}
\end{lemma}
\begin{proof}
Taking the inner product of \eqref{exp-app1} with $\Phi$ in $\mathcal{H}_a$, then using the fact that $\Phi$ is uniformly bounded in $\mathcal{H}_a$, we get 
$$
\int_{\alpha_2}^{\alpha_3-2\varepsilon}\abs{\eta}^2dx+\int_{\alpha_2}^{\alpha_3-2\varepsilon}\abs{\xi}^2dx=-\Re\left<\mathcal{A}_a\Phi,\Phi\right>_{\mathcal{H}_a}=\Re\left<\left(i\la I-\mathcal{A}_a\right)\Phi,\Phi\right>=o(1).
$$
Thus, the proof of the Lemma is complete.
\end{proof}

$\newline$
%%%%%%%%%%%%%%%%%%%%%%%%%%%%%%%%%%%%%%%%%%
%%%%%%%%%%%%%%%%%%%%%%%%%%%%%%%%%%%%%%%%%%
\noindent Substituting $\eta$ and $\xi$ by $i\la \varphi-f_1$ and $i\la \psi-f_3$ respectively in \eqref{exp-app3} and \eqref{exp-app5}, we get the following system 
\begin{eqnarray}
\la^2\varphi+a\varphi_{xx}-i\la \left({\mathds{1}}_{(\alpha_{2},\alpha_{3}-2\varepsilon)}\right)(x)\varphi-i\la c(x)\psi &=&-i\la f_1+\left({\mathds{1}}_{(\alpha_{2},\alpha_{3}-2\varepsilon)}\right)(x)f_1-f_2-c(x)f_3,\label{exp-app6}\\
\la^2\psi+\psi_{xx}-i\la \left({\mathds{1}}_{(\alpha_{2},\alpha_{3}-2\varepsilon)}\right)(x)\psi+i\la c(x)\varphi &=&c(x)f_1-i\la f_3-\left({\mathds{1}}_{(\alpha_{2},\alpha_{3}-2\varepsilon)}\right)(x)f_3-f_4.\label{exp-app7}
\end{eqnarray}
%%%%%%%%%%%%%%%%%%%%%%%%%%%%%%%%%%%%%%%%%%
%%%%%%%%%%%%%%%%%%%%%%%%%%%%%%%%%%%%%%%%%%
\begin{lemma}\label{Second-Estimation-aux}
Let $0<\delta<\frac{\alpha_3-2\varepsilon-\alpha_2}{2}$. The solution $\left(\varphi,\eta,\psi,\xi\right)\in D(\mathcal{A}_a)$ of Equations \eqref{exp-app1}-\eqref{exp-app4} satisfies the following asymptotic behavior estimation
\begin{equation*}
\int_{\alpha_2+\delta}^{\alpha_3-2\varepsilon-\delta}\abs{\varphi_x}^2dx=o(1)\quad \text{and}\quad \int_{\alpha_2+\delta}^{\alpha_3-2\varepsilon-\delta}\abs{\psi_x}^2dx=o(1).
\end{equation*}
\end{lemma}
\begin{proof}
First, we define the first cut-off function $\theta$ in $C^1(0,L)$ by , defined by
\begin{equation}\label{theta}
0\leq \theta\leq 1,\quad \theta=1\ \ \text{on}\ \ (\alpha_2+\delta,\alpha_3-2\varepsilon-\delta)\quad \text{and}\quad \theta=0\ \ \text{on}\ \  (0,\alpha_2)\cup(\alpha_3-2\varepsilon,L).
\end{equation}
Multiplying Equations \eqref{exp-app6} and \eqref{exp-app7} by $\theta \bar{\varphi}$ and $\theta \bar{\psi}$ respectively, integrate over $(0,L)$ and using the fact that $\la \varphi$ and $\la\psi$ are uniformly bounded in $L^2(0,L)$ and $\|F\|\to 0$ in $\mathcal{H}_a$ and taking the real part, we get 
\begin{equation}\label{exp-app8}
\int_0^L\theta \abs{\la \varphi}^2dx-a\int_0^L\theta \abs{\varphi_x}^2dx-a\int_0^L\theta'\bar{\varphi}\varphi_xdx-\Re\left(i\la c_0\int_{\alpha_2}^{\alpha_3-2\varepsilon}\theta\psi\bar{\varphi}dx\right)=o(1)
\end{equation}
and 
\begin{equation}\label{exp-app9}
\int_0^L\theta \abs{\la \psi}^2dx-\int_0^L\theta \abs{\psi_x}^2dx-\int_0^L\theta'\bar{\psi}\psi_xdx+\Re\left(i\la c_0\int_{\alpha_2}^{\alpha_3-2\varepsilon}\theta\varphi\bar{\psi}dx\right)=o(1).
\end{equation}
%%%
\begin{comment}
It follows that, from the above equations 
\begin{equation}\label{exp-app8}
a\int_0^L\theta\abs{\varphi_x}^2dx \leq \int_{\alpha_2}^{\alpha_2+\varepsilon}\abs{\la\varphi}^2dx+ac_{\theta'}\int_0^L\abs{\varphi}\abs{\varphi_x}dx+c_0\abs{\la}\left(\int_{\alpha_2}^{\alpha_3-2\varepsilon}\abs{\psi}^2dx\right)^{\frac{1}{2}}\left(\int_{\alpha_2}^{\alpha_3-2\varepsilon}\abs{\varphi}^2dx\right)^{\frac{1}{2}}+o(1)
\end{equation}
and 
\begin{equation}\label{exp-app9}
\int_0^L\theta\abs{\psi_x}^2dx \leq \int_{\alpha_2}^{\alpha_2+\varepsilon}\abs{\la\psi}^2dx+ac_{\theta'}\int_0^L\abs{\psi}\abs{\psi_x}dx+c_0\abs{\la}\left(\int_{\alpha_2}^{\alpha_3-2\varepsilon}\abs{\psi}^2dx\right)^{\frac{1}{2}}\left(\int_{\alpha_2}^{\alpha_3-2\varepsilon}\abs{\varphi}^2dx\right)^{\frac{1}{2}}.
\end{equation}
\end{comment}
Using the fact that $\la \varphi$ and $\la \psi$ are uniformly bounded in $L^2(0,L)$, in particular in $L^2\left(\alpha_2,\alpha_3-2\varepsilon\right)$, and the definition of $\theta$, we get 
\begin{equation}\label{exp-app10}
\Re\left(i\la c_0\int_{\alpha_2}^{\alpha_3-2\varepsilon}\theta\psi\bar{\varphi}dx\right)=o(1)\quad \text{and}\quad \Re\left(i\la c_0\int_{\alpha_2}^{\alpha_3-2\varepsilon}\theta\varphi\bar{\psi}dx\right)=o(1).
\end{equation}
On the other hand, using the fact that $\la \varphi,\ \la \psi,\ \varphi_x$ and $\psi_x$ are uniformly bounded in $L^2(0,L)$, we get 
\begin{equation}\label{newexp-app10}
a\int_0^L\theta'\bar{\varphi}\varphi_xdx=o(1)\quad \text{and}\quad \int_0^L\theta'\bar{\psi}\psi_xdx=o(1).
\end{equation}
Furthermore, using Lemma \ref{First-Estimation-aux}, Equations \eqref{exp-app2}, \eqref{exp-app4} and the definition of the function $\theta$ in Equation \eqref{theta}, we get 
\begin{equation}\label{newexp-app11}
\int_0^L\theta \abs{\la \varphi}^2dx=o(1)\quad \text{and}\quad \int_0^L\theta \abs{\la \psi}^2dx=o(1).
\end{equation}
Inserting Equations \eqref{exp-app10}-\eqref{newexp-app11} in Equations \eqref{exp-app8} and \eqref{exp-app9}, we get the desired results. Thus, the proof of this Lemma is complete .
\end{proof}

$\newline$
From Lemma \ref{First-Estimation-aux} and Lemma \ref{Second-Estimation-aux}, we get $\|\Phi\|_{\mathcal{H}_a}=o(1)$ on $(\alpha_2+\delta,\alpha_3-2\varepsilon-\delta)$. In order to complete the proof, we need to show that $\|\Phi\|_{\mathcal{H}_a}$ on $(\alpha_2+\delta,\alpha_3-2\varepsilon-\delta)^c$.

\begin{lemma}\label{Third-Estimation-aux}
Let $h\in C^1(0,L)$. The solution $\left(\varphi,\eta,\psi,\xi\right)\in D(\mathcal{A}_a)$ of Equations \eqref{exp-app2}-\eqref{exp-app5} satisfies the following asymptotic behavior estimation
\begin{equation}\label{exp-app12}
\begin{split}
\int_0^Lh'\left(\abs{\eta}^2+a\abs{\varphi_x}^2+\abs{\xi}^2+\abs{\psi_x}^2\right)dx-\Re\left(\left[ah\abs{\varphi_x}^2\right]_0^L\right)-\Re\left(\left[h\abs{\psi_x}^2\right]_0^L\right)+2\Re\left(\int_0^Lc(x)h\xi\bar{\varphi}_xdx\right)\\-2\Re\left(\int_0^Lc(x)h\eta \bar{\psi}_xdx\right)
=2\int_0^L h\bar{\varphi_x}f_2 dx+2\int_0^L h\eta(\bar{f_1})_x dx+2\int_0^L h\bar{\psi_x}f_4 dx+2\int_0^L h\xi(\bar{f_3})_x dx.
\end{split}
\end{equation}
\end{lemma}
\begin{proof}
Multiplying Equations \eqref{exp-app3} and \eqref{exp-app5} by $2h\bar{\varphi}_x$ and $2h\bar{\psi}_x$ respectively, integrate over $(0,L)$ and using the fact that $\varphi_x$, $\psi_x$ are uniformly bounded in $L^2(0,L)$ and $\|F\|_{\mathcal{H}_a}\to 0$ and Lemma \ref{First-Estimation-aux}, we get 
\begin{eqnarray}
2\int_0^Li\la h \eta \bar{\varphi}_xdx-2a\int_0^Lh\varphi_{xx}\bar{\varphi_x}dx+2\int_0^Lc(x)h\xi\bar{\varphi}_xdx&=&2\int_0^L h\bar{\varphi_x}f_2 dx\label{exp-app13}\\
2\int_0^Li\la h\xi\bar{\psi}_xdx-2\int_0^Lh\psi_{xx}\bar{\psi}_xdx-2\int_0^Lc(x)h\eta\bar{\psi}_xdx&=&2\int_0^L h\bar{\psi_x}f_4 dx.\label{exp-app14}
\end{eqnarray}
From Equations \eqref{exp-app2} and \eqref{exp-app4}, we have 
$$-i\la \bar{\varphi}_x=\bar{\eta}_x+\left(\bar{f_1}\right)_x\quad \text{and}\quad -i\la \bar{\psi}_x=\bar{\xi}_x+\left(\bar{f_3}\right)_x.$$
Inserting the above equations in Equations \eqref{exp-app13} and \eqref{exp-app14} and by taking the real part, we obtain 
\begin{eqnarray}
-\int_0^Lh\abs{\eta}_x^2dx-a\int_0^Lh\abs{\varphi_x}^2_xdx+2\Re\left(\int_0^Lc(x)h\xi\bar{\varphi}_xdx\right)&=&2\int_0^L h\bar{\varphi_x}f_2 dx+2\int_0^L h\eta(\bar{f_1})_x dx,\label{exp-app15}\\
-\int_0^Lh\abs{\xi}_x^2dx-\int_0^Lh\abs{\psi_x}_x^2dx-2\Re\left(\int_0^Lc(x)\eta h\bar{\psi}_xdx\right)&=&2\int_0^L h\bar{\psi_x}f_4 dx+2\int_0^L h\xi(\bar{f_3})_x dx.\label{exp-app16}
\end{eqnarray}
Using by parts integration in Equations \eqref{exp-app15} and \eqref{exp-app16}, we get the desired results. 
\end{proof}
%%%%%%%%%%%%%%%%%%%%%%%%%%%%
%%%%%%%%%%%%%%%%%%%%%%%%%%%%
\begin{lemma}\label{Fourth-Estimation-aux}
Let $0<\delta<\frac{\alpha_3-2\varepsilon-\alpha_2}{2}$. The solution $\left(\varphi,\eta,\psi,\xi\right)\in D(\mathcal{A}_a)$ of Equations \eqref{exp-app2}-\eqref{exp-app5} satisfies the following asymptotic behavior estimation
$$
\int_0^{\alpha_2+\delta}\left(\abs{\eta}^2+a\abs{\varphi_x}^2+\abs{\xi}^2+\abs{\psi_x}^2\right)dx=o(1).
$$
\end{lemma}
\begin{proof}
Define the cut-off function $\tilde{\theta}$ in $C^1([0,L])$ by 
\begin{equation}\label{thetatilde}
0\leq \tilde{\theta}\leq 1,\quad \tilde{\theta} =1\ \ \text{on}\ \ (0,\alpha_2+\delta),\quad \tilde{\theta}=0\ \ \text{on}\ \ (\alpha_3-2\varepsilon-\delta,L).
\end{equation}
Take  $h=x\tilde{\theta}(x)$ in Equation \eqref{exp-app12}, we get 
\begin{equation}\label{exp-app17}
\begin{array}{l}
\displaystyle{\int_0^Lh'\left(\abs{\eta}^2+a\abs{\varphi_x}^2+\abs{\xi}^2+\abs{\psi_x}^2\right)dx+2c_0\Re\left(\int_{\alpha_2}^{\alpha_3-2\varepsilon-\delta}x\tilde{\theta}\xi\bar{\varphi}_xdx\right)}\\
\qquad\displaystyle{-2c_0\Re\left(\int_{\alpha_2}^{\alpha_3-2\varepsilon-\delta}x\tilde{\theta}\eta \bar{\psi}_xdx\right)
=o(1).}
\end{array}
\end{equation}
Using Lemma \eqref{First-Estimation-aux} and $\varphi_x$ and $\psi_x$ are uniformly bounded in $L^2(0,L)$ and in particular in $L^2(\alpha_2,\alpha_3-2\varepsilon-\delta)$, we get 
$$
2c_0\Re\left(\int_{\alpha_2}^{\alpha_3-2\varepsilon-\delta}x\tilde{\theta}\xi\bar{\varphi}_xdx\right)=o(1)\quad \text{and}\quad 2c_0\Re\left(\int_{\alpha_2}^{\alpha_3-2\varepsilon-\delta}x\tilde{\theta}\eta \bar{\psi}_xdx\right)=o(1).
$$
Inserting the above equations in Equation \eqref{exp-app17}, and using Lemmas \eqref{First-Estimation-aux}-\eqref{Second-Estimation-aux} and  the definition the function $\tilde{\theta}$, we get the desired result.
\end{proof}

$\newline$
%%%%%%%%%%%%%%%%%%%
From the preceded results of Lemmas \ref{First-Estimation-aux},\ \ref{Second-Estimation-aux} and  \ref{Fourth-Estimation-aux} , we deduce that
$$
\|\Phi\|_{\mathcal{H}_a}=o(1)\quad \text{on}\quad (\alpha_2+\delta,\alpha_3-2\varepsilon-\delta).
$$ 
\noindent Now, our goal is to prove that $\|\Phi\|_{\mathcal{H}_a}=o(1)$ on $\left(\alpha_{3}-2\varepsilon-\delta,L\right)$. For this aims, let $g\in C^1\left([\alpha_{3}-2\varepsilon-\delta,\alpha_4]\right)$ such that 
$$
g(\alpha_4)=-g(\alpha_3-2\varepsilon-\delta)=1,\quad \displaystyle{\max_{x\in[\alpha_3-3\varepsilon,\alpha_4]}\abs{g(x)}=c_g}\quad \text{and}\quad \max_{x\in[\alpha_3-3\varepsilon,\alpha_4]}\abs{g'(x)}=c_{g'}
$$
where $c_g$ and $c_{g'}$ are strictly positive constant numbers.
\begin{rem}
It is easy to see the existence of $g(x)$. For example, we can take $g(x)=\displaystyle{\cos\left(\frac{(\alpha_4-x)\pi}{\alpha_4-\alpha_{3}+2\varepsilon+\delta}\right)}$ to  get $g(\alpha_4)=-g(\alpha_3-2\varepsilon-\delta)=1$, $g\in C^1\left([\alpha_{3}-2\varepsilon-\delta,4]\right)$, $\abs{g(x)}\leq 1$ and $\abs{g'(x)}\leq \frac{\pi}{\alpha_4-\alpha_3+2\varepsilon+\delta}$.
\end{rem}
%%%%%%%%%%%%%%%%%%%%%%%%%%%%%%%%%%%%%%
%%%%%%%%%%%%%%%%%%%%%%%%%%%%%%%%%%%%%%
\begin{lemma}\label{5-Estimation-aux}
Let $0<\delta<\frac{\alpha_3-2\varepsilon-\alpha_2}{2}$. The solution $\left(\varphi,\eta,\psi,\xi\right)\in D(\mathcal{A}_a)$ of Equations \eqref{exp-app1}-\eqref{exp-app4} satisfies the following asymptotic behavior estimation
\begin{equation*}
\abs{\eta(\alpha_4)}^2=O(1),\ \ \abs{\eta(\alpha_3-2\varepsilon-\delta)}^2=O(1),\ \ \abs{\xi(\alpha_4)}^2=O(1)\ \ \text{and}\ \ \abs{\xi(\alpha_3-2\varepsilon-\delta)}^2=O(1).
\end{equation*}
\end{lemma}
\begin{proof}
From \eqref{exp-app3} and \eqref{exp-app5}, we have 
\begin{equation}\label{exp-app18}
i\la\varphi_x-\eta_x=\left(f_1\right)_x\quad \text{and}\quad i\la\psi_x-\xi_x=\left(f_3\right)_x.
\end{equation}
Multiplying the first equation and the second equation of \eqref{exp-app18} respectively  by $2g(x)\bar{\eta}$ and $2g(x)\bar{\xi}$, integrate over $(\alpha_3-2\varepsilon-\delta,\alpha_4)$ and using the fact that $\|F\|_{\mathcal{H}_a}\to 0$ and $\eta$ and $\xi$ are uniformly bounded in $L^2(0,L)$ in particular in $L^2(\alpha_3-2\varepsilon-\delta,\alpha_4)$, we get
\begin{eqnarray}
\Re\left(2i\la\int_{\alpha_3-2\varepsilon-\delta}^{\alpha_4}g\varphi_x\bar{\eta}dx\right)-\int_{\alpha_3-2\varepsilon-\delta}^{\alpha_4}g(x)\left(\abs{\eta}^2\right)_xdx&=&o(1),\label{exp-app19}\\
\Re\left(2i\la\int_{\alpha_3-2\varepsilon-\delta}^{\alpha_4}g\psi_x\bar{\xi}dx\right)-\int_{\alpha_3-2\varepsilon-\delta}^{\alpha_4}g(x)\left(\abs{\xi}^2\right)_xdx&=&o(1).\label{exp-app20}
\end{eqnarray}
Using integration by parts in Equations \eqref{exp-app19} and \eqref{exp-app20}, we get 
\begin{eqnarray}
\int_{\alpha_3-2\varepsilon-\delta}^{\alpha_4}g'(x)\abs{\eta}^2dx+\Re\left(2i\la\int_{\alpha_3-2\varepsilon-\delta}^{\alpha_4}g\varphi_x\bar{\eta}dx\right)&=&\abs{\eta(\alpha_4)}^2+\abs{\eta(\alpha_3-2\varepsilon-\delta)}^2+o(1),\label{exp-app21}\\
\int_{\alpha_3-2\varepsilon-\delta}^{\alpha_4}g'(x)\abs{\xi}^2dx+\Re\left(2i\la\int_{\alpha_3-2\varepsilon-\delta}^{\alpha_4}g\psi_x\bar{\xi}dx\right)&=&\abs{\xi(\alpha_4)}^2+\abs{\xi(\alpha_3-2\varepsilon-\delta)}^2+o(1).\label{exp-app22}
\end{eqnarray}
Multiplying Equations \eqref{exp-app3} and \eqref{exp-app5} by $2g(x)\bar{\varphi}_x$ and $2g(x)\bar{\psi}_x$ respectively , integrate over $(\alpha_3-2\varepsilon-\delta,\alpha_4)$, using the fact $\|F\|_{\mathcal{H}_a}\to 0$, $\varphi_x$ and $\psi_x$ are uniformly bounded in $L^2(0,L)$ and Lemma \ref{First-Estimation-aux} and taking the real part, we get 
\begin{eqnarray*}
\Re\left(2i\la\int_{\alpha_3-2\varepsilon-\delta}^{\alpha_4}g(x)\eta\bar{\varphi}_xdx\right)-a\int_{\alpha_3-2\varepsilon-\delta}^{\alpha_4}g(x)\left(\abs{\varphi_x}^2\right)_xdx+2\Re\left(c_0\int_{\alpha_3-2\varepsilon-\delta}^{\alpha_4}g(x)\xi\bar{\varphi}_xdx\right)&=&o(1),\\
\Re\left(2i\la\int_{\alpha_3-2\varepsilon-\delta}^{\alpha_4}g(x)\xi\bar{\psi}_xdx\right)-\int_{\alpha_3-2\varepsilon-\delta}^{\alpha_4}g(x)\left(\abs{\psi_x}^2\right)_xdx-2\Re\left(c_0\int_{\alpha_3-2\varepsilon-\delta}^{\alpha_4}g(x)\eta\bar{\psi}_xdx\right)&=&o(1).
\end{eqnarray*}
Using integration by parts in the second terms of the above  Equations, we obtain 
\begin{equation}\label{exp-app23}
\begin{array}{l}
\displaystyle{\Re\left(2i\la\int_{\alpha_3-2\varepsilon-\delta}^{\alpha_4}g(x)\eta\bar{\varphi}_xdx\right)+a\int_{\alpha_3-2\varepsilon-\delta}^{\alpha_4}g'(x)\abs{\varphi_x}^2dx+2\Re\left(c_0\int_{\alpha_3-2\varepsilon-\delta}^{\alpha_4}g(x)\xi\bar{\varphi}_xdx\right)}\nline
\qquad \displaystyle{=a\abs{\varphi_x(\alpha_4)}^2+a\abs{\varphi_x(\alpha_3-2\varepsilon-\delta)}^2+o(1)}
\end{array}
\end{equation}
and 
\begin{equation}\label{exp-app24}
\begin{array}{l}
\displaystyle{\Re\left(2i\int_{\alpha_3-2\varepsilon-\delta}^{\alpha_4}g(x)\xi\bar{\psi}_xdx\right)+\int_{\alpha_3-2\varepsilon-\delta}^{\alpha_4}g'(x)\abs{\psi_x}^2dx-2\Re\left(c_0\int_{\alpha_3-2\varepsilon-\delta}^{\alpha_4}g(x)\eta\bar{\psi}_xdx\right)}\nline
\qquad \displaystyle{=\abs{\psi_x(\alpha_4)}^2+\abs{\psi_x(\alpha_3-2\varepsilon-\delta)}^2+o(1).}
\end{array}
\end{equation}
Adding Equations \eqref{exp-app21}-\eqref{exp-app24}, we get 
\begin{equation}\label{exp-app25}
\begin{split}
M(\alpha_4,\alpha_3-2\varepsilon-\delta)+N(\alpha_4,\alpha_3-2\varepsilon-\delta)=\int_{\alpha_3-2\varepsilon-\delta}^{\alpha_4}g'(x)\left(\abs{\eta}^2+a\abs{\varphi_x}^2+\abs{\xi}^2+\abs{\psi_x}^2\right)dx\\
+2\Re\left(c_0\int_{\alpha_3-2\varepsilon-\delta}^{\alpha_4}g(x)\xi\bar{\varphi}_xdx\right)-2\Re\left(c_0\int_{\alpha_3-2\varepsilon-\delta}^{\alpha_4}g(x)\eta\bar{\psi}_xdx\right)+o(1)
\end{split}
\end{equation}
where 
\begin{eqnarray*}
M(\alpha_4,\alpha_3-2\varepsilon-\delta)&=&\abs{\eta(\alpha_4)}^2+\abs{\eta(\alpha_3-2\varepsilon-\delta)}+a\abs{\varphi_x(\alpha_4)}^2+a\abs{\varphi_x(\alpha_3-2\varepsilon-\delta)}^2,\\
N(\alpha_4,\alpha_3-2\varepsilon-\delta)&=&\abs{\xi(\alpha_4)}^2+\abs{\xi(\alpha_3-2\varepsilon-\delta)}^2+\abs{\psi_x(\alpha_4)}^2+\abs{\psi_x(\alpha_3-2\varepsilon-\delta)}^2.
\end{eqnarray*}
From Equation \eqref{exp-app25}, we get 
\begin{equation*}
\begin{split}
M(\alpha_4,\alpha_3-2\varepsilon-\delta)+N(\alpha_4,\alpha_3-2\varepsilon-\delta)\leq c_{g'}\int_{\alpha_3-2\varepsilon-\delta}^{\alpha_4}\left(\abs{\eta}^2+a\abs{\varphi_x}^2+\abs{\xi}^2+\abs{\psi_x}^2\right)dx\\
+c_0c_g\|\xi\|_{L^2(0,L)}\|\varphi_x\|_{L^2(0,L)}+c_0c_g\|\eta\|_{L^2(0,L)}\|\psi_x\|_{L^2(0,L)}+o(1).
\end{split}
\end{equation*}
Using the fact that $\|\Phi\|$ is uniformly bounded in $\mathcal{H}_a$, we obtain the desired result. The proof of this Lemma has been completed.
\end{proof}
%%%%%%%%%%%%%%%%%%%%%%%%%%%%
%%%%%%%%%%%%%%%%%%%%%%%%%%%%
\begin{lemma}\label{6-Estimation-aux}
Let $0<\delta<\frac{\alpha_3-2\varepsilon-\alpha_2}{2}$. The solution $\left(\varphi,\eta,\psi,\xi\right)\in D(\mathcal{A}_a)$ of Equations \eqref{exp-app2}-\eqref{exp-app5} satisfies the following asymptotic behavior estimation
$$
\int_{\alpha_3-2\varepsilon-\delta}^L\left(\abs{\eta}^2+a\abs{\varphi_x}^2+\abs{\xi}^2+\abs{\psi_x}^2\right)dx=o(1).
$$
\end{lemma}
\begin{proof}
Define the cut-off function $\hat{\theta}$ in $C^1\left([0,L]\right)$ by 
\begin{equation}\label{thetahat}
0\leq \hat{\theta}\leq 1,\quad \hat{\theta}=1\ \ \text{on}\ \ (\alpha_3-2\varepsilon-\delta,L),\quad \text{and}\quad \hat{\theta}=0\ \ \text{on}\ \ (0,\alpha_2+\delta).
\end{equation}
Take $h=(x-L)\hat{\theta}$ in Equation \eqref{exp-app12}, using Lemmas \eqref{First-Estimation-aux}-\eqref{Second-Estimation-aux} and the definition of the function $\hat{\theta}$,  we get 
\begin{equation}\label{exp-app26}
\begin{array}{l}
\displaystyle{\int_{\alpha_3-2\varepsilon-\delta}^L\left(\abs{\eta}^2+a\abs{\varphi_x}^2+\abs{\xi}^2+\abs{\psi_x}^2\right)dx+2c_0\Re\left(\int_{\alpha_3-2\varepsilon-\delta}^{\alpha_4}(x-L)\hat{\theta}\xi\bar{\varphi}_xdx\right)}\nline
\qquad\displaystyle{-2c_0\Re\left(\int_{\alpha_3-2\varepsilon-\delta}^{\alpha_4}(x-L)\hat{\theta}\eta \bar{\psi}_xdx\right)
=o(1).}
\end{array}
\end{equation}
Using the fact that $\xi=i\la\psi-f_3$ and $\eta=i\la\varphi-f_1$ in the second and third  term of Equation \eqref{exp-app26} and that $\varphi_x$, $\psi_x$ are uniformly bounded in $L^2(0,L)$ and the fact that $\|F\|_{\mathcal{H}_a}\to 0$, we get 
\begin{equation*}
\begin{array}{l}
\displaystyle{2c_0\Re\left(\int_{\alpha_3-2\varepsilon-\delta}^{\alpha_4}(x-L)\hat{\theta}\xi\bar{\varphi}_xdx\right) -2c_0\Re\left(\int_{\alpha_3-2\varepsilon-\delta}^{\alpha_4}(x-L)\hat{\theta}\eta \bar{\psi}_xdx\right)=2c_0\Re\left(\int_{\alpha_3-2\varepsilon-\delta}^{\alpha_4}i\la (x-L)\hat{\theta}\psi\bar{\varphi}_xdx\right)}\nline
\qquad \displaystyle{-2c_0\Re\left(\int_{\alpha_3-2\varepsilon-\delta}^{\alpha_4}i\la (x-L)\hat{\theta}\varphi\bar{\psi}_xdx\right)+o(1).}
\end{array}
\end{equation*}
Using integration  by parts in the first term of the right hand side of the above equation and the fact that $\la \varphi$ and $\la \psi$ are uniformly bounded in $L^2(\Omega)$,  we obtain 
\begin{equation}\label{exp-app27}
\begin{array}{l}
\displaystyle{2c_0\Re\left(\int_{\alpha_3-2\varepsilon-\delta}^{\alpha_4}(x-L)\hat{\theta}\xi\bar{\varphi}_xdx\right) -2c_0\Re\left(\int_{\alpha_3-2\varepsilon-\delta}^{\alpha_4}(x-L)\hat{\theta}\eta \bar{\psi}_xdx\right)=}\nline
\qquad\displaystyle{2c_0\Re\left(\left[i\la (x-L)\psi\bar{\varphi}\right]_{\alpha_3-2\varepsilon-\delta}^{\alpha_4}\right)+o(1).}
\end{array}
\end{equation}
Inserting Equation \eqref{exp-app27} in Equation \eqref{exp-app26}, we obtain 
\begin{equation}\label{exp-app28}
\int_{\alpha_3-2\varepsilon-\delta}^L\left(\abs{\eta}^2+a\abs{\varphi_x}^2+\abs{\xi}^2+\abs{\psi_x}^2\right)dx=A(\alpha_4)+B(\alpha_3-2\varepsilon-\delta)+o(1),
\end{equation}
where 
\begin{eqnarray*}
A(\alpha_4)&=&2c_0\Re\left(i\la(L-\alpha_4)\psi(\alpha_4)\bar{\varphi}(\alpha_4)\right),\\
B(\alpha_3-2\varepsilon-\delta)&=&2c_0\Re\left(i\la(\alpha_3-2\varepsilon-\delta-L)\psi(\alpha_3-2\varepsilon-\delta)\bar{\varphi}(\alpha_3-2\varepsilon-\delta)\right).
\end{eqnarray*}
On the other hand, from Equations \eqref{exp-app2} and \eqref{exp-app4}, we have 
\begin{equation}\label{exp-app29}
\abs{\la \varphi(s)}\leq \abs{\eta(s)}+\abs{f_1(s)}\quad \text{and}\quad \abs{\la \psi(s)}\leq \abs{\xi(s)}+\abs{f_3(s)}\ \ \text{for}\ \ s\in \left\{\alpha_3-2\varepsilon-\delta,\alpha_4\right\}.
\end{equation}
Using the fact that $\displaystyle{\abs{ f_{1}(s)}\leq s\int_{0}^{s}\abs{(f_{1})_{x}}^2dx\leq sa^{-1}\|F\|_{\HH_{a}}^{2}}$ and $\displaystyle{\abs{ f_{3}(s)}\leq s\int_{0}^{s}\abs{(f_{3})_{x}}^2dx\leq s\|F\|_{\HH_{a}}^{2}}$ for all $s\in \left\{\alpha_3-2\varepsilon-\delta,\alpha_4\right\}$, and using Lemma \ref{5-Estimation-aux} in Equation \eqref{exp-app29}, we obtain 
\begin{equation*}
\abs{\lambda\varphi(s)}=O(1)\quad \text{and}\quad \abs{\lambda \psi(s)}=O(1),\quad \text{for}\quad  s\in \left\{\alpha_3-2\varepsilon-\delta,\alpha_4\right\}.
\end{equation*}
Its follow that
\begin{equation}
A(\alpha_4)+B(\alpha_3-2\varepsilon-\delta)=o(1).
\end{equation}
Using Equation \eqref{exp-app29} in Equation \eqref{exp-app28}, we obtain 
$$
\int_{\alpha_3-2\varepsilon-\delta}^L\left(\abs{\eta}^2+a\abs{\varphi_x}^2+\abs{\xi}^2+\abs{\psi_x}^2\right)dx=o(1).
$$
Thus, the proof has been completed. 
\end{proof}

%%%%%%%%%%%%%%%%%%%%%%%%%%%%
%%%%%%%%%%%%%%%%%%%%%%%%%%%%
\noindent \textbf{Proof of Theorem \ref{auxiliary-problem-Exp}} Using Lemmas \ref{First-Estimation-aux}, \ref{Second-Estimation-aux}, \ref{Fourth-Estimation-aux} and  \ref{6-Estimation-aux}, we get $\|\Phi\|_{\mathcal{H}_a}=o(1)$ on $[0,L]$, which contradicts Equation \eqref{Phin}. Therefore, \eqref{H4} holds, by  Huang \cite{Huang01} and Pruss \cite{pruss01} we deduce the exponential stability of the auxiliary problem \eqref{aux-prb}. 
\subsection{Definitions and Theorems}
We introduce here the notions of stability that we encounter in this work.

%\noindent %%%%%%%%%%%%%%%%%%%%%%%%%%%%
%%%%%%%%%%%%%%%%%%%%%%%%%%%%
%For this aim, we will use the following result (see Corollary 3.7.18 page 157 in \cite{WArendt}).

\begin{definition}\label{Defsta}
{Assume that $A$ is the generator of a C$_0$-semigroup of contractions $\left(e^{tA}\right)_{t\geq0}$  on a Hilbert space  $\mathcal{H}$. The  $C_0$-semigroup $\left(e^{tA}\right)_{t\geq0}$  is said to be
 %%%%%%%%%%%%%%%%
\begin{enumerate}
\item[1.]  strongly stable if 
$$\lim_{t\to +\infty} \|e^{tA}x_0\|_{H}=0, \quad\forall \ x_0\in H;$$
\item[2.]  exponentially (or uniformly) stable if there exist two positive constants $M$ and $\epsilon$ such that
\begin{equation*}
\|e^{tA}x_0\|_{H} \leq Me^{-\epsilon t}\|x_0\|_{H}, \quad
\forall\  t>0,  \ \forall \ x_0\in {H};
\end{equation*}
\item[3.] polynomially stable if there exists two positive constants $C$ and $\alpha$ such that
\begin{equation*}
 \|e^{tA}x_0\|_{H}\leq C t^{-\alpha}\|x_0\|_{H},  \quad\forall\ 
t>0,  \ \forall \ x_0\in D\left(\mathcal{A}\right).
\end{equation*}
In that case, one says that the semigroup $\left(e^{tA}\right)_{t\geq 0}$ decays  at a rate $t^{-\alpha}$.
\noindent The  $C_0$-semigroup $\left(e^{tA}\right)_{t\geq0}$  is said to be  polynomially stable with optimal decay rate $t^{-\alpha}$ (with $\alpha>0$) if it is polynomially stable with decay rate $t^{-\alpha}$ and, for any $\varepsilon>0$ small enough, the semigroup $\left(e^{tA}\right)_{t\geq0}$  does  not decay at a rate $t^{-(\alpha-\varepsilon)}$.
\end{enumerate}}
%\xqed{$\square$}
\end{definition}
\noindent To show the strong stability of a $C_0-$semigroup of contraction  $(e^{tA})_{t\geq 0}$ we rely on the following result due to Arendt-Batty   \cite{Arendt01}.

\begin{theoreme}\label{arendtbatty}
%\rm{{$\left(\textbf{Arendt and Batty in }\text{\cite{Arendt01}}\right)$}
Assume that $A$ is the generator of a C$_0-$semigroup of contractions $\left(e^{tA}\right)_{t\geq0}$  on a Hilbert space $\mathcal{H}$. If
 %%%%%%%%%%%%%%%%
 \begin{enumerate}
 \item[1.]  $A$ has no pure imaginary eigenvalues,
  \item[2.]  $\sigma\left(A\right)\cap i\mathbb{R}$ is countable,
 \end{enumerate}
where $\sigma\left(A\right)$ denotes the spectrum of $A$, then the $C_0-$semigroup $\left(e^{tA}\right)_{t\geq0}$  is strongly stable.%\xqed{$\square$}
\end{theoreme}

\noindent  Concerning the characterization of exponential stability of a $C_0-$semigroup of contraction  $(e^{tA})_{t\geq 0}$ we rely on the following result due to Huang \cite{Huang01} and Pr$^¨u$ss \cite{pruss01}. 
\begin{theoreme}\label{hp}
Let $A:\ D(A)\subset H\rightarrow H$ generate a $C_0-$semigroup of contractions $\left(e^{tA}\right)_{t\geq 0}$ on $H$. Assume that $i\la \in \rho(A)$, $\forall \la \in \R$. Then, the $C_0-$semigroup $\left(e^{tA}\right)_{t\geq 0}$ is exponentially stable if and only if 
$$
\varlimsup_{\la\in\R,\ \abs{\la}\to +\infty}\|(i\la I-A)^{-1}\|_{\mathcal{L}(H)}<+\infty.
$$
\end{theoreme}
\noindent  Also, concerning the characterization of polynomial stability of a $C_0-$semigroup of contraction  $(e^{tA})_{t\geq 0}$ we rely on the following result due to Borichev and Tomilov \cite{Borichev01} (see also \cite{RaoLiu01} and \cite{Batty01}). 

\begin{theoreme}\label{bt}
Assume that $A$ is the generator of a strongly continuous semigroup of contractions $\left(e^{tA}\right)_{t\geq0}$  on $H$.   If   $ i\mathbb{R}\subset \rho(A)$, then for a fixed $\ell>0$ the following conditions are equivalent
%%%%%%%%%%%%%%%%
\begin{equation}\label{h1}
\sup_{\lambda\in\mathbb{R}}\left\|\left(i\lambda I-A\right)^{-1}\right\|_{\mathcal{L}\left(H\right)}=O\left(|\lambda|^\ell\right),
\end{equation}
\begin{equation}\label{h2}
\|e^{tA}U_{0}\|^2_{H} \leq \frac{C}{t^{\frac{2}{\ell}}}\|U_0\|^2_{D(A)},\hspace{0.1cm}\forall t>0,\hspace{0.1cm} U_0\in D(A),\hspace{0.1cm} \text{for some}\hspace{0.1cm} C>0.
\end{equation}
\end{theoreme}
\noindent Finally, the analytic property of a $C_0-$semigroup of contraction  $(e^{tA})_{t\geq 0}$ is characterized in the following theorem due to Arendt, Batty and Hieber \cite{arendtbatty02}.

\begin{theoreme} \label{analytic}
Let $(\mathcal{S}(t)= e^{tA})_{t \geq 0 }$ be a $C_{0}-$semigroup of contractions in a Hilbert space. Assume that  
\begin{equation}\tag{{\rm {A1}}}
i\R\subset \rho (A).
\end{equation}
 Then, $(e^{tA})_{t \geq 0}$ is analytic if and only if 
\begin{equation}\tag{{\rm {A2}}}
\displaystyle{\limsup_{\lambda \in \mathbb{R}, |\lambda| \rightarrow \infty}}\;\frac{1}{|\lambda|^{-1}} \norm{(i\lambda-A)^{-1}}_{\mathcal{L}(H)}< \infty.  \end{equation}
\end{theoreme}

\section{Conclusion}
We have studied the stabilization of a system of locally coupled wave equations with  only one  internal localized  Kelvin-Voigt damping via  non-smooth coefficients. We proved the strong stability of the system using Arendt-Batty criteria. Lack of exponential stability results has been proved in both cases: The case of global Kelvin-Voigt damping and the case of  localized Kelvin-Voigt damping, taking into consideration that the coupling is global. In addition, if both coupling and damping are localized internally via non-smooth coefficients,  we established a polynomial energy decay rate of type $t^{-1}$. We can conjecture that the energy decay rate $t^{-1}$ is optimal. However, if the intersection between the supports of the domains of the damping and the coupling coefficients is empty, the nature of the decay rate of the system will be unknown. This question is still an open problem.
%%%%%%%%%%%%%%%%%%%%%%%%%%%%%%%%%%%%%%%%%%%%%%%%%%%%%%%%%%%%%%%%%%%%%%

\section*{Acknowledgments}  

\noindent The authors thanks professors Michel Mehrenberger and Kais Ammari for their valuable discussions and comments.\\

\noindent Mohammad Akil would like to thank the Lebanese University for its support. \\

\noindent Ibtissam Issa would like to thank the Lebanese University for its support. \\

\noindent Ali Wehbe would like to thank the CNRS and the LAMA laboratory of Mathematics of the Universit\'e Savoie Mont Blanc for their supports.

%\protect\bibliographystyle{abbrv}
%\protect\bibliographystyle{alpha}
%\bibliography{References}

\end{document}